\documentclass[a4paper]{article}
\usepackage{hyperref}
\usepackage[english]{babel}
\usepackage[utf8]{inputenc}
\usepackage{bm}
\usepackage{amsmath}
\usepackage{amsfonts}
\usepackage{amsthm}
\usepackage{chngcntr}
\usepackage{apptools}
\usepackage{enumitem}
\AtAppendix{\counterwithin{lem}{section}}
\AtAppendix{\counterwithin{cor}{section}}
\usepackage{graphicx}
\usepackage{subcaption}
\usepackage[square,sort,comma,numbers]{natbib}
\usepackage{color}
\usepackage{textcomp}
\begin{document}
\newtheorem{cex}{Counter Example}
\newtheorem{conj}{Conjecture}
\newtheorem{cor}{Corollary}
\newtheorem{lem}{Lemma}
\newtheorem{rem}{Remark}
\newtheorem{defn}{Definition}
\newtheorem{model}{Random Graph Model}
\newtheorem{prob}{Problem}
\newtheorem{thm}{Theorem}
\newtheorem{examp}{Example}
\newtheorem{duplicate}{Appendix Theorem}
\newcommand*{\LargerCdot}{\raisebox{-0.25ex}{\scalebox{1.5}{$\cdot$}}}
\newenvironment{psmallmatrix}
  {\left(\begin{smallmatrix}}
  {\end{smallmatrix}\right)}

\title{Asymptotics of the spectral radius for directed Chung-Lu random graphs with community structure}
\author{David Burstein}
\maketitle
\begin{abstract}
The spectral radius of the adjacency matrix can impact both algorithmic efficiency as well as the stability of solutions to an underlying dynamical process.  Although much research has considered the distribution of the spectral radius for undirected random graph models, as symmetric adjacency matrices are amenable to spectral analysis, very little work has focused on directed graphs.  Consequently, we provide novel concentration results for the spectral radius of the directed Chung-Lu random graph model.  We emphasize that our concentration results are applicable both asymptotically and to networks of finite size.  Subsequently, we extend our concentration results to a generalization of the directed Chung-Lu model that allows for community structure.  
\end{abstract}


\section{Introduction}
In an effort to understand how the architecture of real world networks impacts the underlying dynamics, we want to prove results pertaining to a broader class of plausible networks.  Random graph models help us address this challenge.  Realizations from an appropriately chosen random graph model can emulate many of the properties found in real world networks.  Consequently, using these random graph models, we would like to identify families of graphs that behave similarly under a dynamical process.  To this end, prior work has demonstrated how the spectral radius of the adjacency matrix can promote various dynamical behaviors on genetic, epidemiological and biological neural networks \cite{restrepo2006emergence,pomerance2009effect,zhao2011synchronization,pastor2015epidemic,
ferreira2012epidemic,pastor2002epidemic,ganesh2005effect}.  And eventhough the existing literature has primarily focused on the distribution of the spectral radius for undirected random graph models \cite{castellano2010thresholds,chung2003eigenvalues,chung2011spectra,le2015concentration,lu2012spectra,zhao2012consistency, zhang2014spectra, tran2013sparse,oliveira2009concentration}, many real world networks are indeed directed graphs \cite{malliaros2013clustering}.  

To address this gap, we provide concentration inequalities and asymptotics for the distribution of the spectral radius for the {\em directed} Chung-Lu random graph model, where realizations possess degree heterogeneity much like real world networks \cite{chakrabarti2006graph}.  We then demonstrate how to extend our proof technique to a more general graph model that enables community structure within the graph.  Subsequently, we conclude this work by providing simulations of an epidemiolgical SIS process, where we illustrate how the spectral radius (and community structure) can influence the stability of the healthy state.

Our interest in the spectral radius of random directed graphs extends beyond modeling dynamical processes on networks.  For example, we often measure the performance of an algorithm, not based on the worst case performance, but instead by evaluating the algorithm's efficiency on realizations from random graph models that emulate real world networks \cite{bader2008snap,edmonds2010space,vineet2009fast}.  As such, the Chung-Lu random graph model, among others, has been proposed as a benchmark to test algorithmic efficiency \cite{pinar2011similarity,chakrabarti2004r,leskovec2007scalable,leskovec2010kronecker}.    In a prior work \cite{burstein2016k}, we exploited the distribution of the spectral radius for undirected Chung-Lu random graphs to justify the efficiency for an almost shortest path algorithm.   Analogously, knowledge about the distribution of the spectral radius for directed Chung-Lu random graphs would allow us to evaluate the algorithmic efficiency for real world directed networks as well.    

In addition, spectral analysis for directed networks could also lead to new techniques for community detection on directed graphs.   Succinctly, many community detection algorithms classify nodes into hidden communities by extrapolating community membership using statistical anomalies from a null model, such as the Chung-Lu random graph model.    In context to our work, the techniques employed could be helpful in analyzing the spectral radius for asymmetric adjacency and laplacian matrices commonly encountered in the community detection problem for directed graphs \cite{fortunato2010community,malliaros2013clustering}.  Furthermore, our results illustrating how the spectral radius varies with community structure could be helpful for identifying communities in the network.     All of the aforementioned applications underscore the importance of developing concentration bounds for the spectral radius on directed graphs.  

At this juncture, we provide an outline of our proof strategy to bound the spectral radius for Chung-Lu random graphs, starting with a definition for the directed Chung-Lu Model.  

\begin{defn}[Directed Chung-Lu Model \cite{chung2002connected}]
Let $N$ be the number of nodes in a directed graph and let $\mathbf{d}=(\mathbf{a},\mathbf{b})\in\mathbb{Z}^{N\times 2}$ be the expected degree sequence, where $a_{i}$ ($b_{i}$) corresponds to the expected in (out) degree of node $i$. Denote $S = \sum_{a_i\in\mathbf{a}} a_i= \sum_{b_i\in\mathbf{b}} b_i$ and suppose that $\max_{i,j} a_ib_j \leq S$.  We then model edges in the graph as independent Bernoulli random variables, where we denote the probability a directed edge from $i$ to $j$ exists as $p_{ij}$; in particular, $p_{ij}=\frac{b_{i}a_{j}}{S}$.  
\end{defn}

As mentioned earlier, we want to identify the distribution of the spectral radius for these random Chung-Lu graphs.  To this end, Restrepo, Ott and Hunt \cite{restrepo2007approximating} provide a novel technique for deriving asymptotics for the spectral radius, but do not conclusively address the validity for using such an approximation.    Consequently, we seek rigorous results illustrating when the following conjecture holds.
\begin{conj}\label{conj:1}
Consider a sequence of realizations from the Directed Chung-Lu model, where the expected number of edges, $S$ becomes arbitrarily large in the limit. Denote the absolute value of the dominating eigenvalue of $A$ by $\rho(A)$.
Then almost surely, $$\lim_{N \rightarrow \infty}\frac{\rho(A)}{\frac{\mathbf{a}\cdot\mathbf{b}}{S}}=1$$
\end{conj}

Unfortunately when constructing bounds on the dominating maximum eigenvalue of an asymmetric (directed) matrix, we do not have quite as many tools to construct concentration results compared to the symmetric (undirected) case.  Our main theoretical tool will be the following lemma, but first we introduce the following definition for clarity.

\begin{defn}
We define the spectral radius of a matrix $A\in\mathbb{R}^{N\times N}$ to be the absolute value of the maximal magnitude eigenvalue.  More precisely, $$\rho(A) = max\{|\lambda|: Ax = \lambda x \hspace{3pt} and \hspace{3pt} x \neq \mathbf{0}\}$$
\end{defn}

We now present our Lemma that will enable us to compute the spectral radius corresponding to a directed graph.

\begin{lem}\label{lem:theory} 
Let $A\in\mathbb{R}^{N\times N}$ be an entrywise non-negative matrix.    For simplicity, assume that the maximal magnitude eigenvalue of $A$ is real. 
\newline\newline
 Define $\lambda_{max}$ as the maximum of the eigenvalues of $A$.  It then follows that  $\rho(A) = \lambda_{max}$.
\newline\newline
Furthermore, we have for every positive integer $r$,

\begin{equation}\label{eq:lemma} \frac{trace(A^{r})}{N}\leq \rho(A^{r})=\rho(A)^{r}\leq \mathbf{1}^{T}A^{r}\mathbf{1} 
\end{equation}\newline where $\mathbf{1}$ is the vector of one's.  
\end{lem}

We are primarily interested in applications of Lemma \ref{lem:theory} in studying the spectral radius of an adjacency matrix for random graphs.  In this case, if $A$ corresponds to a directed graph, the quantities that bound $\rho(A)$, $\mathbf{1}^{T}A^{r}\mathbf{1}$ and $trace(A^{r})/N$ both have combinatorial interpretations \cite{bapat2010graphs}.  More specifically,  $trace(A^{r})=\sum_{j=1}^{N}\mathbf{e}_j^{T}A^{r}\mathbf{e}_j$ (where $\mathbf{e}_j$ is the standard unit vector) is the number of cycles of length $r$.  Similarly,$\mathbf{1}^{T}A^{r}\mathbf{1}$ is the number of paths of length $r$.  

To see this consider $\mathbf{e}_{j}^{T}A\mathbf{1}$.  Now, $ \mathbf{e}_j^{T}A$ is a vector whose kth entry is $1$ if there is an edge from node $j$ to node $k$.  Consequently, $\mathbf{e}_{j}^{T}A\mathbf{1}$ is the number of paths starting at node $j$ of length $1$ (and analogously $\mathbf{1}^{T}A\mathbf{1}$ is the number of paths of length $1$). One can then proceed inductively, to show that  $\mathbf{e}_{j}^{T}A^{2}\mathbf{1}$ is the number of paths starting at node $j$ of length $2$ and so on.

The basic idea behind our proofs is as follows.  We compute bounds on the expected value and variance for the number cycles/paths of length $r$.  Then we can use concentration inequalities, like Markov's Inequality or Chebyshev's Inequality, to show that with high probability the number of cycles  of length $r$ must be close to the quantity, $C (\frac{\mathbf{a}\cdot\mathbf{b}}{S})^{r}$.  Furthermore, $C$ satisfies the constraint that $\frac{1}{N^{3}}\leq C\leq 1$.   By Lemma \ref{lem:theory}, it follows that  $C (\frac{\mathbf{a}\cdot\mathbf{b}}{S})^{r}\leq \rho(A)^{r}$ and hence $C^{\frac{1}{r}} (\frac{\mathbf{a}\cdot\mathbf{b}}{S})\leq \rho(A)$. Consequently,  if we choose $r = O(log(N)^{1+\delta})$, where $\delta > 0$, we are guaranteed that the $$\limsup_{N\rightarrow \infty}\frac{C^{\frac{1}{r}} (\frac{\mathbf{a}\cdot\mathbf{b}}{S})}{\rho(A)}= \limsup_{N\rightarrow \infty}\frac{(\frac{\mathbf{a}\cdot\mathbf{b}}{S})}{\rho(A)}\leq 1.$$  We then repeat an analogous argument for the number of paths of length $r$ to bound $\rho(A)$ from above.  The following list highlights the main contributions in this work.

\begin{itemize}
\item In Section 2, we illustrate how path counting can yield concentration inequalities and asymptotics on the spectral radius for the Chung-Lu model  when  $\frac{\mathbf{a}\cdot\mathbf{b}}{S}\rightarrow \infty$. In particular, to compute bounds on the moments for the number of paths and cycles, we prove Lemma \ref{lem:probpath}, which enables us to efficiently compute the probability a path exists even if we revisit edges in the path multiple times.  
\item Subsequently in Section 3, we  extend our results in the case where 
$p_{max} = max_{i,j}\frac{a_{i}b_{j}}{S} \rightarrow 0$.  To derive the desired asymptotic result, we prove that for such sparse graphs, there are restrictions on how we can revisit edges in a given path; in particular, from Lemma \ref{lem:cyclelength} two cycles of {\em modest} length cannot be close together in distance.  
\item In Section 4 we introduce the Partitioned Chung-Lu model, which allows for community structure in the network.  While introducing communities results in a more flexible model, requiring that the probability that two nodes share an edge depends on the community membership of the two nodes also makes bookeeping (and in turn the analysis) more challenging.  To address this issue, we illustrate in  Lemma \ref{lem:generalized}  how to express the sum of the number of paths as a norm of a matrix vector product and provide extensions of our results in Sections 2 and 3.
\item We then conclude this work with Section 5, where we consider an application of our spectral radius results by simulating an susceptible-infected-susceptible process, illustrating how community structure (and the spectral radius) can influence the stability of the healthy state in our network.
\end{itemize}

\section{Spectral Concentration Bounds, $\frac{\mathbf{\lowercase{a}}\cdot \mathbf{\lowercase{b}}}{S}\rightarrow \infty$}
With Lemma \ref{lem:theory} in mind, we initiate our discussion on bounding the spectral radius for the Chung-Lu model by bounding the expected number of paths of length $r$.  We will consider two cases seperately, where either $\frac{\mathbf{a}\cdot\mathbf{b}}{S} \rightarrow \infty$ (this section) or $p_{max} = \max_{i,j} p_{ij} \rightarrow 0$ (next section).   Though the results below hold in considerable generality, many of the initially stated results will only be asymptotically useful for the case where  $\frac{\mathbf{a}\cdot\mathbf{b}}{S} \rightarrow \infty$.  We start by identifying a lower bound on the expectation of $\mathbf{e}_{j}^{T}A^{r}\mathbf{e}_{i}$ as our first step for constructing spectral radius concentration results.

\begin{lem}
\label{lem:lower}
Consider a realization of the Directed Chung-Lu random graph model with expected degree sequence $\mathbf{d} = (\mathbf{a},\mathbf{b})$, where $\sum a_{i} = \sum b_{i} = S$.  Then the expected number of paths from node $y$ of length $r$ is bounded below by $$ b_{y}[\frac{\mathbf{a}\cdot\mathbf{b}}{S}]^{r-1}$$ Furthermore, the expected number of paths of length $r$ is bounded below by $$S[\frac{\mathbf{a}\cdot\mathbf{b}}{S}]^{r-1}$$
\end{lem}
\begin{proof}
First, we define the indicator random variable $\mathbf{1}_{(y,i_{1},...,i_{r})}$ to equal $1$ if the path $(y,i_{1},...,i_{r})$ exists and $0$ otherwise.  Then it follows that  the expected number of paths from node y of length r is $$\sum_{i_{1}=1,...,i_{r}=1}^{N}E(\mathbf{1}_{(y,i_{1},...,i_{r})}),$$
where we sum over all possible choices of nodes; for example, $i_{1}$, the second node in the path, can equal any of the $N$ nodes in the graph, etc.

Since the probability that a collection of distinct edges exist are independent events, if no edge repeats in the path $(y,i_{1},...,i_{r})$, then $$E(\mathbf{1}_{(y,i_{1},...,i_{r})})=p_{yi_{1}}\Pi_{k=1}^{r-1}p_{i_{k}i_{k+1}}.$$  

If however an edge does repeat in the path, since each $0\leq p_{ij}\leq 1$ for all $i,j$, it follows that for such a path $$E(\mathbf{1}_{(y,i_{1},...,i_{r})})\geq p_{yi_{1}}\Pi_{k=1}^{r-1}p_{i_{k}i_{k+1}}.$$

Consequently, we conclude that the $$\sum_{i_{1}=1,...,i_{r}=1}^{N}E(\mathbf{1}_{(y,i_{1},...,i_{r})})\geq \sum_{i_{1}=1,...,i_{r}=1}^{N}p_{yi_{1}}\Pi_{k=1}^{r-1}p_{i_{k}i_{k+1}}.$$

Let $a_{x},b_{x}$ denote the expected in-degree/out-degree of node $x$.  
Now applying the definition for Chung-Lu random graphs, we simplify $$\sum_{i_{1}=1,...,i_{r}=1}^{N}p_{yi_{1}}\Pi_{k=1}^{r-1}p_{i_{k}i_{k+1}}=\sum_{i_{1}=1,...,i_{r}=1}^{N}\frac{b_{y}a_{i_{1}}}{S}\Pi_{k=1}^{r-1}\frac{b_{i_{k}}a_{i_{k+1}}}{S}=\sum_{i_{1}=1,...,i_{r}=1}^{N}\frac{b_{y}a_{i_{r}}}{S}\Pi_{k=1}^{r-1}\frac{b_{i_{k}}a_{i_{k}}}{S},$$

where the last equality follows from rearranging the terms.  We then conclude that
\begin{equation} \label{eq:basiclower}
\sum_{i_{1}=1,...,i_{r}=1}^{N}E(\mathbf{1}_{(y,i_{1},...,i_{r})})\geq b_{y}(\frac{\mathbf{a}\cdot\mathbf{b}}{S})^{r-1}.
\end{equation}

It then follows that we can construct a lower bound for the expected {\em total} number of paths of length $r$ by using (\ref{eq:basiclower}), where we invoke the lower bound for the expected number of paths of length $r$ starting from node $y$ and sum over all possible initial node choices.  Consequently, $\sum_{i} b_{i}(\frac{\mathbf{a}\cdot\mathbf{b}}{S})^{r-1} = S(\frac{\mathbf{a}\cdot\mathbf{b}}{S})^{r-1}$, the proof is complete.
\end{proof}

Since computing the probability that a path exists is quite challenging when a path revisits the same edge more than once, we introduce the following definitions to help us address this problem.

\begin{defn}
	Consider an edge in a given path.  If the edge has not appeared before, that edge is a \textbf{new edge}.  Alternatively, if the edge has appeared before, that edge is a \textbf{repeating edge}.  Furthermore, a list of consecutive repeating edges of maximal size in a path is called a \textbf{repeating edge block}.  In addition, the \textbf{length} of a repeating edge block is the number of edges that appear in the edge block.  Define the \textbf{new edge interior} to be a list of nodes that includes the $mth$ node in the path if the incoming edge to the $mth$ node and the outgoing edge from the $mth$ node are both new edges. 
\end{defn}  

\noindent \textbf{Remark:} The new edge interior can contain multiple copies of the same node.  Equivalently, in a path of length $r$, where $2 \leq m \leq r$, we say that the $mth$ node in the path belongs to the new edge interior if the $m-1$st and $mth$ edges are new edges. 

\noindent \textbf{Example 1:}  Consider the path $(1,2,3,4,1,2,3,7,8)$, which consists of $8$ edges.  The first four edges and the last two edges in the path are new edges.  The fifth and sixth edges in the path $\{(1,2),(2,3)\}$ are repeating edges.  Since the repeating edges $(1,2)$ and $(2,3)$ appear consecutively in the path (and there are no other repeating edges that appear next to these edges), they form a repeating edge block.  Additionally, $\{2,3,4\}$ consists of the new edge interior, as the second, third and fourth nodes in the path are part of new incoming and new outgoing edges.

Before we can introduce the following lemma, we will need a bit of notation.  Consider a function $f$ that maps elements in a list to the real numbers.   As a list $\mathbf{N}$ can contain multiple copies of the same node, we define $\Pi_{i\in\mathbf{N}} f(i) = \Pi_{i=1}^{|\mathbf{N}|}f(X_{i})$, where $|\mathbf{N}|$ is the number of entries in the list and $X_{i}$ is the $ith$ entry in the list.  In particular if $\mathbf{N} = \{1,2,1\}$, then $\Pi_{i\in\mathbf{N}} f(i) = f(1)\cdot f(2) \cdot f(1)$.  We now provide the desired result, which will help us compute the probability that a path exists.

\begin{lem}\label{lem:probpath}
	Define $\mathbf{1}_{(i,j)}$ as an indicator random variable that equals $1$ if the edge $(i,j)$ exists and $0$ otherwise.  Consequently, $\Pi_{k=1}^{r}\mathbf{1}_{(i_{k},i_{k+1})} = \mathbf{1}_{(i_{1},...,i_{r+1})}$ is an indicator random variable that equals $1$ if there is a path $(i_{1},...,i_{r+1})$.  Let $\mathbf{N}$ be the new edge interior and let $\mathbf{R}$ be a list of pairs of the first and last nodes  for each repeating edge block. If the first and last edges are new edges, then
	
	\begin{equation}\label{eq:prob}
		Pr(\Pi_{k=1}^{r}\mathbf{1}_{(i_{k},i_{k+1})} = 1) = \frac{b_{i_{1}}a_{i_{r+1}}}{S}\Pi_{i\in\mathbf{N}}\frac{a_{i}b_{i}}{S}\Pi_{(j,k)\in \mathbf{R}}\frac{a_{j}b_{k}}{S}.
	\end{equation}
	
	Furthermore, if $k_{i}$ is the number of repeating edge blocks of length $i$, then the number of nodes in $\mathbf{N}$, \begin{equation}\label{eq:count} |\mathbf{N}| = r - 1 - \sum_{i=1}^{r-2}(i+1)k_{i}. 
	\end{equation}
\end{lem}
\begin{proof}

\begin{figure}
\centering
\includegraphics[scale=.4]{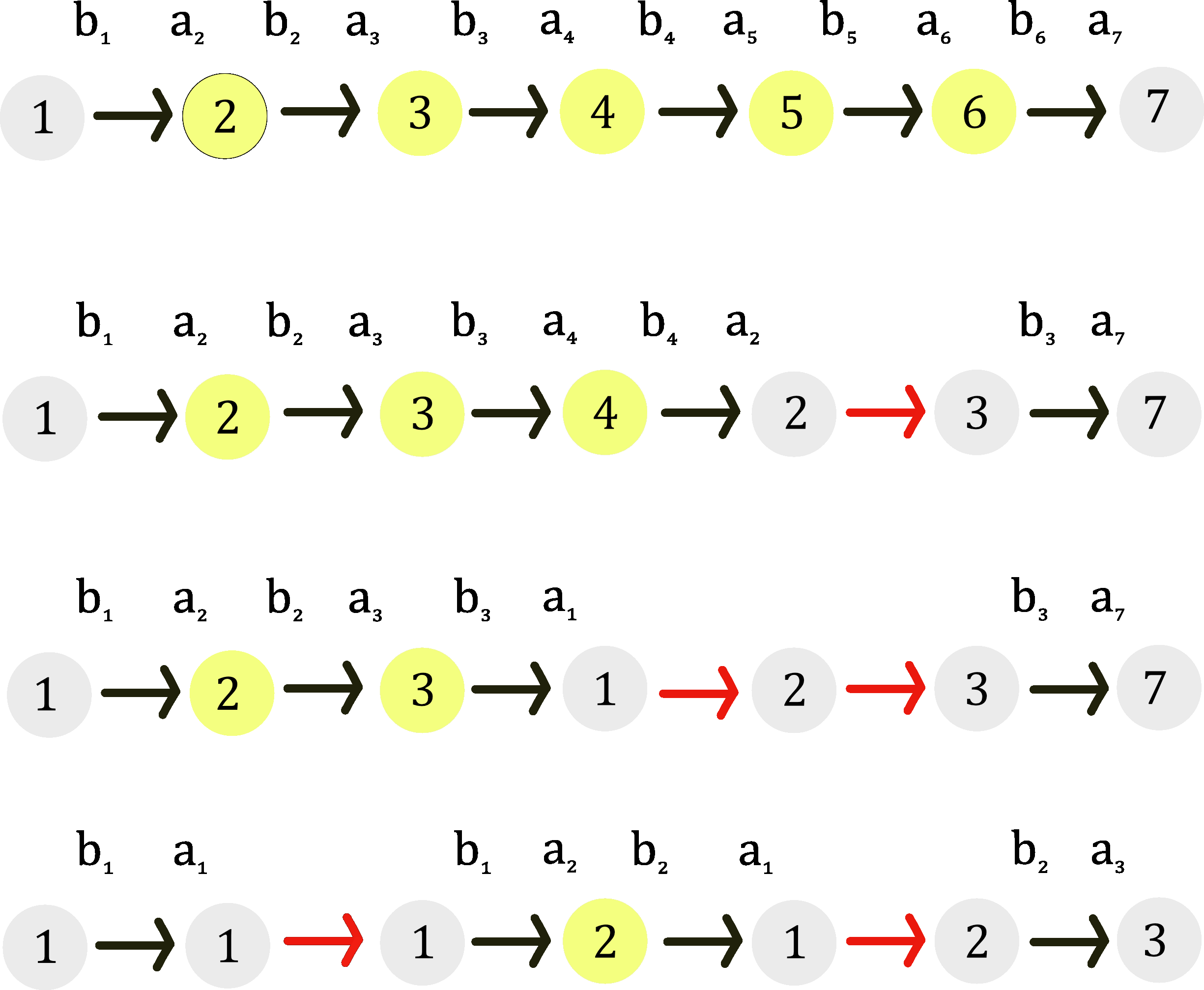}
\caption[Computing the probability that a path exists]{Pictured are four paths.  Yellow nodes are in the new edge interior, while red edges indicate that the edge is a repeating edge.  For each new edge the corresponding nodes' expected in-degree and out-degree are listed.  The product of all of those terms in a path is proportional to the probability that a path exists.}
\label{fig:probpath}
\end{figure}  
To develop the intuition behind (\ref{eq:prob}), we first consider some examples listed in Figure \ref{fig:probpath}.  	Consider the first path, where except for the first and last nodes in the path, all of the nodes in the path are in the new edge interior (highlighted in yellow).  Note that the nodes in the new edge interior supply an incoming edge and an outgoing edge.  In particular $\mathbf{N} = \{2,3,4,5,6\}$ and the probability that the first path exists is precisely,
$$\frac{b_{1}a_{7}}{S}\Pi_{i=2}^{6}\frac{a_{i}b_{i}}{S}=\frac{b_{1}a_{7}}{S}\Pi_{i\in\mathbf{N}}\frac{a_{i}b_{i}}{S}.$$

Of course there may be repeating edges, as listed in Figure \ref{fig:probpath}, which we highlight in red.  In particular if we consider the probability the second path exists in the figure, we identify the nodes in the new edge interior $\mathbf{N} = \{2,3,4\}$ and then express the pairs of the first and last node in each repeating edge block $\mathbf{R} = \{(2,3)\}$.  It then follows that the probability this path exists is,
$$\frac{b_{1}a_{7}}{S}\frac{a_{2}b_{3}}{S}\Pi_{i=2}^{4}\frac{a_{i}b_{i}}{S}=\frac{b_{1}a_{7}}{S}\Pi_{(j,k)\in \mathbf{R}}\frac{a_{j}b_{k}}{S}\Pi_{i\in\mathbf{N}}\frac{a_{i}b_{i}}{S}.$$

As illustrated in the third and fourth paths, we can have many repeating edge blocks of various lengths.  Now to derive the general formula, the basic idea is that if a node appears in the new edge interior, then that node corresponds to both a new incoming edge and a new outgoing edge.  If however a node is not in the new edge interior, then it is either in a repeating edge block or is the first or last node in the path.  The first and last nodes in the path only correspond to one new outcoming edge or one new incoming edge respectively, so we only express their expected out-degree (in-degree) once in the product.  Similarly, nodes at the start or end of a repeating edge block only correspond to one new incoming edge or one new outgoing edge.  Hence we have that the probability that the path exists is,

	\begin{equation}
		\frac{b_{i_{1}}a_{i_{r+1}}}{S}\Pi_{i\in\mathbf{N}}\frac{a_{i}b_{i}}{S}\Pi_{(j,k)\in \mathbf{R}}\frac{a_{j}b_{k}}{S}
	\end{equation}

	To verify $$|\mathbf{N}| = r - 1 - \sum_{i=1}^{r-2}(i+1)k_{i},$$
	we consider the list of the nodes that are not in the new edge interior, $|\mathbf{N}^{c}|$.  Alternatively, $|\mathbf{N}^{c}|$ counts the number of times a node appears in a repeating edge block in addition to the first and last nodes of the path.  This quantity is precisely $2 + \sum_{i}(i+1)k_{i}$, where $k_{i}$ is the number of repeating edge blocks of length $i$ and $i+1$ are the number of nodes in a repeating edge block of length $i$.  Consequently, since there are $r + 1$ nodes in a path of length $r$, $r + 1 - 2 - \sum_{i\geq 1}(i+1)k_{i}$ is precisely the right hand side of equation (\ref{eq:count}).  (We derive the upperlimit in the summation from the assumption that the first and last edges are new edges and that the path has length $r$, which implies that a maximal of length of a repeating edge block could be at most $r-2$.) 
\end{proof}

 To calculate the number of paths in the graph, we employ Hoare-Ramshaw notation for a closed set of integers, namely $$[a..b] = \{x\in \mathbb{Z}: a\leq x \leq b\}.$$ We now provide an upper bound for the expected number of paths of length r from a given node y.  

\begin{thm}
\label{thm:upper}   
Denote $p_{max} = \max_{i,j} p_{ij}$.
Assume $\frac{S}{\mathbf{a}\cdot \mathbf{b}}<\frac{1}{2}$ and that $r < \frac{\mathbf{a}\cdot\mathbf{b}}{S} $, then we have the following upperbound for the expected number of paths from any node $y$ of length $r$,

$$\sum_{i_{1},...,i_{r}}Pr[\mathbf{1}_{(y,i_{1},...,i_{r})}=1]\leq 2b_{y}(\frac{\mathbf{a}\cdot\mathbf{b}}{S})^{r-1}exp(p_{max}\frac{r^{2}(\frac{S}{\mathbf{a}\cdot\mathbf{b}})^{2}}{1-r\frac{S}{\mathbf{a}\cdot\mathbf{b}}}).$$ 
\end{thm}
\begin{proof}
As demonstrated in Lemma \ref{lem:lower} when evaluating the likelihood that a particular path exists, it is helpful to identify the edges that repeat multiple times throughout the path.   We want a method that identifies 'repeating' edges that will help simplify our calculations.  
\newline\newline
 Informally, a repeating edge is an edge that has been {\em observed} before.  So consider the probability that the path $(i_0,i_{1},...,i_{r})$ exists, that is,
 
 \begin{multline} Pr[\mathbf{1}_{(i_0,i_{1},...,i_{r})}=1] = Pr(\mathbf{1}_{(i_0,i_{1})}=1)Pr(\mathbf{1}_{(i_{r-1},i_{r})}=1|\mathbf{1}_{(y,i_{1})}=1)\LargerCdot \\ \Pi_{j=1}^{r-2}Pr(\mathbf{1}_{(i_j,i_{j+1})}=1|\Pi_{k=0}^{j-1} \mathbf{1}_{(i_k,i_{k+1})}\mathbf{1}_{(i_{r-1},i_{r})}=1 ),
 \end{multline}
 
 where we observe the first edge in the path first, then we observe the { \em last} edge in the path.  After that, we observe the second edge in the path, followed by the third edge in the path, etc.   

Now we condition on the possibility that the last edge could be a repeating edge.  If the last edge is indeed a repeating edge, then we also know that the last edge has to equal the first edge.  Furthermore, we know that the first edge starts at node $y$, similarly the last edge must also start at node $y$. We can rewrite  this path of length $r$ as  a cycle starting and ending at node $y$ (of length $r-1$) where we append this repeated last edge to the cycle.  Since this last edge must be identical to the first edge, we conclude that the expected number of paths of length $r$ starting at node $y$, where the last edge  is determined by the first edge in the cycle, equals the expected number of cycles of length $r-1$ starting and ending at node $y$. 
\newline\newline
Now define $P_{r}(y)$ to be the number of paths of length $r$ starting at node $y$ and $C_{r}(y)$ to be the number of cycles of length $r$ starting at node $y$. Furthermore, denote $P_{r}^{L}(y)$ to be the number of paths of length $r$ where the last edge is a new edge. We have the following decomposition.

\begin{equation}\label{eq:recurrence} 
E(P_{r}(y)) = E(P_{r}^{L}(y))+E(C_{r-1}(y)) \leq E(P_{r}^{L}(y)) + E(P_{r-1}(y)).
\end{equation}

Applying  inequality (\ref{eq:recurrence}) over again, we have that $E(P_{r}(y)) \leq E(P_{r}^{L}(y)) + E(P_{r-1}^{L}(y))  + E(P_{r-2}(y))$.
Repeating this trick inductively will yield that

\begin{equation}  \label{eq:recurrencegoal}
E(P_{r}(y)) \leq \sum_{m=2}^{r}E(P_{m}^{L}(y)) + E(P_{1}(y)) = \sum_{m=1}^{r}E(P_{m}^{L}(y)),
\end{equation}

where the last equality follows from the fact that a path of length $1$ can never have a repeating edge.  

By inequality (\ref{eq:recurrencegoal}), to construct a meaningful bound for $E(P_{r}(y))$, it will suffice to consider a bound for $E(P_{r}^{L}(y))$, \textbf{where we are only interested in cases where the last edge (and the first edge) cannot be a repeating edge.}  

Since the first and last edges cannot be repeating edges, we can now invoke Lemma \ref{lem:probpath} to compute the probability that a given path exists.

Define $k_{0}$ to be the number of new edges.  For $i\geq 1$, let $k_{i}$ be the number of repeating edge blocks of length $i$.  So to compute $E[P_{r}^{L}(y)]$, we will fix (integer) values for $k_{i}$, consider all possible arrangements for each of the $k_{i}$, repeating edge blocks and then consider all possible choices of last node in the path $z$, choices for the  lists $\mathbf{N}$, consisting of nodes in the new edge interior, and $\mathbf{R}_{*}$, consisting of nodes in the repeating edge blocks with their corresponding positions in each repeating edge block. (Note that $\mathbf{R}_{*}$ uniquely determines $\mathbf{R}$, the list of pairs of nodes at the beginning and end of a repeating edge block.) It then follows that we have the following upperbound,

\begin{multline} \label{eq:E1}
E[P_{r}^{L}(y)] \leq \\ \sum_{\substack{k_{0}+\sum_{i=1}^{r-2}i(k_{i})=r \\ \forall i\in[0..r-2], k_{i} \in [0..r]}}\binom{\sum_{i=0}^{r-2}k_{i}}{k_{0},k_{1},...,k_{r-2}} [\sum_{z=1}^{N}\sum_{\mathbf{N},\mathbf{R}_*} \frac{b_{y}a_{z}}{S}\Pi_{l\in\mathbf{N}}\frac{a_{l}b_{l}}{S}\Pi_{(j,k)\in\mathbf{R}}\frac{a_{j}b_{k}}{S}].
\end{multline}

We can then construct an upperbound to (\ref{eq:E1}) by identifying the nodes in $\mathbf{R}$ that must equal other nodes in the summation and bound the contrubition of that node's expected in (out) degree by $a_{max}$ ($b_{max}$). Recalling that we defined $p_{max}=a_{max}b_{max}/S$ yields the following,

\begin{multline} \label{eq:E2}
E[P_{r}^{L}(y)] \leq \sum_{\substack{k_{0}+\sum_{i=1}^{r-2}i(k_{i})=r \\ \forall i\in[0..r-2], k_{i} \in [0..r]}}\binom{\sum_{i=0}^{r-2}k_{i}}{k_{0},k_{1},...,k_{r-2}}[\sum_{z=1}^{N}\sum_{\mathbf{N},\mathbf{R}_*} \frac{b_{y}a_{z}}{S}p_{max}^{|\mathbf{R}|}\Pi_{l\in\mathbf{N}}\frac{a_{l}b_{l}}{S}].
\end{multline}

Furthermore, since $\sum_{z=1}^{N} a_{z}=S$, (\ref{eq:E2}) simplies further to,

\begin{multline} \label{eq:E3}
E[P_{r}^{L}(y)] \leq \sum_{\substack{k_{0}+\sum_{i=1}^{r-2}i(k_{i})=r \\ \forall i\in[0..r-2], k_{i} \in [0..r]}}\binom{\sum_{i=0}^{r-2}k_{i}}{k_{0},k_{1},...,k_{r-2}} [\sum_{\mathbf{N},\mathbf{R}_*} b_{y}p_{max}^{|\mathbf{R}|}\Pi_{l\in\mathbf{N}}\frac{a_{l}b_{l}}{S}].
\end{multline}

Now $|\mathbf{R}|$ is precisely the number of repeating edge blocks, $\sum_{i=1}^{r-2}k_{i}$ and the number of possible choices for nodes in a repeating edge block $|\mathbf{R}_{*}|$, is at most $k_{0}^{\sum_{i=1}^{r-2}ik_{i}}$ as any edge in a repeating edge block must equal one of the ($k_{0}$) new edges in the path.  This results in the bound, 

\begin{multline} \label{eq:E4}
E[P_{r}^{L}(y)] \leq b_{y}\sum_{\substack{k_{0}+\sum_{i=1}^{r-2}i(k_{i})=r \\ \forall i\in[0..r-2], k_{i} \in [0..r]}}\binom{\sum_{i=0}^{r-2}k_{i}}{k_{0},k_{1},...,k_{r-2}} p_{max}^{\sum_{i=1}^{r-2}k_{i}}k_{0}^{\sum_{i=1}^{r-2}ik_{i}}[\sum_{\mathbf{N}} \Pi_{l\in\mathbf{N}}\frac{a_{l}b_{l}}{S}].
\end{multline}

For fixed $k_{0},...,k_{r-2}$, this uniquely determines the number of nodes that appear in $\mathbf{N}$ by Lemma \ref{lem:probpath}.  That is, $|\mathbf{N}| = r - 1 - \sum_{i=1}^{r-2}(i+1)k_{i}$. Furthermore, since we are summing over all possible choices of nodes that could be in the new edge interior, we get that

\begin{multline} \label{eq:E5}
E[P_{r}^{L}(y)] \leq \\ b_{y}
(\frac{\mathbf{a}\cdot\mathbf{b}}{S})^{r-1}\sum_{\substack{k_{0}+\sum_{i=1}^{r-2}i(k_{i})=r \\ \forall i\in[0..r-2], k_{i} \in [0..r]}}\binom{\sum_{i=0}^{r-2}k_{i}}{k_{0},k_{1},...,k_{r-2}} p_{max}^{\sum_{i=1}^{r-2}k_{i}}r^{\sum_{i=1}^{r-2}ik_{i}}[\frac{\mathbf{a}\cdot\mathbf{b}}{S}]^{- \sum_{i=1}^{r-2}(i+1)k_{i}},
\end{multline}

where we used the fact that $k_{0} \leq r$. Lemma \ref{lem:expineq}, proved in the appendix, asserts that in general for $l,m,r\in \mathbb{N}$ and $\alpha,\beta \in \mathbb{R}$, where $\beta < 1$, that 

\begin{equation} 
\sum_{\substack{k_{0}+\sum_{i=1}^{m}ik_{i} = r \\ \forall i \in [0..m], k_{i}\in [0..r]}}\binom{\sum_{i=0}^{m}k_{i}}{k_{0},...,k_{m}}\Pi_{i=1}^{l}\alpha^{k_{i}}\Pi_{i=l+1}^{m}\alpha^{k_{i}}\beta^{(i-l)k_{i}} \leq exp(\frac{lr\alpha}{1-\beta}).
\end{equation}

From the above inequality and (\ref{eq:E5}), choosing $l = 1$, $m=r-2$, $\alpha = p_{max}r(\frac{S}{\mathbf{a}\cdot\mathbf{b}})^{2}$ and $\beta = \frac{rS}{\mathbf{a}\cdot\mathbf{b}}$ gives us that,

\begin{equation} \label{eq:E8}
E[P_{r}^{L}(y)] \leq  b_{y}(\frac{\mathbf{a}\cdot\mathbf{b}}{S})^{r-1} exp(\frac{p_{max}(\frac{rS}{\mathbf{a}\cdot\mathbf{b}})^{2}}{1-\frac{rS}{\mathbf{a}\cdot\mathbf{b}}}).
\end{equation}

Now from (\ref{eq:E8}) and (\ref{eq:recurrencegoal}) we conclude that  $$E(P_{r}(y))\leq \sum_{m=1}^{r}b_{y}(\frac{\mathbf{a}\cdot\mathbf{b}}{S})^{m-1}exp(p_{max}\frac{m^{2}(\frac{S}{\mathbf{a}\cdot\mathbf{b}})^{2}}{1-m\frac{S}{\mathbf{a}\cdot\mathbf{b}}})\leq $$

\begin{equation}\label{eq:upperalmostdone} b_{y}(\frac{\mathbf{a}\cdot\mathbf{b}}{S})^{r-1}exp(p_{max}\frac{r^{2}(\frac{S}{\mathbf{a}\cdot\mathbf{b}})^{2}}{1-r\frac{S}{\mathbf{a}\cdot\mathbf{b}}})\sum_{m=1}^{r}(\frac{S}{\mathbf{a}\cdot\mathbf{b}})^{r-m},
\end{equation}

where we factored out an $(\frac{\mathbf{a}\cdot\mathbf{b}}{S})^{r-1}$ and bounded the each of the $exp(p_{max}\frac{m^{2}(\frac{S}{\mathbf{a}\cdot\mathbf{b}})^{2}}{1-m\frac{S}{\mathbf{a}\cdot\mathbf{b}}})$ by $exp(p_{max}\frac{r^{2}(\frac{S}{\mathbf{a}\cdot\mathbf{b}})^{2}}{1-r\frac{S}{\mathbf{a}\cdot\mathbf{b}}})$, as $exp(p_{max}\frac{m^{2}(\frac{S}{\mathbf{a}\cdot\mathbf{b}})^{2}}{1-m\frac{S}{\mathbf{a}\cdot\mathbf{b}}})$ is an increasing function in $m$.  

We can then proceed bounding (\ref{eq:upperalmostdone}) by,
$$b_{y}(\frac{\mathbf{a}\cdot\mathbf{b}}{S})^{r-1}exp(p_{max}\frac{r^{2}(\frac{S}{\mathbf{a}\cdot\mathbf{b}})^{2}}{1-r\frac{S}{\mathbf{a}\cdot\mathbf{b}}})\frac{1}{1-\frac{S}{\mathbf{a}\cdot\mathbf{b}}}\leq$$

$$2b_{y}(\frac{\mathbf{a}\cdot\mathbf{b}}{S})^{r-1}exp(p_{max}\frac{r^{2}(\frac{S}{\mathbf{a}\cdot\mathbf{b}})^{2}}{1-r\frac{S}{\mathbf{a}\cdot\mathbf{b}}}),$$

where the last inequality follows from the assumption that $\frac{\mathbf{a}\cdot\mathbf{b}}{S}>2$.

\end{proof}

Now that we have results regarding the expectation of the number of paths of length r , we now 
seek concentration results regarding the dominating eigenvalue of the adjacency matrix.

\begin{thm}
\label{thm:Markov}
Denote $A$ as a realization of a random Chung-Lu graph with expected degree sequence $\mathbf{d}=(\mathbf{a},\mathbf{b})\in\mathbb{Z}^{N\times 2}$.  Furthermore let $p_{max} =\frac{a_{max}b_{max}}{S}$ and $S=\sum_i a_i =\sum_i b_i$.   Then for every $\epsilon\in (0,\frac{1}{2})$ there exists  $N_1,N_2,N_3 \in \mathbb{N}$ such that if there exists an $r\in\mathbb{N}$ such that $r\frac{S}{\mathbf{a}\cdot\mathbf{b}}<\frac{1}{2}$, $r(\frac{S}{\mathbf{a}\cdot\mathbf{b}})^{2}<\frac{1}{N_1}$, $\frac{\log N}{r}<\frac{1}{N_2}$ and $N > N_{3}$, then $$Pr(\rho(A)\leq(1+\epsilon)\frac{\mathbf{a}\cdot\mathbf{b}}{S})\geq 1-\epsilon$$
\end{thm}
\begin{proof}
First, suppose that 
\begin{equation}\label{eq:epsilon} \frac{\log N}{r}<\frac{\epsilon}{20} \hspace{3pt}, \frac{1}{N}<\epsilon \hspace{3pt} and \hspace{3pt} r(\frac{S}{\mathbf{a}\cdot\mathbf{b}})^{2}<\frac{\epsilon}{20},
\end{equation}
where we assume without loss of generality that $\epsilon < \frac{1}{2}$.  Consider the expected number of paths of length $r$, $E(P_{r})$.  By Theorem \ref{thm:upper},
$$E(P_{r})\leq 2S(\frac{\mathbf{a}\cdot\mathbf{b}}{S})^{r-1}exp(p_{max}\frac{r^{2}(\frac{S}{\mathbf{a}\cdot\mathbf{b}})^{2}}{1-r\frac{S}{\mathbf{a}\cdot\mathbf{b}}}).$$

Then from Markov's Inequality we have that,

$$Pr(P_{r}>Z)\leq \frac{E(P_{r})}{Z}.$$

Now if $Z = N\cdot E(P_{r})$, then, 

$$Pr(P_{r}>Z)\leq \frac{1}{N}.$$

Consequently, with probability at least $1-\frac{1}{N}$ from Lemma \ref{lem:theory},

\begin{equation} \rho(A)^{r}=\rho(A^{r})\leq 2NS(\frac{\mathbf{a}\cdot\mathbf{b}}{S})^{r}exp(p_{max}\frac{r^{2}(\frac{S}{\mathbf{a}\cdot\mathbf{b}})^{2}}{1-r\frac{S}{\mathbf{a}\cdot\mathbf{b}}}).
\end{equation}

Taking the $rth$ root on both sides yields that

\begin{equation}\label{eq:Lmax}\rho(A) \leq (2NS)^{\frac{1}{r}}\frac{\mathbf{a}\cdot\mathbf{b}}{S}exp(p_{max}\frac{r(\frac{S}{\mathbf{a}\cdot\mathbf{b}})^{2}}{1-r\frac{S}{\mathbf{a}\cdot\mathbf{b}}}).
\end{equation}

We first consider the term from (\ref{eq:Lmax}), \begin{equation} \label{eq:part1} 
exp(p_{max}\frac{r(\frac{S}{\mathbf{a}\cdot\mathbf{b}})^{2}}{1-r\frac{S}{\mathbf{a}\cdot\mathbf{b}}})
 \end{equation}

By assumption (\ref{eq:epsilon}) and the facts that $p_{max}\leq 1$, $r\frac{S}{\mathbf{a}\cdot\mathbf{b}}<\frac{1}{2}$, we have that (\ref{eq:part1}) is bounded above by

\begin{equation}\label{eq:answer2}
exp(\frac{2\epsilon}{20})\leq (1+\frac{3\epsilon}{10}),
\end{equation}

where we also used the fact from (\ref{eq:epsilon}) that $exp(x)\leq 1 + 3x$ for $x<1$ (since $\epsilon < \frac{1}{2}$).  Now we consider the other coefficient from (\ref{eq:Lmax}),

\begin{equation}  \label{eq:part2}
(2NS)^{\frac{1}{r}}\leq (2N^{3})^{\frac{1}{r}},
\end{equation}

where the right hand side comes from the fact that the total number of edges must be bounded above by $N^{2}$.  We consider the log of the right hand side of (\ref{eq:part2}).

\begin{equation}  
log(2N^{3})^{\frac{1}{r}} =\frac{3log(N) + log(2)}{r}\leq
\end{equation}

\begin{equation}  
\frac{4log(N)}{r} < \frac{2\epsilon}{10}.
\end{equation}

So we conclude that (\ref{eq:part2}) is bounded above by

\begin{equation}\label{eq:answer1}
exp(\frac{2\epsilon}{10}) \leq (1 + \frac{6\epsilon}{10}).
\end{equation}

since for $x < 1$, $exp(x) \leq 1 + 3x$.  

Using (\ref{eq:answer2}) and (\ref{eq:answer1}) to bound (\ref{eq:Lmax}), we get that,

\begin{equation}\rho(A) \leq \frac{\mathbf{a}\cdot\mathbf{b}}{S}(1 + \frac{3\epsilon}{10})(1+\frac{6\epsilon}{10}) \leq \frac{\mathbf{a}\cdot\mathbf{b}}{S}(1 + \frac{9\epsilon}{10}+\frac{9\epsilon}{100}) \leq \frac{\mathbf{a}\cdot\mathbf{b}}{S}(1 + \epsilon)
\end{equation}

, where in the second to last inequality we invoked the fact that  $\epsilon < \frac{1}{2}$.  

Note that this bound is valid for at least probability $1 - \frac{1}{N}$ and from (\ref{eq:epsilon}), $ 1 - \epsilon \leq 1 - \frac{1}{N}$ and the proof is complete.

\end{proof}

\begin{rem}
Note that from Theorem \ref{thm:Markov}, we have asymptotic convergence as long as $r \gg \log N$ and $r\frac{S}{\mathbf{a}\cdot\mathbf{b}} < \frac{1}{2}$.  In particular from the last inequality if $\frac{\log N}{r} \leq \epsilon \implies r(\frac{S}{\mathbf{a}\cdot\mathbf{b}})^{2}\leq \epsilon$.   So if $r = O(\frac{\mathbf{a}\cdot\mathbf{b}}{S})$, this suggests that we can choose $r$ such that $\epsilon = O({(\log N)}/{\frac{\mathbf{a}\cdot\mathbf{b}}{S}})$. 
\end{rem}

To prove the lower bound, we want to evaluate  $\frac{trace(A^{r})}{N}$, the average of the number of cycles of length $r$ in the network.

\begin{cor}
\label{cor:bounds}
Given an expected bidegree sequence $\mathbf{d}$, for a given realization $A$ it follows that 
\begin{equation}
(\frac{\mathbf{a}\cdot\mathbf{b}}{S})^{r}\leq E(trace(A^{r})).
\end{equation}
\end{cor}
\begin{proof}
This Corollary is essentially an extension of Lemma \ref{lem:lower}, where Lemma \ref{lem:lower} provides an lower bound for the expected number of paths of length $r$, Corollary \ref{cor:bounds} is a statement about the lower bound for the expected number of cycles of length $r$.
To (partially) explain the differences found between this Corollary \ref{cor:bounds} and Lemma \ref{lem:lower}, consider the expected number of cycles of length $1$.  This quantity equals $\sum_{i=1}^{N} \frac{a_ib_i}{S} = \frac{\mathbf{a}\cdot\mathbf{b}}{S}$.
\end{proof}

In order to bound $trace(A^{r})$ from below with high probability, we will compute the variance for the number of paths of length r that start and end at the same node.  We can express $trace(A^{r})$, as a summation of indicator random variables for each possible cycle that could be in a realization of a graph from the Chung-Lu random graph model.   Denote each of these indicator variables as $\mathbf{1}_z$.  It follows then that  $trace(A^r) = \sum\mathbf{1}_z$ and that $$Var(trace(A^{r}))=Var(\sum \mathbf{1}_z)=[\sum Var(\mathbf{1}_z) + \sum_{y\neq z} Cov(\mathbf{1}_y,\mathbf{1}_z)].$$  We are primarily interested in pairs $y,z$ such that $Cov(\mathbf{1}_y,\mathbf{1}_z)\neq 0$ (as we can trivially bound $Var(\mathbf{1}_z)=E(\mathbf{1}_z^{2})-E(\mathbf{1}_z)^{2}=E(\mathbf{1}_z)-E(\mathbf{1}_z)^{2}$ where we used the fact that $E(\mathbf{1}_z^{2})=E(\mathbf{1}_z)$ as $\mathbf{1}_z$ is a Bernoulli random variable and it follows that $Var(\mathbf{1}_z)\leq E(\mathbf{1}_z)$).

But before we bound the variance on the number of cycles of prescribed length, we consider a special case of Lemma \ref{lem:probpath} as it pertains to cycles to simplify the analysis.

\begin{lem} \label{lem:probpathcycle}
Given a cycle of length $r$, $(n_{1},...,n_{r},n_{1})$, we amend the definition of the new edge interior, $\mathbf{N}$ to include the first node in the cycle, if the first and last edges in the cycle are new edges.  If we use this definition for the new edge interior, we can simplify the formula found in Lemma \ref{lem:probpath} that the cycle exists to be that,

\begin{equation}
Pr(\mathbf{1}_{(n_{1},...,n_{r},n_{1})} = 1) = \Pi_{i\in\mathbf{N}}\frac{a_{i}b_{i}}{S}\Pi_{(j,k)\in\mathbf{R}}\frac{a_{j}b_{k}}{S}.
\end{equation}

Furthermore assuming that there are no repeating edge blocks of length $r$, 
$$|\mathbf{N}| = r - \sum_{i\geq 1} (i+1)k_{i},$$
where $k_{i}$ is the number of repeating edge blocks of length $i$.
\end{lem}
\noindent \textbf{Remark:} Note that Lemma \ref{lem:probpathcycle} does not require that the first and last edges in the cycle are new edges.  Additionally, since the first node in the cycle (which also equals the last node in the cycle) can be part of the new edge interior, we have a slightly different formula for $|\mathbf{N}|$ compared to Lemma \ref{lem:probpath}.
\begin{proof}
We omit the proof as it is analogous to the proof of Lemma \ref{lem:probpath}.
\end{proof}
With Lemma \ref{lem:probpathcycle} at hand, we can now compute the bound for the variance for the number of cycles of prescribed length with relative ease.
\begin{thm}\label{thm:cov}
As defined earlier, we denote $trace(A^{r})=\sum_{z=1}^{Z} \mathbf{1}_z$ where $\mathbf{1}_z$ is an indicator random variable denoting the existence (or lack thereof) of a specific cycle of length $r$. Suppose that $2r < \frac{\mathbf{a}\cdot\mathbf{b}}{S}$, then

\begin{equation}
\sum_{\substack{y \neq z}} Cov(\mathbf{1}_y,\mathbf{1}_z)\leq   E(trace(A^{r}))
\frac{\mathbf{a}\cdot\mathbf{b}}{S}^{r}[exp(\frac{(\frac{2rS}{\mathbf{a}\cdot\mathbf{b}})^{2}p_{max}}{1-\frac{2rS}{\mathbf{a}\cdot\mathbf{b}}}) - 1 + (\frac{2rS}{\mathbf{a}\cdot\mathbf{b}})^{r}].
\end{equation}

Furthermore it follows that
\begin{multline}
Var(trace(A^{r})) \leq  E(trace(A^{r}))+\\ E({trace(A^{r})})
\cdot \frac{\mathbf{a}\cdot\mathbf{b}}{S}^{r}[exp(\frac{(\frac{2rS}{\mathbf{a}\cdot\mathbf{b}})^{2}p_{max}}{1-\frac{2rS}{\mathbf{a}\cdot\mathbf{b}}}) - 1 + (\frac{2rS}{\mathbf{a}\cdot\mathbf{b}})^{r}].
\end{multline}
\end{thm}

\begin{proof}
For each indicator random variable $\mathbf{1}_{z}$ that corresponds to the existence of a cycle, define a set $D(z)$, which includes all of the indices of the indicator random variables that are dependent with $\mathbf{1}_{z}$ except for $z$.

We then have that 

\begin{equation}\label{eq:covcrit}
\sum_{y\neq z} Cov(\mathbf{1}_y,\mathbf{1}_z) \leq \sum_{\substack{y \\ z \in D(y)}} E(\mathbf{1}_{y}\mathbf{1}_{z}) = \sum_{\substack{y\\  z\in D(y)}} Pr(\mathbf{1}_{z}=1|\mathbf{1}_{y}=1)Pr(\mathbf{1}_{y}=1).
\end{equation}

We slightly amend our definition of repeating edges to include any edge that we have already observed; in particular if we are conditioning that the cycle $y$ exists, then we have in effect already observed those edges.  We can invoke an extension of Lemma \ref{lem:probpathcycle} to compute the probability that the edges found in the cycle $z$ exist.  

At this juncture, the proof strategy is analogous to Theorem \ref{thm:upper}.   For $i\geq 1$, let $k_{i}$ be the number of repeating edge blocks of length $i$ and let $k_{0}$ be the number of new edges.  We define $\mathbf{R}_{*}$ to be the list of all nodes in a repeating edge block.   To analyze (\ref{eq:covcrit}), with the knowledge of the number of repeating edge blocks of various lengths, we will sum over all possible choices for $\mathbf{N}$ and $\mathbf{R}_*$  But to simplify the analysis, we will consider two cases, where either $k_{r}=0$ or $k_{r} = 1$.

\noindent \textbf{Case 1:} $k_{r} = 0$. \newline 
We are only interested in paths that have a non-trivial covariance.
It follows analogous to equation (\ref{eq:E1}) in Theorem \ref{thm:upper}, we have that 
\begin{multline}
\sum_{\substack{y \\ z \in D(y)}} Cov(\mathbf{1}_y,\mathbf{1}_z)\leq \\ \sum_{y}Pr(\mathbf{1}_{y}=1)\sum_{\substack{k_{0}+\sum_{i=1}^{r-1} ik_{i}=r\\k_{0}\in[0...r-1]}}\binom{\sum k_{i}}{k_{0},...,k_{r-1}}\sum_{\mathbf{N},\mathbf{R}_*}\Pi_{j\in\mathbf{N}}\frac{b_{j}a_{j}}{S}p_{max}^{|\mathbf{R}|},
\end{multline}

where $k_{0}$ cannot equal $r$ as we are summing over cycles with at least one repeating edge, $\mathbf{R}_*$ uniquely determines $\mathbf{R}$, consisting of the first and last nodes for each repeating edge block, where the first node appears in a new edge that precedes the repeating edge block and the last node appears in the new edge that succeeds the repeating edge block.  Now with our revised definition of $\mathbf{N}$, from Lemma \ref{lem:probpathcycle}, we know that $|\mathbf{N}| = (r - 1 - \sum_{i=1}^{r}(i+1)k_{i} + 1) = (r - \sum_{i=1}^{r}(i+1)k_{i})$, where we add $1$ since we can now have the first node in the cycle be part of the new edge interior.  So by summing over all possible choices for nodes in $\mathbf{N}$, we get that 

\begin{multline}\label{eq:cov1}
\sum_{\substack{y \\ z \in D(y)}} Cov(\mathbf{1}_y,\mathbf{1}_z)\leq \\ \sum_{y}Pr(\mathbf{1}_{y}=1)\sum_{\substack{k_{0}+\sum_{i=1}^{r-1} ik_{i}=r\\k_{0}\in[0...r-1]}}\binom{\sum k_{i}}{k_{0},...,k_{r-1}}(\frac{\mathbf{a}\cdot\mathbf{b}}{S})^{(r -\sum_{i=1}^{r}(i+1)k_{i})}\sum_{\mathbf{R}_*}p_{max}^{|\mathbf{R}|}.
\end{multline}

Noting that $|\mathbf{R}|=\sum_{i=1}^{r-1}k_{i}$ and summing over all possible choices for nodes in $\mathbf{R}_{*}$,  we can simplify (\ref{eq:cov1}) further.

\begin{multline}\label{eq:cov01}
\sum_{\substack{y \\ z \in D(i)}} Cov(\mathbf{1}_y,\mathbf{1}_z)\leq \\ \sum_{y}Pr(\mathbf{1}_{y}=1)\sum_{\substack{k_{0}+\sum_{i=1}^{r-1} ik_{i}=r\\k_{0}\in[0...r-1]}}\binom{\sum k_{i}}{k_{0},...,k_{r-1}}(\frac{\mathbf{a}\cdot\mathbf{b}}{S})^{(r-\sum_{i=1}^{r}(i+1)k_{i})}2r^{\sum_{i=1}^{r}ik_{i}} p_{max}^{\sum_{i=1}^{r-1}k_{i}},
\end{multline}
where there are at most $(2r)^{\sum_{i=1}^{r}ik_{i}}$ ways of choosing nodes to be in the list $\mathbf{R}_{*}$.  By factoring out $(\frac{\mathbf{a}\cdot\mathbf{b}}{S})^{r}$, we get that

\begin{multline}\label{eq:cov11}
\sum_{\substack{y \\ z \in D(y)}} Cov(\mathbf{1}_y,\mathbf{1}_z)\leq \\ \sum_{y}Pr(\mathbf{1}_{y}=1)(\frac{\mathbf{a}\cdot\mathbf{b}}{S})^{r}\sum_{\substack{k_{0}+\sum_{i=1}^{r-1} ik_{i}=r\\k_{0}\in[1...r-1]}}\binom{\sum k_{i}}{k_{0},...,k_{r-1}}(\frac{\mathbf{a}\cdot\mathbf{b}}{S})^{-\sum_{i=1}^{r}(i+1)k_{i}}2r^{\sum_{i=1}^{r-1}ik_{i}} p_{max}^{\sum_{i=1}^{r-1}k_{i}}=\\
\sum_{y}Pr(\mathbf{1}_{y}=1)(\frac{\mathbf{a}\cdot\mathbf{b}}{S})^{r}[(\sum_{\substack{k_{0}+\sum_{i=1}^{r-1} ik_{i}=r\\k_{0}\in[0...r]}}\binom{\sum k_{i}}{k_{0},...,k_{r-1}}(\frac{\mathbf{a}\cdot\mathbf{b}}{S})^{-\sum_{i=1}^{r}(i+1)k_{i}}2r^{\sum_{i=1}^{r-1}ik_{i}} p_{max}^{\sum_{i=1}^{r-1}k_{i}}) - 1].
\end{multline}
Note that in the last line we let $k_{0} = r$ and maintain the equality by subtracting off $1$.
Now invoking Lemma \ref{lem:expineq}, where $\alpha = 2rp_{max}\frac{S}{\mathbf{a}\cdot\mathbf{b}}$, $l = 1$, and $\beta = \frac{2rS}{\mathbf{a}\cdot\mathbf{b}}$, we have the upperbound,

\begin{equation}\label{eq:cov15}
 \sum_{y}Pr(\mathbf{1}_{y}=1)\frac{\mathbf{a}\cdot\mathbf{b}}{S}^{r}[exp(\frac{(\frac{2rS}{\mathbf{a}\cdot\mathbf{b}})^{2}p_{max}}{1-\frac{2rS}{\mathbf{a}\cdot\mathbf{b}}}) - 1].
\end{equation}

\noindent \textbf{Case 2:} $k_{r} = 1$ \newline
If $k_{r} = 1$, then the entire cycle $z$ is considered a repeating edge block that exists with probability $1$ if we know the cycle $y$ exists.  Since there are at most $(r)^{r}$ ways of choosing edges for the cycle $z$, the contribution is precisely $ \sum_{y}Pr(\mathbf{1}_{y}=1)(r)^{r}$.

Adding the two cases together, we have that, 

\begin{equation}
\sum_{\substack{y \\ z \in D(y)}} Cov(\mathbf{1}_y,\mathbf{1}_z)\leq   \sum_{y}Pr(\mathbf{1}_{y}=1)
\frac{\mathbf{a}\cdot\mathbf{b}}{S}^{r}[exp(\frac{(\frac{2rS}{\mathbf{a}\cdot\mathbf{b}})^{2}p_{max}}{1-\frac{2rS}{\mathbf{a}\cdot\mathbf{b}}}) - 1 + (\frac{rS}{\mathbf{a}\cdot\mathbf{b}})^{r}],
\end{equation}

which implies that 
\begin{equation}
\sum_{\substack{y \\ z \in D(y)}} Cov(\mathbf{1}_y,\mathbf{1}_z)\leq   E(trace(A^{r}))
\frac{\mathbf{a}\cdot\mathbf{b}}{S}^{r}[exp(\frac{(\frac{2rS}{\mathbf{a}\cdot\mathbf{b}})^{2}p_{max}}{1-\frac{2rS}{\mathbf{a}\cdot\mathbf{b}}}) - 1 + (\frac{rS}{\mathbf{a}\cdot\mathbf{b}})^{r}],
\end{equation}

as $\sum \mathbf{1}_y = trace(A^{r})$.  

And from the discussion preceding this theorem, it follows that

\begin{multline}
Var(trace(A^{r})) \leq  E(trace(A^{r}))+\\ E({trace(A^{r})})
\cdot \frac{\mathbf{a}\cdot\mathbf{b}}{S}^{r}[exp(\frac{(\frac{2rS}{\mathbf{a}\cdot\mathbf{b}})^{2}p_{max}}{1-\frac{2rS}{\mathbf{a}\cdot\mathbf{b}}}) - 1 + (\frac{rS}{\mathbf{a}\cdot\mathbf{b}})^{r}].
\end{multline}
\end{proof}

To construct the desired lowerbound on $\rho(A)$, we could appeal to Chebyshev's Inequality.  Instead we will use a more distribution specific approach.  We will state the result from Janson in full generality and then discuss the implications of their work in context to counting paths and cycles of prescribed length.

\begin{thm}\label{thm:Janson}[Janson 1990 \cite{janson1990poisson}]
Consider a set of independent random indicator variables $\{\mathbf{1}_i\}_{i\in Q}$ and a family 
$\{Q(\alpha)\}_{\alpha\in B}$ of subsets of index set $Q$.  
Define $\mathbf{1}_{\alpha} = \Pi _{i\in Q(\alpha)} \mathbf{1}_{i}$ and $T=\sum_{\alpha\in B} \mathbf{1}_{\alpha}$.  
Define $Cov = \sum_{\alpha_1\neq\alpha_2}Cov(\mathbf{1}_{\alpha_1},\mathbf{1}_{\alpha_2})$  
Then for $\beta \in [0,1]$,

\begin{equation} Pr((1-\beta)E[T]\leq T)\geq 1 - exp(-\frac{1}{2}\frac{(\beta*E[T])^{2}}{E[T]+Cov})
\end{equation}
\end{thm}

\begin{rem}
In the language of Theorem \ref{thm:Janson} according to the Chung-Lu random graph model, the existence of a particular edge ($\mathbf{1}_{i}$) is independent with respect to the existence of another edge in the graph.  As such we define $Q$ to be a set that contains all of the possible $N^{2}$ edges that could exist in our graph of $N$ nodes.  Consider a particular cycle, and denote it by $\alpha$.  We can express the corresponding indicator random variable $\mathbf{1}_{\alpha}$ as a product of independent indicator variables $\Pi_{i\in Q(\alpha)} \mathbf{1}_{i}$, where the elements in the set $Q(\alpha)\subset Q$ identify the edges (the independent indicator variables $\mathbf{1}_{i}$) used to form the cycle $\alpha$.  Let $B$ correspond to all of the sets $Q(\alpha)$ formed by all of the possible cycles $\alpha$ of length $r$ that could appear in our graph.  Hence Theorem \ref{thm:Janson} provides us with a bound that can make it difficult for the sum of cycles (of  prescribed length $r$) to be too much smaller than the expected value.  This leads us to the desired result.
\end{rem}

\begin{thm}\label{thm:SpectralLower}
Denote A as a realization of a random Chung-Lu graph with expected degree sequence $\mathbf{d}=(\mathbf{a},\mathbf{b})\in\mathbb{Z}^{N\times 2}$.  Furthermore let $p_{max} =\frac{a_{max}b_{max}}{S}$ and $S=\sum_i a_i =\sum_i b_i$.   Then for every $\epsilon\in (0,1)$ there exists a $\delta_{1}, \delta_{2}$ such that if for some $r\in\mathbb{N}$,
$r\frac{S}{\mathbf{a}\cdot\mathbf{b}}<\delta_1$ and $\frac{\log N}{r}<\delta_{2}$  then 
$$Pr((1-\epsilon)\frac{\mathbf{a}\cdot\mathbf{b}}{S}\leq \rho(A))\geq 1-\epsilon.$$
\end{thm}
\begin{proof}
First fix an arbitrary $\epsilon \in (0,1)$ and suppose that 
\begin{equation} \label{eq:as1}
r\frac{S}{\mathbf{a}\cdot\mathbf{b}}<min(\frac{1}{10},\sqrt{\frac{-96}{log(\epsilon)}}) \hspace{5pt} and 
\end{equation} 
\begin{equation}\label{eq:as2} \frac{log(2N)}{r} < \frac{\epsilon}{3}.
\end{equation} 

Consider the number of cycles of length $r$, which we denote by $C_r$.  Then we have from Theorems \ref{thm:cov}, \ref{thm:Janson} and Lemma \ref{lem:lower2} that 

 \begin{equation}\label{eq:lower} Pr(\frac{E[C_r]}{2}\leq C_r)\geq 1 - exp(-\frac{1}{8}\frac{E[C_r]^{2}}{E[C_r](1+r^{r} + (\frac{\mathbf{a}\cdot\mathbf{b}}{S})^{r}[exp(\frac{p_{max}(\frac{2rS}{\mathbf{a}\cdot\mathbf{b}})^{2}}{1-\frac{2rS}{\mathbf{a}\cdot\mathbf{b}}})-1])}). 
 \end{equation}
 
Recall that by assumption (\ref{eq:as1}), $r\frac{S}{\mathbf{a}\cdot\mathbf{b}}<\frac{1}{10}$ and get an upperbound that

 \begin{equation}\label{eq:lower22} Pr(\frac{E[C_r]}{2}\leq C_r)\geq 1 - exp(-\frac{1}{8}\frac{E[C_r]^{2}}{E[C_r](1+r^{r} + (\frac{\mathbf{a}\cdot\mathbf{b}}{S})^{r}[exp(5[\frac{rS}{\mathbf{a}\cdot\mathbf{b}}]^{2})-1]}). 
 \end{equation}
 
 We can simplify this further using the facts that $exp(x) \leq 1 + 2x$ for $0\leq x \leq \frac{1}{2}$ and $(\frac{\mathbf{a}\cdot\mathbf{b}}{S})^{r}\leq E[C_r]$ to say that

 \begin{equation}\label{eq:lower23} Pr(\frac{E[C_r]}{2}\leq C_r)\geq 1 - exp(-\frac{1}{96}\frac{(\frac{\mathbf{a}\cdot\mathbf{b}}{S})^{r}}{r^{2}(\frac{\mathbf{a}\cdot\mathbf{b}}{S})^{r-2}}). \implies
 \end{equation}
 
 \begin{equation}\label{eq:lower24} Pr(\frac{E[C_r]}{2}\leq C_r)\geq 1 - exp(-\frac{1}{96}\frac{1}{r^{2}(\frac{S}{\mathbf{a}\cdot\mathbf{b}})^{2}}).
 \end{equation}
 
 Using assumption (\ref{eq:as1}) yields,
 
  \begin{equation}\label{eq:lower25} Pr(\frac{E[C_r]}{2}\leq C_r)\geq 1 - \epsilon.
 \end{equation} 
 
Finally, with probability $1-\epsilon$ we have that
\begin{equation}  \label{eq:loweralmost}
(\frac{1}{2N})^{\frac{1}{r}}\frac{\mathbf{a}\cdot\mathbf{b}}{S}\leq (\frac{E(C_r)}{2N})^{\frac{1}{r}}  \leq (\frac{C_r}{N})^{r}\leq \rho(A), 
\end{equation}

where the first inequality comes from Lemma \ref{lem:lower} and the final inequality holds from Lemma \ref{lem:theory}.

Since assumption (\ref{eq:as2}) implies that \begin{equation} |log (\frac{1}{2N}^{\frac{1}{r}})|= |\frac{log(2N)}{r}| \leq \frac{\epsilon}{3}
\end{equation}

and $1 - 3x \leq exp(-x) \implies 1-\epsilon \leq exp(log(\frac{1}{2N})^{\frac{1}{r}})$.   Consequently  from (\ref{eq:loweralmost}) with probability at least $1 - \epsilon$,

 $$ (1-\epsilon) \frac{\mathbf{a}\cdot\mathbf{b}}{S}\leq \rho(A).   $$

\end{proof}

\begin{rem}
Note that Theorem \ref{thm:SpectralLower} provides an asymptotic lowerbound for the spectral radius if $\frac{\mathbf{a}\cdot\mathbf{b}}{S} \gg \log N$.  In particular, if we choose $$r \approx \frac{\mathbf{a}\cdot\mathbf{b}}{S}\cdot\sqrt{\frac{1}{\log(\frac{\mathbf{a}\cdot\mathbf{b}}{S}/ \log N)}},$$ it follows from Theorem \ref{thm:SpectralLower} that $\epsilon = O(\frac{\log N}{r})=O(\frac{(\log N)\sqrt{\log(\frac{\mathbf{a}\cdot\mathbf{b}}{S}/\log N)}}{\frac{\mathbf{a}\cdot\mathbf{b}}{S}})$.
\end{rem}

We conclude this section with a simulation (Figure \ref{fig:speed}) plotting the empirical distribution of the spectral radius from $100$ realizations of the Chung-Lu random graph model from a fixed expected (bi)-degree sequence where there the networks consist of $600$ nodes and $\frac{\mathbf{a}\cdot\mathbf{b}}{S}\approx 161$.   
More specifically, applying our bounds for the expected value and variance for the number of paths and cycles with length $r$ demonstrate that there is at least $95\%$ probability that the relative error between the empirically observed spectral radius and the asymptotic predictor $\frac{\mathbf{a}\cdot\mathbf{b}}{S}\approx 161$ is bounded by $4\%$.  In spite of the fact that the numeric simulations suggest that we should be able to improve our concentration bounds, as the empirical distribution of the spectral radius is tightly concentrated about the mean, our concentration bounds provide a reasonable method for estimating the spectral radius of a given realization without computing the spectral radius of a realization from the Chung-Lu random graph model.
\begin{figure}
\centering
\includegraphics[scale=.6]{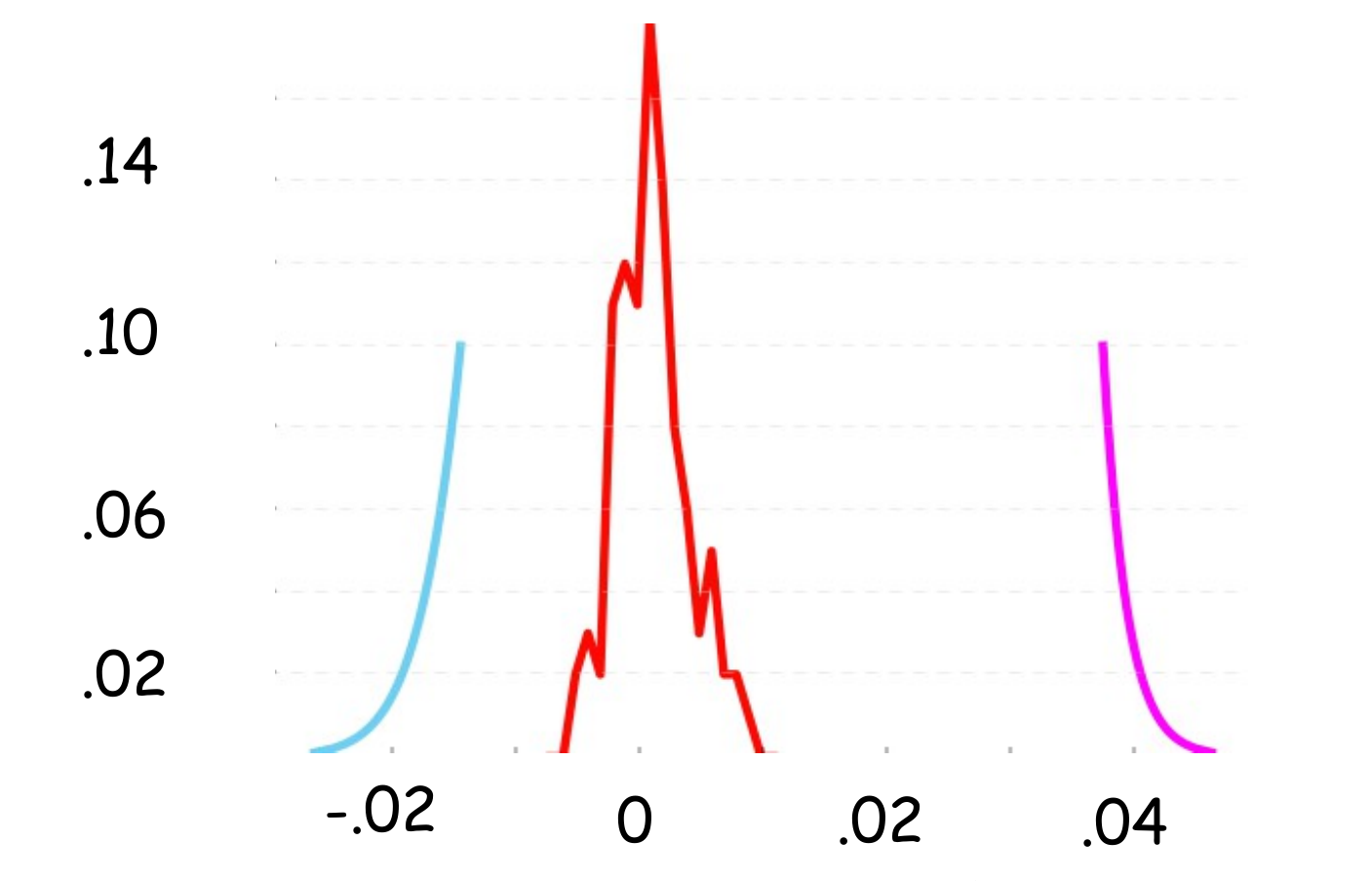}
\caption[Speed of the Convergence of $\rho(A)$ to $\frac{\mathbf{a}\cdot\mathbf{b}}{S}$ in the Chung-Lu Random Graph Model.]{Speed of the Convergence of $\rho(A)$ to $\frac{\mathbf{a}\cdot\mathbf{b}}{S}$ in the Chung-Lu random graph model.  The x-axis indicates the relative error $\epsilon$ of the empirically observed $\rho(A)$ such that $\rho(A) = (1+\epsilon)\frac{\mathbf{a}\cdot\mathbf{b}}{S}$. The $y$ axis marks probabilities, whose meanings vary based on the three curves.  The red curve is the empirical probability distribution function plotting the relative error of the dominating eigenvalue from $\frac{\mathbf{a}\cdot\mathbf{b}}{S}$, where we constructed $100$ realizations of Chung-Lu random graphs with a prescribed expected degree sequence $\mathbf{d}\in \mathbb{Z}^{600 \times 2}$ such that $\frac{\mathbf{a}\cdot\mathbf{b}}{S}\approx 161$.  The magenta curve is an application of the concentration result (Theorem \ref{thm:Markov}) regarding the distribution of the dominating eigenvalue $\rho(A)$ from realizations of Chung-Lu random graphs with the same prescribed expected degree sequence $\mathbf{d}$.  For this curve, the $y-axis$ provides an upperbound on the probability the spectral radius exceeds the relative error from $\frac{\mathbf{a}\cdot\mathbf{b}}{S}$ on the x-axis.  Similarly, the blue curve is an application of the concentration result from Theorem \ref{thm:SpectralLower}.}
\label{fig:speed}
\end{figure}

\section{Concentration Bounds when $\lowercase{p}_{max} \rightarrow 0$}
We now wish to extend our results to the case where $\mathbf{a}\cdot\mathbf{b}/S$ is (asymptotically) finite.  In order to prove results of this nature, we will require that  $\lowercase{p}_{max} \rightarrow 0$, that is the likelihood any two fixed nodes share an edge should vanish asymptotically.  We again stress, as suggested in Figure  \ref{fig:speed}, that our results not only provide asymptotic information regarding the concentration of the dominating eigenvalue for a sequence of realizations of Chung-Lu random graphs, but also computable  bounds that the dominating eigenvalue deviates from $\mathbf{a}\cdot\mathbf{b}/S$  for  randomly generated Chung-Lu graphs with a fixed number of nodes as well.  For this purpose, we introduce the following notation.

\begin{defn}
A simple cycle is a path that begins and ends at the same node where no other node is visited more than once.
Denote the number of simple cycles of length $r$ by $SC_r$.
\end{defn}

Since the number of simple cycles of length $r$ is a lower bound for the $trace(A^{r})$, it suffices to show that with high probability that the number of simple cycles of length $r$ is roughly $(\frac{\mathbf{a}\cdot\mathbf{b}}{S})^{r}$.

\begin{lem}
\label{lem:lower2}
Consider a realization of the Directed Chung-Lu random graph model with expected degree sequence $\mathbf{d} = (\mathbf{a},\mathbf{b})$, where $\sum a_{i} = \sum b_{i} = S$.  Then  
\begin{equation} \label{eq:mainlower2} (\frac{\mathbf{a}\cdot\mathbf{b}}{S})^{r} - \binom{r}{2}p_{max}(\frac{\mathbf{a}\cdot\mathbf{b}}{S})^{r-1}\leq E(SC_{r}).
\end{equation}
\end{lem}
\begin{proof}
Denote the set $\mathbf{D}$ as a collection of all possible simple cycles allowing for cyclic permutations.  It then follows that the expected number of simple cycles of length $r$ (allowing for cyclic permulations) is,

$$E(SC_{r}) = \sum_{(i_1,i_2,...,i_r,i_1)\in \mathbf{D}}\Pi_{k=1}^{r} p_{i_{k}i_{k+1}} =\sum_{(i_1,i_2,...,i_r,i_1)\in \mathbf{D}}\Pi_{k=1}^{r} \frac{a_{i_{k}}b_{i_{k}}}{S}.$$

The first  equality follows from the fact that with simple cycles we do not have to worry about an edge repeating in a cycle (as in Lemma \ref{lem:lower}) and the second equality follows from invoking the definition of Chung-Lu and rearranging terms.

To derive the lower bound we use inclusion-exclusion, 
\begin{equation}\label{eq:lowersimple}
\sum_{(i_1,i_2,...,i_r,i_1)\in \mathbf{D}}\Pi_{k=1}^{r} \frac{a_{i_{k}}b_{i_{k}}}{S} \geq \sum_{i_1=1,...,i_r = 1}^{N}\Pi_{k=1}^{r} \frac{a_{i_{k}}b_{i_{k}}}{S} - \binom{r}{2}\sum_{\substack{i_1=i_2=1\\ i_3=1,...,i_r = 1}}^{N}\Pi_{k=1}^{r} \frac{a_{i_{k}}b_{i_{k}}}{S},
\end{equation}
where for the first term in the right hand side of (\ref{eq:lowersimple}), we sum over all possible cycles, not just simple cycles.  Subsequently to correct this, we subtract off the cases where two nodes equal each other.  By symmetry, the choice of which two nodes equal each other is arbitrary, so we select $i_1 = i_2$ and multiply the quantity by $\binom{r}{2}$, the number of ways of choosing two nodes to equal each other.  It then follows that, 

\begin{multline}
E(SC_{r}) = \sum_{(i_1,i_2,...,i_r,i_1)\in \mathbf{D}}\Pi_{k=1}^{r} \frac{a_{i_{k}}b_{i_{k}}}{S} \geq \\ (\frac{\mathbf{a}\cdot\mathbf{b}}{S})^{r} - \binom{r}{2}p_{max}\sum_{\substack{i_2=1\\ i_3=1,...,i_r = 1}}^{N}\Pi_{k=2}^{r} \frac{a_{i_{k}}b_{i_{k}}}{S}\geq (\frac{\mathbf{a}\cdot\mathbf{b}}{S})^{r} - \binom{r}{2}p_{max}(\frac{\mathbf{a}\cdot\mathbf{b}}{S})^{r-1},
\end{multline}

which completes the proof.
\end{proof}

We now provide a bound for the variance for the number of simple cycles of prescribed length. 

\begin{lem}  \label{lem:lowercov}
If $\frac{\mathbf{a}\cdot\mathbf{b}}{S} > 1$, we then have that \begin{equation} \label{eq:covpmax} Var(SC_{r}) \leq E(SC_{r})\cdot [r + (\frac{\mathbf{a}\cdot\mathbf{b}}{S})^{r}(exp(\frac{p_{max}(\frac{rS}{\mathbf{a}\cdot\mathbf{b}})^{2}}{1-\frac{S}{\mathbf{a}\cdot\mathbf{b}}}) - 1)].
\end{equation}

\end{lem}

\begin{proof}
The proof is similar to Theorem \ref{thm:cov}.  In particular, we can express 

\begin{equation}\label{eq:SC}
SC_{r} = r\sum_{y\in \mathbf{Y}}\mathbf{1}_y,
\end{equation} 

where $\mathbf{1}_y$ is an indicator random variable corresponding to the existence of the cycle $y$ of length $r$ and all of the cycles in $\mathbf{Y}$ are {\em distinct} under cyclic permutations.  

Analogous to (\ref{eq:covcrit}), we have that 

\begin{equation}
\sum_{y\neq z} Cov(\mathbf{1}_y,\mathbf{1}_z) \leq \sum_{\substack{y \\ z \in D(y)}} E(\mathbf{1}_{y}\mathbf{1}_{z}) = \sum_{\substack{y\\  z\in D(y)}} Pr(\mathbf{1}_{z}=1|\mathbf{1}_{y}=1)Pr(\mathbf{1}_{y}=1),
\end{equation}

where $z \in D(y)$ if the simple cycles $y,z\in\mathbf{Y}$ share an edge and $z \neq y$.  Following the argument from Theorem \ref{thm:cov}, we will argue that

\begin{multline}
\sum_{\substack{y\in\mathbf{Y} \\ z \in D(y)}} Cov(\mathbf{1}_y,\mathbf{1}_z)\leq \\ \frac{1}{r}\sum_{y}Pr(\mathbf{1}_{y}=1)\sum_{\substack{k_{0}+\sum_{i=1}^{r-2} ik_{i}=r\\k_{0}\in[0...r-1]}}\binom{\sum k_{i}}{k_{0},...,k_{r-2}}\sum_{\mathbf{N},\mathbf{R}_*}\Pi_{j\in\mathbf{N}}\frac{b_{j}a_{j}}{S}p_{max}^{|\mathbf{R}|},
\end{multline}
where $\mathbf{N}$ is a list of all of the nodes in the cycle $z$ that are part of the new edge interior, $\mathbf{R}_*$ identifies the repeating edge blocks that appear in $z$ after observing the edges in the cycle $y$ and $k_{0} < r$ as the cycles $y$ and $z$ share at least one edge.  Since we sum over all possible choices for $\mathbf{N}$ and $\mathbf{R}_{*}$, we are counting the same simple cycle (under cyclic permutation) $r$ times.  To correct this issue, we multiply the result by $\frac{1}{r}$.  $k_{0}$ indicates the number of new edges and for $i\geq 1$, $k_{i}$ denotes the number of repeating edge blocks in cycle $z$ of length $i$.  Note that since $\mathbf{Y}$ consists of simple cycles that are unique under cyclic permutations, a repeating edge block cannot consist of lengths equal to $r-1$ and $r$ as the cycle $z$ is simple.

We will now show that for simple cycles, 
\begin{equation}|\mathbf{R}_*| = r^{\sum_{i=1}^{r-2}k_{i}}. 
\end{equation}  
Since $z$ is a simple cycle, all repeating edge blocks must come from edges observed from the cycle $y$. Furthermore, for a repeating edge block of length $i$, as $y$ is simple, there are at most $r$ ways of constructing a repeating edge block.  (For example, if we are selecting a repeating edge block of length $2$ using edges from a simple cycle $y = (i_1,i_2,...,i_r,i_1)$, then there are only $r$ choices for choosing repeating edge block.  It could be, $(i_1,i_2,i_3),(i_2,i_3,i_4),...$,etc.  Other choices like $(i_1,i_7,i_8)$ would be impossible as we know there is an edge from $i_6$ to $i_7$ in $y$ and if there was an edge from $i_1$ to $i_7$ in $y$ as well, then this would imply that $y$ is not simple.)  Hence, to compute $|\mathbf{R}_*| $, we multiply by $r$ for every repeating edge block.  It then follows that,

\begin{multline}
\sum_{\substack{y\in\mathbf{Y} \\ z \in D(y)}} Cov(\mathbf{1}_y,\mathbf{1}_z)\leq \\ \frac{1}{r}\sum_{y}Pr(\mathbf{1}_{y}=1)\sum_{\substack{k_{0}+\sum_{i=1}^{r-2} ik_{i}=r\\k_{0}\in[0...r-1]}}\binom{\sum k_{i}}{k_{0},...,k_{r-2}}\sum_{\mathbf{N}}\Pi_{j\in\mathbf{N}}\frac{b_{j}a_{j}}{S}(rp_{max})^{\sum_{i=1}^{r-2}k_{i}}.
\end{multline}

Using the fact that from Lemma \ref{lem:probpathcycle}, $|\mathbf{N}| = r - \sum_{i=1}^{r-2}(i+1)k_{i}$, summing over all possible choices for nodes to appear in $\mathbf{N}$, allowing $k_{0}$ to equal $r$, and subtracting off the contribution when $k_{0}=r$, we get that

\begin{multline}\label{eq:simplecov1}
\sum_{\substack{y\in\mathbf{Y} \\ z \in D(y)}} Cov(\mathbf{1}_y,\mathbf{1}_z)\leq \\ \frac{1}{r}\sum_{y}Pr(\mathbf{1}_{y}=1)(\frac{\mathbf{a}\cdot\mathbf{b}}{S})^{r}([\sum_{\substack{k_{0}+\sum_{i=1}^{r-2} ik_{i}=r\\k_{0}\in[0...r]}}\binom{\sum k_{i}}{k_{0},...,k_{r-2}}\Pi_{i=1}^{r-2}(r(\frac{S}{\mathbf{a}\cdot\mathbf{b}})^{i+1}p_{max})^{k_{i}}]-1).
\end{multline}

By invoking Lemma \ref{lem:expineq}, where $\alpha = \frac{rp_{max}S}{\mathbf{a}\cdot\mathbf{b}}$, $\beta = \frac{S}{\mathbf{a}\cdot\mathbf{b}}$ and $l = 1$, we bound (\ref{eq:simplecov1}) by

\begin{equation}
\label{eq:simplecov4}
\frac{1}{r}\sum_{y}Pr(\mathbf{1}_{y}=1)(\frac{\mathbf{a}\cdot\mathbf{b}}{S})^{r}(exp(\frac{r^{2}(\frac{S}{\mathbf{a}\cdot\mathbf{b}})^{2}p_{max}}{1-\frac{S}{\mathbf{a}\cdot\mathbf{b}}}) - 1).
\end{equation}

To finish off the proof, note that $\sum_{y\in\mathbf{Y}}Pr(\mathbf{1}_y = 1) = \frac{E(SC_{r})}{r},$
as $\mathbf{Y}$ only consists of simple cycles that are unique under cyclic permutation.  Using the fact that 
$Var(SC_{r}) = Var(r\sum_{y\in\mathbf{Y}}\mathbf{1}_y ) = r^{2}\sum_{y\in\mathbf{Y}}Pr(\mathbf{1}_y = 1)+r^{2}\sum_{\substack{y\in\mathbf{Y} \\ z \in D(y)}} Cov(\mathbf{1}_y,\mathbf{1}_z)$ proves the result.
\end{proof}

Consequently, we have the following concentration result.

\begin{cor}
Denote A as a realization of a random Chung-Lu graph with expected degree sequence $\mathbf{d}=(\mathbf{a},\mathbf{b})\in\mathbb{Z}^{N\times 2}$.  Furthermore let $p_{max} =\frac{a_{max}b_{max}}{S}$ and $S=\sum_i a_i =\sum_i b_i$.  Then for every $\epsilon\in (0,1)$ there exists  a $\delta_{1}$,$\delta_{2}$ such that if for some choice of $r\in\mathbb{N}$,
$p_{max}r^{2}<\delta_{1}$, $\frac{\mathbf{a}\cdot\mathbf{b}}{S}>1$ and $\frac{\log N}{r}<\delta_{2}$  then 
$$Pr((1-\epsilon)\frac{\mathbf{a}\cdot\mathbf{b}}{S}\leq \rho(A))\geq 1-\epsilon.$$
\end{cor}

\begin{proof}
The proof is analogous to Theorem \ref{thm:SpectralLower} and as such we only provide a sketch. From Lemmas \ref{lem:lowercov} and \ref{lem:lower2}, since $r^{2}p_{max}\rightarrow 0$, we have that $Var(SC_{r})\ll E(SC_{r})^{2}$.  It then follows that with high probability the number of simple cycles of length $r$ cannot be less than half the expected number of simple cycles.  
 
 $$ \frac{E(SC_{r})}{2N} \leq \frac{SC_{r}}{N} \leq \frac{trace(A^{r})}{N} \leq \rho(A)^{r}. $$
 
 As $\frac{\frac{\mathbf{a}\cdot\mathbf{b}}{S}^{r}(1-r^{2}p_{max})}{2N}\leq \frac{E(SC_{r})}{2N}$, this would impliy that,
 
 $$(\frac{1}{2N})^{\frac{1}{r}}\frac{\mathbf{a}\cdot\mathbf{b}}{S}(1-r^{2}p_{max})^{\frac{1}{r}}\leq \rho(A).$$
 
 Since $(\frac{1}{2N})^{\frac{1}{r}} = exp(\frac{-\log 2N }{r})\rightarrow 1$ and $r^{2}p_{max}\rightarrow 0$ by assumption, we conclude that with probability of at least $(1-\epsilon)$,
 
 $$(1-\epsilon)\frac{\mathbf{a}\cdot\mathbf{b}}{S}\leq \rho(A).$$
 
\end{proof}

Constructing a meaningful upperbound on the number of paths of length $r$ when $p_{max} \rightarrow 0$ is more challenging.   We cannot simply follow the proof strategy used in Theorem \ref{thm:upper} as we have to worry about the contribution from repeating edge blocks.  (Note that in the proof of Theorem \ref{thm:upper}, we were able to ignore this complication altogether due to the assumption that $\frac{\mathbf{a}\cdot\mathbf{b}}{S}\rightarrow \infty$ sufficiently fast.)    In particular, consider the (unlikely) event that a subgraph of $k$ nodes exists where each node has bidirectional edges with each of the other $k$ nodes and $\frac{\mathbf{a}\cdot\mathbf{b}}{S}\ll k$.   To circumvent this issue, we will illustrate that with high probability when $p_{max}\rightarrow 0$, two cycles of modest length cannot be close together in distance in the graph.  Subsequently, we will employ this fact to construct a more precise bound on the number of choices for constructing repeating edge blocks in a path, to prove the desired asymptotics.  For this purpose, we will need the following machinery.

\begin{defn}
Define the minimal edge list of a path $P$ to be an ordered list of new edges (of minimal size) required for the entire path to exist.  Note that since $P$ can only contain new edges, we cannot have the same edge appear twice in the minimal edge list.  By convention when constructing a minimal edge list, as we observe edges in a path, we simply add the edge to the minimal edge list if we have never observed that particular edge before.  
\end{defn}
\noindent \textbf{Example:} Consider the path $P = (1,2,3,1,2,4)$.  For the path $P$ to exist, we only need the following edges to exist: $1\rightarrow 2$, $2\rightarrow 3$, $3\rightarrow 1$ and $2\rightarrow 4$.  These edges $\{(1,2),(2,3),(3,1),(2,4)\}$ would then form the minimal edge list of path $P$.

\begin{defn}
Given an list edges, we say that the $mth$ edge in the list is a cycle inducing edge if we can construct a simple cycle, where the cycle contains the $mth$ edge and can only consist of the first $m$ edges in $E$.  Analogously, we can say an edge is a cycle inducing edge in a path if the edge is a cycle inducing edge in the minimal edge list of the path.
\end{defn}

Essentially, we will seek a result that says that paths cannot have too many cycle inducing edges.  More precisely, it will be burdensome to consider the minimal edge list of a path $P$ with excessively many cycle inducing edges.  Instead, we will want to construct a smaller list (of edges) from the minimal edge list with a smaller number of cycle inducing edges.  The following lemma formalizes this claim and the proof explains how to construct such a list.

\begin{lem} \label{lem:reduce} Consider a path $P$ with $k>1$ cycle inducing edges and its corresponding minimal edge list $M$.  Then for any $j \leq k$, we can construct a \textbf{reduced edge list} $E \subset M$ with the following properties:
\begin{itemize} \item $E$ has exactly $j$ cycle inducing edges.
\item The first node in the first edge in $E$ and the last node in the last edge in $E$ must belong to a simple cycle that can be formed from the edges in $E$.
\end{itemize}
\end{lem}
Before we provide the proof of the Lemma, we first provide an example.
\textbf{Example:} Consider the path $P = (1,2,2,3,4,3,3)$.  For the path $P$ to exist, we only need the following edges to exist: $\{(1,2),(2,2),(2,3),(3,4),(4,3),(3,3)\}$, which would then form the minimal edge list for the path $P$.

\begin{itemize}
	\item  First, starting with the empty list we add edges from our minimal edge list $M$ to our reduced edge list $E$ until we find $j$ cycle inducing edges.  For this example, consider $j = 2$. This yields the edge list $$E=\{(1,2),(2,2),(2,3),(3,4),(4,3)\}.$$
	\item Then, we remove edges from the beginning of the edge list $E$ until we reach an edge such that its removal would decrease the number of simple cycles that we can form our current edge list.  This yields the list $$E  = \{(2,2),(2,3),(3,4),(4,3)\}.$$
	\item Observe that if any of the edges in the edge list $E$ do not exist, then the path $P$ cannot exist.  Furthermore by construction, the first and last edge in our new edge list must belong to a simple cycle that we can construct from $E$.  
	\end{itemize}

\begin{proof} 
\textbf{Step 1:} Initialize a list $E = \emptyset$.  \textbf{Step 2:} Starting with the first edge in $M$ we proceed by inductively adding edges to $E$ until $E$ contains $j$ cycle inducing edges.  Then in \textbf{Step 3} starting from the first edge in $E$, we remove edges from $E$ if the edge is not a part of a simple cycle that can be formed from the edges in $E$.  Once we reach an edge that is part of a simple cycle, we stop removing edges in $E$.  It follows that all cycle inducing edges in $E$ are still cycle inducing edges as we only removed edges in Step 3 that are not part of a simple cycle.  Hence $E$ has precisely $j$ cycle inducing edges.  Furthermore by construction, the first edge in $E$ is part of a cycle, as otherwise it would have been deleted from $E$ and the last edge in $E$ is also part of a cycle as it is a cycle inducing edge.

\end{proof}

With Lemma \ref{lem:reduce} at hand, we seek one more lemma so that we can prove the desired result that there cannot exist a path with many cycle inducing edges with high probability.

\begin{lem} \label{lem:represent}
Suppose $E$ is a reduced edge list with $t$ cycle inducing edges, as we constructed in Lemma \ref{lem:reduce}.  We can map the reduced edge list $E$ to a union of lists $\cup_{i=1}^{t}M_{i}$ where each $M_{i}$ consists of a list of nodes.  Furthermore, the $M_{i}$ lists have the following properties:
\begin{itemize} 
\item The first node in $M_{1}$ and the last node in $M_{t}$ all belong to a simple cycle that can be constructed from the edges in $E$ (provided that the edges in $E$ exist).
\item The last node in $M_{j}$ for each $j$ is a node  that has appeared earlier either in the same list $M_{j}$ or in a list $M_{i}$ where $i<j$.
\item The first node in $M_{j}$ for each $j>1$ is a node that has appeared in some list $M_{i}$ where $i<j$.
\item And finally, let $x_{k,i}$ be the $kth$ node in $M_{i}$.  Then the probability that all edges in the reduced edge list $E$ exist equals \begin{equation}
\Pi_{i=1}^{t}\Pi_{k=1}^{|M_{i}|-1} \frac{b_{x_{k,i}}a_{x_{k+1,i}}}{S}. 
\end{equation}
\end{itemize}
\end{lem}

\begin{proof}
\textbf{Proof of Bullet 1:} We first illustrate the mapping of the reduced edge list $E$ with $t$ cycle inducing edges to the union of lists $\cup_{i=1}^{t}M_{i}$ where each $M_{i}$ consists of a list of nodes.  We start by decomposing our reduced edge list $E$ as a union of edge lists $\cup_{i=1}^{t} E_{i}$, where we add the edges from $E$ to $E_{1}$ and stop once $E_{1}$ has precisely one cycle inducing edge.  Then starting where we left off, we add edges to $E_{2}$ and stop once $E_{1}\cup E_{2}$ have precisely two cycle inducing edges.  Note that we can express $E = \cup_{i=1}^{t} E_{i}$ such that the list $\cup_{i=1}^{k} E_{i}$ has precisely $k$ cycle inducing edges for all $k\leq t$.
\newline\newline
\textbf{Example:} We clarify the above procedure with the following example.  Consider the reduced edge list $E = \{(2,2),(2,3),(3,4),(4,3)\}$.  Then since $E$ has two cycle inducing edges, we can write $E_{1} = \{(2,2)\}$, since $E_{1}$ already has one cycle inducing edge we stop here and then let $E_{2} = \{(2,3),(3,4),(4,3)\}$.  
\newline\newline
Now given $E_{i}$, we can construct $M_{i}$ as follows.  Consider the first node of each edge in $E_{i}$ and add those nodes to $M_{i}$ in that order.  Then add the last node of the last edge of $E_{i}$ to $M_{i}$.  This completes the construction of the $M_{i}$.
\newline\newline
\textbf{Example:} As before consider $E = E_{1} \cup E_{2} = \{(2,2)\} \cup \{(2,3),(3,4),(4,3)\}$.  It then follows that $M_{1} = \{2,2\}$, where we added the first node of each edge of $E_{1}$ and then added the last node of the last edge of $E_{1}$ to $M_{1}$ .  Similarly, $M_{2} = \{2,3,4,3\}$. 
\newline\newline
By construction, the first node of the first edge and the last node of the last edge in $E$ belong to a simple cycle that can be constructed from the edges of $E$.  Furthermore, these nodes are precisely the first and last nodes of $M_{1}$ and $M_{t}$ respectively.  Hence the first bulleted statement holds. 
\newline\newline
\textbf{Proof of Bullet 2:} By construction of $\cup_{i=1}^{k} E_{i}$, the last edge of $E_{k}$ is a cycle inducing edge.  Consequently, the last node of the last edge appears as the last node in $M_{k}$ and must have appeared elsewhere in the list $\cup_{i=1}^{k} M_{i}$.
\newline\newline
\textbf{Proof of Bullet 3:} Suppose that we have a contradiction, that is the first node for some $M_{j}$, $j>1$, does not appear in an earlier $M_{i}$, where $i < j$.   As we demonstrated earlier in the proof, we can express $E$ as $\cup_{i=1}^{t}E_{i}$. Now denote the first edge in $E_{j}$ as $(x,y_{1})$ and the edge that precedes $(x,y_{1})$ in the path as $(w,x)$.  If $(w,x)$ is the first edge in $E_{1}$ or the first appearance of the edge $(w,x)$ is after the first edge in $E_{1}$, then $x$ would necessarily appear in an earlier $M_{i}$.
\newline\newline
So suppose that is not the case and consider the path that corresponds to the reduced edge list; it must have the form $(...,w,x,...m_{1},...,w,x,y_{1},...)$, where $m_{1}$ is the first node in $M_{1}$.  Now consider the node that follows the first appearance of the edge $(w,x)$; call it $y_{2}$.  It then follows that the original path has the form, $(...,w,x,y_{2},...m_{1},...,w,x,y_{1})$.  Now if we consider the graph formed by the edges of this path, it follows that there exists a simple cycle containing the node $x$.  As such, we must be able to construct this simple cycle with the node $x$ from the reduced edge list $\cup_{i=1}^{j-1}E_{i}$ as we only omit edges from the minimal edge list if they do not belong to cycles.  Furthermore, since $x$ appears in a cycle, it needs an incoming and an outgoing edge and hence, the node $x$ must appear in an earlier $E_{i}$. 
 \newline
\noindent \textbf{Proof of Bullet 4:} Denote the probability that all edges in $E$ exist by $Pr(E)$.  It follows that since $E$ is a only contains edges from the minimal edge list of a path $P$, all of the edges are new edges and that
\begin{equation}Pr(E) = \Pi_{(x,y)\in E}\frac{b_{x}a_{y}}{S}=\Pi_{i=1}^{t}\Pi_{(x,y)\in E_{i}}\frac{b_{x}a_{y}}{S}. 
\end{equation}

Now consider the $kth$ edge in $E_j$ for some particular choice of $j$. Denote this edge as $(x,y)$. It follows from construction of $M_j$ that $x$ is precisely the $kth$ element in $M_j$ and $y$ is the $k+1st$ element in $M_j$.  Consequently, we conclude that

\begin{equation}Pr(E) = \Pi_{i=1}^{t}\Pi_{(x,y)\in E_{i}}\frac{b_{x}a_{y}}{S} = \Pi_{i=1}^{t}\Pi_{k=1}^{|M_{i}|-1} \frac{b_{x_{k,i}}a_{x_{k+1,i}}}{S},
\end{equation}

 where $x_{k,i}$ is the $kth$ node in $M_{i}$.
\end{proof}

The strength of Lemma \ref{lem:represent} lies in the fact that it helps us identify which nodes repeat in a reduced edge list.  For example  consider some arbitrary bounded function $f:\mathbb{N}\rightarrow \mathbb{R}$ and define $f_{*}$ to be the smallest upperbound of $f$.  Then it follows that, $\sum_{i=1,j=1}^{N}[f(i)]^{2}f(j)\leq f_{*}\sum_{i=1,j=1}^{N}f(i)f(j)$; if $f_{*}\rightarrow 0$ and $\sum_{i=1}^{N}f(i)=O(1)$, then this upperbound converges to $0$.  To summarize, we will want to identify which indices (nodes) repeat in the reduced edge list as if for example we chose the wrong index and bounded $\sum_{i=1,j=1}^{N}[f(i)]^{2}f(j)$ by $f_{*}\sum_{i=1,j=1}^{N}[f(i)]^{2}\leq Nf_{*}\sum_{i=1}^{N}[f(i)]^{2}$, the presence of the $N$ would lead us to a potentially useless upperbound as it is possible that $f_{*}\rightarrow 0$ and $\frac{1}{f_{*}^{2}}\ll N$. At this juncture, we present the following result that restricts the number of cycle inducing edges that can appear in a reduced edge list.

\begin{lem} \label{lem:pmax}  
Consider a sequence of (expected) degree sequences $\mathbf{d} = (\mathbf{a},\mathbf{b})$ where $p_{max} \leq \frac{R}{N^{\tau}}$, $R$ is a fixed constant, $\tau > 0$ and $\frac{\mathbf{a}\cdot\mathbf{b}}{S} > 1$.  Then with probability at least $p_{*} = 1 - \epsilon$, all paths of length not exceeding   $L=\frac{(t-1)\tau}{2}*log_{\frac{\mathbf{a}\cdot\mathbf{b}}{S}}(N)$, have less than  $t$ cycle inducing edges, where $\epsilon = \frac{tR^{t-1}(L+1)^{3t-2}}{N^{\frac{(t-1)\tau}{2}}}$.  \newline
\end{lem}
\noindent \textbf{Remark}: Note that asymptotically we are guaranteed that $\epsilon \rightarrow 0$ in Lemma \ref{lem:pmax} as for fixed $t$, $p_{max} = O(N^{-\tau})$, $L = O(log(N))$ and consequently $\epsilon = O((N^{\frac{-\tau}{2}}(log N)^{3})^{t})\rightarrow 0$.

\begin{proof}
To bound the likelihood of the existence of any path of length at most $L$ that consists of at least $t$ cycle inducing edges, we will instead consider the likelihood of the existence of a reduced edge list, $E_{*}$, with $t$ cycle inducing edges containing no more than $L+1-t$ distinct nodes. Denote $Pr(E_*)$ as the probability all edges in the list $E_*$ exist. It follows from Lemma \ref{lem:represent} that for a given reduced edge list $E_*$ with $t$ cycle inducing edges that

\begin{equation}\label{eq:crit}
Pr(E_*)=\Pi_{i=1}^{t}\Pi_{k=1}^{|M_{i}|-1} \frac{b_{x_{k,i}}a_{x_{k+1,i}}}{S} = \Pi_{i=1}^{t}\frac{b_{x_{1,i}}a_{x_{|M_{i}|,i}}}{S}\Pi_{k=2}^{|M_{i}|-1}\frac{b_{x_{k,i}}a_{x_{k,i}}}{S}=
\end{equation}
\begin{equation} \frac{b_{x_{1,1}}a_{x_{|M_{t}|,t}}}{S} \Pi_{j=1}^{t-1}{\color{black}\frac{a_{x_{|M_{j}|,j}}b_{x_{1,j+1}}}{S}}\Pi_{i=1}^{t}\Pi_{k=2}^{|M_{i}|-1}{\color{black}\frac{b_{x_{k,i}}a_{x_{k,i}}}{S}}=
\end{equation}
\begin{equation}
\frac{b_{x_{1,1}}}{S^{t}}\Pi_{j=1}^{t}a_{x_{|M_{j}|,j}}\Pi_{h=1}^{t-1}b_{x_{1,h+1}}\Pi_{i=1}^{t}\Pi_{k=2}^{|M_{i}|-1}{\color{black}\frac{b_{x_{k,i}}a_{x_{k,i}}}{S}},
\end{equation}
where $x_{k,i}$ is the $kth$ node in $M_{i}$. 

But by Lemma \ref{lem:represent} we have constraints on some of the nodes in the lists $M_{i}$.

More precisely for some functions $\alpha, \beta, \delta$ and $\eta$, we have that 
\begin{enumerate}
	\item \label{constraint1} For all $j>1$, $x_{1,j} = x_{\alpha(j),\beta(j)}$, where $\beta(j) < j$.
	\item \label{constraint2} For all $j$, $x_{|M_{j}|,j} = x_{\delta(j),\eta(j)}$, where $\eta(j) \leq j$ and if $\eta(j) = j \implies \delta(j) < |M_{j}|$. 
	\item \label{constraint3} And finally, there exists a $\gamma$ such that $x_{|M_{\gamma}|,\gamma}=x_{1,1}$.	
\end{enumerate}

The last bulleted claim comes from the fact that by construction of our reduced edge list, the first node $x_{1,1}$ is part of a simple cycle that can be formed using the edges in our reduced edge list; furthermore, this node is also part of a cycle inducing edge.  

Let $\mathbf{X}_z$ denote a particular list whose elements, $\mathbf{x}$, are vectors that satisfy the constraints \ref{constraint1}, \ref{constraint2} and \ref{constraint3}, such that $z$ uniquely defines the functions $\alpha,\beta,\delta,\eta$, the sizes of the lists $M_{i}$, and the parameter $\gamma$ above.

Now define $\mathbf{E}_t$ to be the event that there exists a reduced edge lists $E_{*}$ with $t$ cycle inducing edges that contains no more than $L+1-t$ distinct nodes, where $t$ is a parameter with the constraint that $t\geq 2$.

By summing over all possible choices of $z$ and $\mathbf{x} \in \mathbf{X}_{z}$, we have that
\begin{equation}\label{eq:pmaxpr}
Pr(\mathbf{E}_t)\leq \sum_{z}\sum_{\mathbf{x}\in \mathbf{X}_{z}}\frac{b_{x_{1,1}}}{S^{t}}\Pi_{j=1}^{t}a_{x_{|M_{j}|,j}}\Pi_{h=1}^{t-1}b_{x_{1,h+1}}\Pi_{i=1}^{t}\Pi_{k=2}^{|M_{i}|-1}{\color{black}\frac{b_{x_{k,i}}a_{x_{k,i}}}{S}}.
\end{equation}

But for a fixed $\mathbf{X}_z$, we know that $x_{|M_{\gamma(z)}|,\gamma(z)}=x_{1,1}$, (\ref{eq:pmaxpr}) simplifies to,

\begin{equation}
 \sum_{z}\sum_{\mathbf{x} \in \mathbf{X}_{z}}\frac{b_{x_{1,1}}a_{x_{1,1}}}{S^{t}}\Pi_{j=1,j\neq \gamma(z)}^{t}a_{x_{|M_{j}|,j}}\Pi_{h=1}^{t-1}b_{x_{1,h+1}}\Pi_{i=1}^{t}\Pi_{k=2}^{|M_{i}|-1}{\color{black}\frac{b_{x_{k,i}}a_{x_{k,i}}}{S}}.
\end{equation}
Furthermore, since for a fixed choice of $\mathbf{X}_{z}$ $x_{|M_{j}|,j}$ and $x_{1,j}$ are uniquely determined by $\alpha, \beta, \delta$ and $\eta$, this yields that,

\begin{equation}
Pr(\mathbf{E}_t)\leq\sum_{z}\sum_{\mathbf{x}\in \mathbf{X}_{z}}\frac{b_{x_{1,1}}a_{x_{1,1}}}{S^{t}}a_{x_{max}}^{t-1}b_{max}^{t-1}\Pi_{i=1}^{t}\Pi_{k=2}^{|M_{i}|-1}{\color{black}\frac{b_{x_{k,i}}a_{x_{k,i}}}{S}}=
\end{equation}

\begin{equation}
\sum_{z}\sum_{\mathbf{x}\in \mathbf{X}_{z}}\frac{b_{x_{1,1}}a_{x_{1,1}}}{S}p_{max}^{t-1}\Pi_{i=1}^{t}\Pi_{k=2}^{|M_{i}|-1}{\color{black}\frac{b_{x_{k,i}}a_{x_{k,i}}}{S}}.
\end{equation}

Now we know that for any choice of $\mathbf{X}_{z}$, $[\sum_{i=1}^{t} (|M_{i}|-1) -1] + 1\leq L$.  As a result,

\begin{equation} \label{eq:starpmax}
Pr(\mathbf{E}_t)\leq\sum_{z}p_{max}^{t-1}(\frac{\mathbf{a}\cdot\mathbf{b}}{S})^{L}.
\end{equation}

At this juncture, we want to identify the number of possible of lists $\mathbf{X}_{z}$ that we can construct to simplify (\ref{eq:starpmax}).

Notice that since there are at most $L+1$ nodes in the path, all of the possible sizes for each of the $t$ $'M_{i}'$'lists is bounded above by $(L+1)^{t}$.  We also have $t$ choices for choosing which $x_{|M_{\gamma}|,\gamma}$ must equal $x_{1,1}$.  And finally, we have that at most $(L+1)^{2t-2}$ choices for requiring that $x_{1,j} = x_{\alpha(j),\beta(j)}$, where $j\neq 1$ and $x_{|M_{j}|,j} = x_{\delta(j),\eta(j)}$, where $j\neq \gamma$.  Hence there are at most $t(L+1)^{3t-2}$ choices for constructing the list $\mathbf{X}_{z}$. 

So we conclude that 

\begin{equation}
Pr(\mathbf{E}_t)\leq t(L+1)^{3t-2}p_{max}^{t-1}(\frac{\mathbf{a}\cdot\mathbf{b}}{S})^{L}.
\end{equation}

Choose $L \leq \frac{(t-1)\tau}{2}log_{\frac{\mathbf{a}\cdot\mathbf{b}}{S}}(N)$.  Note that $(\frac{\mathbf{a}\cdot\mathbf{b}}{S})^{L} \leq N^{\frac{(t-1)\tau}{2}}$.

Hence it then follows that
\begin{equation}
Pr(\mathbf{E}_t) \leq \frac{t(L+1)^{3t-2}N^{\frac{(t-1)\tau}{2}}R^{t-1}}{N^{(t-1)\tau}}=\frac{t(L+1)^{3t-2}R^{t-1}}{N^{(\frac{t-1}{2})\tau}}.
\end{equation}

where we invoked the fact that $p_{max} \leq \frac{R}{N^{\tau}}$.  
\end{proof}

Lemma \ref{lem:pmax} asymptotically guarantees that with high probability any path in a realization of a Chung-Lu graph of sufficiently small length cannot contain more than one cycle inducing edge.  Such a restriction will help us in counting the number of repeating edge blocks.  More specifically, we would like to relate the number of distinct simple cycles we can construct from a reduced edge list with $t$ cycle inducing edges. 

\begin{lem} \label{lem:cyclelength}
Consider a graph where no path of length less than  $l_{*}= \frac{\tau}{2}log_{\frac{\mathbf{a}\cdot\mathbf{b}}{S}}(N)$ has more than one cycle inducing edge.  Then from the minimal edge list of any path in the graph with precisely $t$ cycle inducing edges, we can construct at most $t$ simple cycles of length less than $\frac{l_{*}}{2}$ that are distinct under cyclic permutation.  
\end{lem}
\begin{proof}
	Suppose there exists a path with $t$ cycle inducing edges, where we can construct $t+1$ distinct simple cycles with length less than $\frac{l_{*}}{2}$ from the minimal edge list.  It follows that there exists at least one cycle inducing edge that could be used to construct two distinct simple cycles, both of length less than  $\frac{l_{*}}{2}$.  Now, consider an auxilliary path formed by adjoining the two simple cycles together (as they share a common edge).  It follows immediately that for this auxilliary path, there are at least two cycle inducing edges in the corresponding minimal edge list.  Since the total length of the two cycles is less than $l_{*}$, we arrive at a contradiction, as we assumed that our graph cannot have a path of length less than $l_{*}$ with more than one cycle inducing edge.
\end{proof}

Lemmas \ref{lem:pmax} and \ref{lem:cyclelength} have both practical and theoretical significance.  In particular, while an assortment of results pertaining to the diameter of various undirected random graphs demonstrate that the diameter is $O(\log N)$ \cite{molloy1998size,chung2002connected,chung2004average,boccaletti2006complex}, relatively little work has focused on the diameter for realizations of directed random graph models \cite{karp1990transitive,cooper2004size,flaxman2004diameter,newman2001random}.  As Lemma \ref{lem:pmax} suggests, if the diameter of a graph is $O(\log N)$, then paths connecting two distinct simple cycles must appear at $O(\log N)$ length.   \newline

In addition, Lemmas \ref{lem:pmax} and \ref{lem:cyclelength} add clarity to the adage that realizations of certain random graph models are locally tree like, as simple cycles of small length must be relatively far apart in distance.  As a side remark, it is notworthy that while real world networks are not locally tree like, there is evidence that supports that dynamical processes that occur on locally tree like graphs approximate the dynamics on real world networks \cite{melnik2011unreasonable}.   \newline

We now sketch the remainder of our proof strategy for showing asymptotic convergence of the spectral radius.  The basic idea is to partition a repeating edge block into subblocks of length $O(\log N)$.   Then, Lemmas \ref{lem:pmax} and \ref{lem:cyclelength} bound the number of distinct simple cycles that can appear in the repeating edge subblock, which is at most $1$.  In particular we claim (and will prove later on) that for any such repeating edge subblock, there exists two {\bf simple} paths, $P_{1},P_{2}$ and a {\bf simple} cycle $C$, such that the repeating edge subblock is a concatenation of paths of the form: $P_{1},C,...,C,P_{2}$, where the same cycle $C$ can appear any number of times in between the two simple paths.   \newline

From Lemma \ref{lem:cyclelength}, we know that the number of simple cycles $C$ that we can choose from in the repeating edge subblock is bounded by the number of cycle inducing edges.  Analogously, we want to find a bound for the number of choices of simple paths, $P_{1},P_{2}$ in the repeating edge subblock. 

\begin{defn}
Given a path with $t$ cycle inducing edges, we call the subgraph formed by these edges a $t-path$ graph.
\end{defn}

\begin{lem}\label{lem:tpath}
For any $t-path$ graph, the number of simple paths of length $l$ from a fixed node is bounded by $(1+\frac{t}{l})^{l}$.  Furthermore, the number of simple paths from a fixed node of any prescribed length is bounded by $exp(t)$.
\end{lem}

\textbf{Remark:} Care should be taken when invoking Lemma \ref{lem:tpath}, as for $t$ sufficiently large, the bound in Lemma \ref{lem:tpath} could exceed the number of simple paths in the entire graph.

\begin{proof}
Since the edges in our $t-path$ graph comes from a path with $t$ cycle inducing edges, we can construct our $t-path$ graph using the following procedure.  Start by fixing the number of nodes in the $t-path$ graph, $N_{*}$, and having each node possess one incoming and one outgoing edge.  Then we can add up to $t$ additional edges to the graph based on  the $t$ cycle inducing edges in the path.  (Using induction leads to quick proof that any $t-path$ graph can be constructed in this manner).  \newline\newline Consequently, we define $f_{d}(t)$ to be an upperbound for the number of simple paths of length $d$ from any fixed node in any  $t-path$ graph.  Suppose the initial node that maximizes this quantity has out-degree $x_{1}+1$.  It follows that 
$$f_{d}(t) \leq \max_{x_{1}\in [0..t]}\sum_{i=1}^{x_{1}+1}f_{d-1}(t-x_{1}) =\max_{x_{1}\in [0..t]}\ (1+x_{1})f_{d-1}(t-x_{1}),$$
where the first inequality comes from the fact that we are only considering simple paths.  More specifically by considering the particular node, with out-degree $1+x_{1}$, and graph that maximizes $f_{d}(t)$, we can instead bound this quantity by the number of paths of length $d-1$ coming from this node's neighbors, where we essentially delete the node as we do not consider non-simple paths. Note that if we have a $t-path$ graph with no cycle inducing edges, then we can let $f_{d_*}(0)=1$ for any $d_{*}$.  Furthermore, since the maximum number of paths of length $0$ from a given node is bounded by $1$, we get that  $f_{0}(t_{*})=1$ for any $t_{*}$.  Proceeding inductvely we conclude that,

$$f_{d}(t)\leq \max_{\substack{\sum_{i=1}^{d} x_{i} \leq t \\ \forall i\in[1..d],x_{i}\geq 0 }}\Pi_{i=1}^{d} (1+x_{i})\leq (1+\frac{t}{d})^{d}\leq exp(t),$$

where the second to last inequality can be proven by induction on $d$. 
\end{proof}
We now verify the claim that any path with at most one cycle inducing edge has the following special decomposition.

\begin{lem} \label{lem:cycledecomp}
Consider a graph, where no path of length less than $l_{*}$ has more than one cycle inducing edge.  Then for any path $P$ of length less than $\frac{l_{*}}{2}$, $P$ has the decomposition that $$P = (P_{1},C,...,C,P_{2}),$$ where $P_{1}$ and $P_{2}$ are simple paths and $C$ is a simple cycle.
\end{lem}
\begin{proof}
By the assumption in the lemma, there can be only one cycle inducing edge in the path in  $P$, which we will identify as $(c_{l},c_{0})$.  Denote the portion of the path up to the cycle inducing edge as $(x_{0},....,c_{l},c_{0})$.
Furthermore, since $(c_{l},c_{0})$ is an edge that can be used to form a cycle from the edges in the path $(x_{0},....,c_{l},c_{0})$, this implies that $c_{0}$ appears earlier in the path; hence we can express this initial portion of the path $P$ as $(x_{0},...,c_{0},....,c_{l},c_{0})$.    \newline

Note that if we remove the node $c_{0}$ at the end of the path, we are guaranteed that the path $(x_{0},...,c_{l})$ is simple as there are no cycle inducing edges.  Hence we are guaranteed that if we decompose the path $(x_{0},...,c_{0},....,c_{l},c_{0})$ into a path $P_{1} = (x_{0},...,c_{0})$ and the cycle $C = (c_{0},...c_{l},c_{0})$, that both $P_{1}$ and $C$ are simple.   \newline

We can continue to traverse the edges in the simple cycle $C$ as the path $(P_{1},C,...,C)$ still would only have one cycle inducing edge.  Eventually we may traverse an edge that is not in $C$, so we can define $P = (P_{1},C,...,C,P_{2})$.  As we want this representation of $P$ to be unique, we require that $P_{2}\neq (C,P_{3})$, for some path $P_{3}$.  To prove the desired claim, at this juncture we just need to show that $P_{2}$ is also simple. \newline

Now suppose $P_{2}$ is not simple. If we revisit a node in $P_{2}$ that is not in the simple cycle $C$ then this implies we can construct two distinct cycles from a path with one cycle inducing edge, a contradiction according to Lemma \ref{lem:cyclelength}.  \newline

\begin{figure}
\centering
\includegraphics[scale=.3]{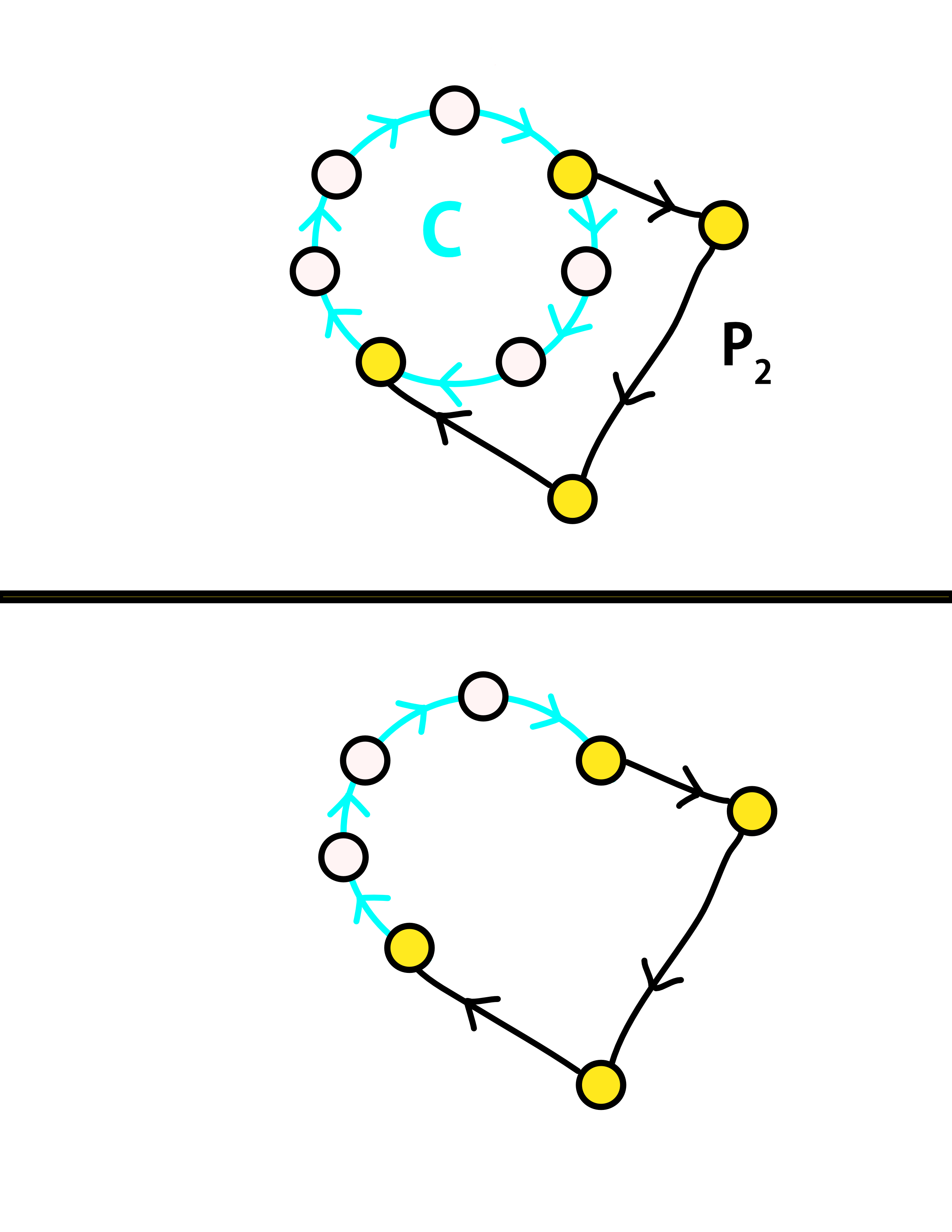}
\caption[Hi]{An illustration that after departing from the cycle $C$, $P_{2}$ cannot revisit a node in $C$ as then we would be able to construct two distinct simple cycles, $C$ and the cycle illustrated in the bottom panel.  To distinguish between the edges in $C$ and the nodes/edges after departing from the cycle $C$, the edges from $C$ are highlighted in turquoise. In contrast, the nodes in the path $P_{2}$ after departing from the cycle $C$ are highlighted in yellow, while the edges are in black.}
\label{fig:cycledecomp}
\end{figure}

In the case that we do revisit a node in $P_{2}$ and that node also appears in the simple cycle $C$, we provide a sketch of the proof (literally) and leave the details to the reader.  See the top panel in Figure \ref{fig:cycledecomp}.  We highlighted the cycle $C$ with blue edges and consider the edges from $P_{2}$ starting with the first edge that deviates from the cycle until we return to a node in the cycle $C$, denoted with black edges.  As noted in the bottom panel, if we do return to a node in the cycle $C$, we can use the edges in $C$ and $P_{2}$ to construct another distinct simple cycle.  But this would imply we could construct two distinct simple cycles from a path $P$ with only one cycle inducing edge, a contradiction according to Lemma \ref{lem:cyclelength}. 

\end{proof}

At this juncture, we can present an upperbound for the number of paths that is especially helpful when $p_{max}$ is small.

\begin{thm} \label{thm:pmaxtrick}
Suppose $\frac{\mathbf{a}\cdot\mathbf{b}}{S} > 1$.  For some parameter $m$, let $r$, the length of the path, satisfy the inequality  $\frac{(m-1)\tau}{2} \log_{\frac{\mathbf{a}\cdot\mathbf{b}}{S}}(N)\leq r \leq  \frac{m\tau}{2} \log_{\frac{\mathbf{a}\cdot\mathbf{b}}{S}}(N)$.  Define a collection of graphs $\mathbf{G}$, such that for any graph $G \in \mathbf{G}$, there is no path of length $r$ that has more than $m+1$ cycle inducing edges and there is no path in $\mathbf{G}$ of length less than $\frac{\tau}{2}\log_{\frac{\mathbf{a}\cdot\mathbf{b}}{S}}(N)$ that has more than $1$ cycle inducing edge.  Now let $P_{r}(y|G\in \mathbf{G})$ be the number of paths of length $r$ starting at node $y$ given that the randomly generated graph $G$ is in $\mathbf{G}$. Define $$\eta =  3(m+1)[\frac{\tau}{4}\log_{\frac{\mathbf{a}\cdot\mathbf{b}}{S}}(N)]^{2}exp(2m),$$ and suppose that  $$\frac{\tau}{4}\log_{\frac{\mathbf{a}\cdot\mathbf{b}}{S}}(N)> 1$$ and $$(\eta)^{\frac{4}{\tau\log_{\frac{\mathbf{a}\cdot\mathbf{b}}{S}}(N)}}<\frac{\mathbf{a}\cdot\mathbf{b}}{S}.$$  Furthermore denote the bound on the probability that there exists a path of length $L$ with $t$ cycle inducing edges, from Lemma \ref{lem:pmax} as $$p_{*}(L,t) = 1 - \frac{tR^{t-1}(L+1)^{3t-2}}{N^{\frac{(t-1)\tau}{2}}},$$ where $p_{max} \leq \frac{R}{N^{\tau}}$, $R$ is a fixed constant and $\tau > 0$. Then

$$E(P_{r}(y|G\in \mathbf{G}))\leq \frac{b_{y}}{Pr(G\in\mathbf{G})(1-\frac{S}{\mathbf{a}\cdot\mathbf{b}})}(\frac{\mathbf{a}\cdot\mathbf{b}}{S})^{r-1}exp(\frac{r\eta p_{max}(\frac{S}{\mathbf{a}\cdot\mathbf{b}})^{2}(\eta)^{\frac{4}{\tau\log_{\frac{\mathbf{a}\cdot\mathbf{b}}{S}}(N)}}}{1-\frac{S}{\mathbf{a}\cdot\mathbf{b}}(\eta)^{\frac{4}{\tau\log_{\frac{\mathbf{a}\cdot\mathbf{b}}{S}}(N)}}}),$$

where \begin{equation}\label{eq:probbound}Pr(G\in\mathbf{G}) \geq p_{*}(\frac{\tau \log_{\frac{\mathbf{a}\cdot\mathbf{b}}{S}}N}{2},2)+p_{*}(r,m+1)-1.
\end{equation}
\end{thm}
\noindent{\bf Remark:} Theorem \ref{thm:pmaxtrick} proves asymptotically that the spectral radius is bounded above by $\frac{\mathbf{a}\cdot\mathbf{b}}{S}$, assuming that $\frac{\mathbf{a}\cdot\mathbf{b}}{S}>1$.  More specifically if $p_{max} = O(N^{-\tau})$ and $\log N \ll r \ll (\log N)^{2}$ then it follows that $m \ll \log N$, $\lim_{N\rightarrow \infty}Pr(G\in\mathbf{G})= 1$ by (\ref{eq:probbound}),$\lim_{N\rightarrow \infty} p_{max}r\eta = 0$ and $\lim_{N\rightarrow \infty} \eta^{\frac{4}{\tau*log_{\frac{\mathbf{a}\cdot\mathbf{b}}{S}}(N)}} = 1$.  Hence an application of Markov's Inequality, as used in Theorem \ref{thm:Markov}, would prove the result.

\begin{proof}

The proof is nearly identical to Theorem \ref{thm:upper}. 

To construct a lower bound for the $Pr(G \in \mathbf{G})$, denote $A$ as the event that there does not exist a path of length at most $r$ with $m+2$ cycle inducing edges and $B$ as the event that there does not exist a path of length at most $\frac{\tau log_{\frac{\mathbf{a}\cdot\mathbf{b}}{S}}N}{2}$ with $2$ cycle inducing edges.  It follows from Lemma \ref{lem:pmax} that the $Pr(G \in \mathbf{G}) = Pr(A \cap B) = Pr(A) + Pr(B) - Pr(A \cup B) \geq p_{*}(r,m+1)+p_{*}(\frac{\tau log_{\frac{\mathbf{a}\cdot\mathbf{b}}{S}}N}{2},2) - 1$, where the last inequality holds as we bound the $Pr(A \cup B)$ above by $1$. \newline\newline

But before we can proceed, since we are restricting the number of cycle inducing edges for paths of certain lengths in $G$, the probability an edge exists changes.  Let $\mathbf{1}_{\mathbf{E}}$ be the indicator random variable, where $\mathbf{1}_{\mathbf{E}} = 1$ if all of the edges in $\mathbf{E}$ appear in our random realization $G$.   Then we have that,

\begin{equation} \label{eq:probG} Pr(\mathbf{1}_{\mathbf{E}} = 1| G \in \mathbf{G}) \leq \frac{Pr(\mathbf{1}_{\mathbf{E}} = 1)}{Pr(G \in \mathbf{G})}.
\end{equation}

So (\ref{eq:probG}) allows us to proceed with constructing an upperbound on the number of paths with length $r$ for a graph $G \in \mathbf{G}$, where we do not have to worry about the influence that $G$ belongs to $\mathbf{G}$ on the probability that a path exists.  \newline\newline

Now let $P_{r}^{L}(y|G\in \mathbf{G})$ be the number of paths starting from node $y$ of length $r$, where we require the last edge in the path to not be a repeating edge, we condition on the fact that  $G\in\mathbf{G}$.  Denote the $k_{i}$ for $i>0$ as the number of repeating edge blocks of length $i$.   Similar to equation (\ref{eq:E5}), we will argue that 

\footnotesize
\begin{equation} \label{eq:star111}  E(P_{r}^{L}(y|G\in \mathbf{G}))\leq \hspace{350pt}
\end{equation}
$$\frac{b_{y}}{Pr(G\in\mathbf{G})}\sum_{\substack{k_{0}+\sum_{i=1}^{r} ik_{i}=r \\ k_{0} \in [1..r]}}\binom{\sum k_{i}}{k_{0},...,k_{r}}(\frac{\mathbf{a}\cdot\mathbf{b}}{S})^{r-1}\Pi_{i\geq 1}(\eta p_{max}[\frac{S}{\mathbf{a}\cdot\mathbf{b}}]^{i+1})^{k_{i}} (\eta)^{\frac{4ik_{i}}{\tau*log_{\frac{\mathbf{a}\cdot\mathbf{b}}{S}}(N)}},
$$
\normalsize
where  $\eta =3(m+1)[\frac{\tau}{4}\log_{\frac{\mathbf{a}\cdot\mathbf{b}}{S}}(N)]^{2}exp(2m)$ and $r \leq \frac{m\tau}{2}log_{\frac{\mathbf{a}\cdot\mathbf{b}}{S}}(N)$.   The key difference from equation (\ref{eq:E5}) is by letting $p_{max}\rightarrow 0$, instead of having $r^{ik_{i}}$ possible choices for each of the $k_{i}$ repeating edge blocks of length $i$, we claim that we have at most \newline $\eta^{1+\frac{4i}{\tau\log_{\frac{\mathbf{a}\cdot\mathbf{b}}{S}}(N)}}$ choices for each repeating edge block of length $i$.  We will temporarily assume that (\ref{eq:star111}) holds. Applying Lemma \ref{lem:expineq}, where $\alpha = \eta p_{max}(\frac{S}{\mathbf{a}\cdot\mathbf{b}})^{2}(\eta)^{1+\frac{4}{\tau\log_{\frac{\mathbf{a}\cdot\mathbf{b}}{S}}(N)}}$, $\beta = \frac{S}{\mathbf{a}\cdot\mathbf{b}}(\eta)^{\frac{4}{\tau\log_{\frac{\mathbf{a}\cdot\mathbf{b}}{S}}(N)}}$ and $l = 1$, yields the upperbound,

\begin{equation}
\frac{b_{y}}{Pr(G\in\mathbf{G})}(\frac{\mathbf{a}\cdot\mathbf{b}}{S})^{r-1}exp(\frac{rp_{max}(\frac{S}{\mathbf{a}\cdot\mathbf{b}})^{2}(\eta)^{1+\frac{4}{\tau*log_{\frac{\mathbf{a}\cdot\mathbf{b}}{S}}(N)}}}{1-\frac{S}{\mathbf{a}\cdot\mathbf{b}}(\eta)^{\frac{4}{\tau*log_{\frac{\mathbf{a}\cdot\mathbf{b}}{S}}(N)}}}).
\end{equation}

Invoking the fact that $E(P_{r}(y|G\in \mathbf{G})) \leq \sum_{j=1}^{r}E(P_{j}^{L}(y|G\in\mathbf{G}))$ yields the inequality in the theorem. \newline
To prove that (\ref{eq:star111}) holds, we first consider the number of ways for filling in a repeating edge block of length less than $\frac{\tau}{4}log_{\frac{\mathbf{a}\cdot\mathbf{b}}{S}}(N)$.  (In the theorem statement we assume this quantity to be greater than $1$).  By Lemma \ref{lem:cycledecomp}, we know that the repeating edge block must be of the form $(P_{1},C,...,C,P_{2})$, where $P_{1}, P_{2}$ are simple paths and $C$ is a simple cycle. \newline

\textbf{Case 1:} Suppose that there is a simple cycle $C$ in the repeating edge block.
Fix a particular choice for the simple cycle $C$ and assume that it has length $l$. 
\begin{itemize}
\item Since the repeating edge block has length at most  $\frac{\tau}{4}\log_{\frac{\mathbf{a}\cdot\mathbf{b}}{S}}(N)$, there are {\color{red} $\frac{\tau}{4}\log_{\frac{\mathbf{a}\cdot\mathbf{b}}{S}}(N)$} places to put the first cycle $C$ in the repeating edge block $(P_{1},C,...,C,P_{2})$.
\item We can continue traversing the cycle $C$ in the repeating edge block up to {\color{red}$\frac{\tau}{4l}\log_{\frac{\mathbf{a}\cdot\mathbf{b}}{S}}(N)$} times.
\item Because $C$ has length $l$, the the first node in the cycle $C$ can be any one of the {\color{red}$l$} nodes in the cycle.
\item Since we identified the first and last nodes in $C$ from the prior step,Lemma \ref{lem:tpath} tells us that there are at most {\color{red}$exp(2m)$} choices for choosing the simple paths $P_{1}$ and $P_{2}$.
\end{itemize}
Multiplying all of the red terms together, we have that for a fixed cycle $C$, there are 
$$ exp(2m)[\frac{\tau}{4}\log_{\frac{\mathbf{a}\cdot\mathbf{b}}{S}}(N)]^{2}$$ choices for filling in the repeating edge block.  Now by Lemma \ref{lem:cyclelength} and the assumptions in the theorem, we know that there are at most $m+1$ such cycles to choose from.  Hence we get that Case 1 gives us at most 
\begin{equation}
 (m+1)exp(2m)[\frac{\tau}{4}\log_{\frac{\mathbf{a}\cdot\mathbf{b}}{S}}(N)]^{2}
\end{equation}
choices for filling in a repeating edge block of length at most  $\frac{\tau}{4}\log_{\frac{\mathbf{a}\cdot\mathbf{b}}{S}}(N)$.

\textbf{Case 2:}There is no simple cycle $C$ in the repeating edge block.  In this case, Lemma \ref{lem:cycledecomp} implies that we can represent the repeating edge block as a single simple path. As we have at most $r$ possible nodes that could be at the start of the repeating edge block, this implies that from Lemma \ref{lem:tpath} there are at most $$rexp(m)$$ choices for filling in the repeating edge block.

Combining the brounds from Case 1 and Case 2 and employing the inequality relating $r$ and $m$ in the theorem statement, yield that we have at most

\begin{equation}
 rexp(m) + (m+1)[\frac{\tau}{4}\log_{\frac{\mathbf{a}\cdot\mathbf{b}}{S}}(N)]^{2}exp(2m) \leq 3(m+1)[\frac{\tau}{4}\log_{\frac{\mathbf{a}\cdot\mathbf{b}}{S}}(N)]^{2}exp(2m) = \eta
\end{equation}

choices for filling in a repeating edge block of length at most $\frac{\tau}{4}\log_{\frac{\mathbf{a}\cdot\mathbf{b}}{S}}(N)$.

To apply a similar bound to a repeating edge block of any length, we divide the edge block into subblocks each of length approximately $\frac{\tau}{4}\log_{\frac{\mathbf{a}\cdot\mathbf{b}}{S}}(N)$ and possibly one edge block that is smaller than the others.  The number of choices for this larger repeating edge block is then bounded by $(\eta)^{1+\frac{4i}{\tau\log_{\frac{\mathbf{a}\cdot\mathbf{b}}{S}}(N)}}$ by considering the number of possible choices for each smaller repeating edge subblock of length bounded by $\frac{\tau}{4}\log_{\frac{\mathbf{a}\cdot\mathbf{b}}{S}}(N)$.   
\end{proof}

\section{Partitioned Chung-Lu Model}
Since real world networks exhibit community structure, we want to consider a random graph model that allows for this feature.  But we also want a random graph model that is easily amenable to analysis and hence emulates many of the features of the Chung-Lu random graph model.  We achieve this goal as follows.  In the special case where there are two communities we can envision partitioning our adjacency matrix into four submatrices as illustrated in Figure \ref{fig:partition}.  For each submatrix we have expected row sums and column sums, given by $b^{(x,y)}$ and $a^{(x,y)}$, where the superscript identifies the submatrix under consideration.  With this information, we assign edges in each of the submatrices based on the values of these expected row sums and column sums.
\begin{figure}
\centering
\includegraphics[scale=.6]{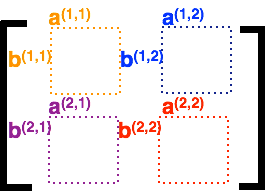}
\caption[An Illustration of the Partitioned Chung-Lu model.]{ We partition our adjacency matrix with submatrices each possessing expected column sums,  $\mathbf{a}^{(x,y)}$, and row sums, $\mathbf{b}^{(x,y)}$.  The probability that the ijth entry in the submatrix (x,y) equals $\frac{\mathbf{b}_{i}^{(x,y)}\mathbf{a}_{j}^{(x,y)}}{S_{xy}}$, where $S_{xy}$ is the expected sum of the entries in the x,y submatrix.}
\label{fig:partition}
\end{figure}

More formally, we have the following definition.
\begin{defn}
We define the Directed Partitioned Chung-Lu Random Graph Model such that we are given a collection of expected degree sequences for the subgraphs corresponding to the partitioned submatrices in the adjacency matrix.  We construct a directed edge from node $i$ to $j$ by means of an independent Bernoulli random variable $p_{ij}$ where $p_{ij}$ is proportional to the product of the expected out-degree of node $i$ and expected in-degree of node $j$ in the corresponding subgraph.
\end{defn}

Closely related models to the Directed Partitioned Chung-Lu model, have been considered in application to the community detection for undirected graphs \cite{chaudhuri2012spectral,karrer2011stochastic,nadakuditi2013spectra,peixoto2013eigenvalue}.   In contrast, \cite{chauhan2009spectral} studies a similar model that generates directed graphs, but their analysis assumes that the off-diagonal submatrices are very sparse, which would allow for block diagonal approximations of the adjacency matrix.  

The first result we prove holds in considerable generality.  Therefore, we introduce the following  (more general) definition.
\begin{defn}
In the K-Partitioned Random Graph Model, we assign each node to one of $K$ groups (or communities), denoted by the function $G(\cdot)$.   We then construct an edge from node $i$ to $j$ by means of an independent Bernoulli random variable $p_{ij}$ where $p_{ij}$ depends on $G(i)$ and $G(j)$ . 
\end{defn}

We will find that the following norm will be helpful in proving bounds for the dominating eigenvalue in the Chung-Lu Partitioned Random Graph model.  

\begin{defn}Consider a vector $\mathbf{x}\in\mathbb{R}^{N\times 1}$.  Denote $|\mathbf{x}|$ as the l1 norm (or taxicab norm for the vector).  That is $|\mathbf{x}| = \sum_{i=1}^{N} |x_i|$.  Furthermore for a matrix $B\in\mathbb{R}^{N\times N}$, we can also define $|B|$ to be the l1 norm of the matrix where $|B| = \sum_{i,j} |b_{ij}| $.
\end{defn}

For simplicity we will consider the case where there are two communities (analogous results hold when there are more than two communities). Unsurprisingly, computing the number of paths and cycles becomes much more challenging when we  incorporate partitions (communities) into our 
random graph model.  We therefore introduce the following lemma that will facilitate the computation of otherwise unweildy expressions.

\begin{lem} \label{lem:generalized}
Consider the 2-Partitioned Random Graph Model.
Denote the number of paths from node $i_0$ to node $i_r$ of length $r$ as $P_{r}[i_0\rightarrow i_r]$.
Define \[   p_{ij}(x,y) = \left\{
\begin{array}{ll}
      p_{ij} & \text{if}\hspace{3pt}  G(i) = x \hspace{3pt} \text{and} \hspace{3pt} G(j) = y \\
      0 & \text{otherwise} \\
\end{array} 
\right. \]  and let $$\mathbf{p} = [p_{i_{0}i_{1}}(1,1),p_{i_{0}i_{1}}(2,1),p_{i_{0}i_{1}}(1,2),p_{i_{0}i_{1}}(2,2)]^{T}.  $$
Then,
\begin{equation} 
E(P_{r}[i_0\rightarrow i_r])= |\sum_{i_1,...,i_{r-1}} [\Pi_{k=1}^{r-1}\mathbf{A}(i_k,i_{k+1},\mathbf{i}_{k})]\mathbf{p}|
\end{equation}

where $| \cdot |$ denotes the taxicab norm, $\mathbf{i}_{k} = [ (i_1,i_2),(i_2,i_3),...,(i_{k-1},i_{k}) ] $

and if $(i,j)\notin \mathbf{i}$ then,

$$ \mathbf{A}(i,j,\mathbf{i}) = \begin{psmallmatrix}
p_{ij}(1,1) & p_{ij}(1,1) & 0 & 0 \\
0 & 0 & p_{ij}(2,1) & p_{ij}(2,1)  \\
p_{ij}(1,2) & p_{ij}(1,2)& 0 & 0 \\
0 & 0 & p_{ij}(2,2) & p_{ij}(2,2)
\end{psmallmatrix},$$

else if $(i,j) \in \mathbf{i}$, 

$$\mathbf{A}(i,j,\mathbf{i}) = \mathbf{G}_{ij} = \begin{psmallmatrix}
G_{ij}(1,1) & G_{ij}(1,1) & 0 & 0 \\
0 & 0 & G_{ij}(2,1) & G_{ij}(2,1)  \\
G_{ij}(1,2) & G_{ij}(1,2)& 0 & 0 \\
0 & 0 & G_{ij}(2,2) & G_{ij}(2,2)
\end{psmallmatrix},$$

where

\[  G_{ij}(c,d) = \left\{
\begin{array}{ll}
      1 & \text{if}\hspace{3pt}  G(i) = c \hspace{3pt} \text{and} \hspace{3pt} G(j) = d \\
      0 & \text{otherwise} \\
\end{array}\right. \]

\end{lem}

\noindent \textbf{Remark:} We can also express the number of paths as the norm of a matrix (as opposed to a vector) by noting that,
\begin{equation} 
E(P_{r}[i_0\rightarrow i_r])= |\frac{1}{2}\sum_{i_1,...,i_{r-1}} [\Pi_{k=0}^{r-1}\mathbf{A}(i_k,i_{k+1},\mathbf{i}_{k})]|,
\end{equation}
where $\mathbf{i}_{0}=\emptyset$.
\begin{proof}
We proceed by induction starting with a base case $r=2$, to compute the probability that a path $(i_{0},i_{1},i_{2})$, exists where $(i_{0},i_{1})\neq (i_{1},i_{2})$.  

So consider the matrix vector product,
$$\begin{pmatrix}
p_{i_1i_2}(1,1) & p_{i_1i_2}(1,1) & 0 & 0 \\
0 & 0 & p_{i_1i_2}(2,1) & p_{i_1i_2}(2,1)  \\
p_{i_1i_2}(1,2) & p_{i_1i_2}(1,2)& 0 & 0 \\
0 & 0 & p_{i_1i_2}(2,2) & p_{i_1i_2}(2,2)
\end{pmatrix} \begin{pmatrix} p_{i_{0}i_{1}}(1,1)\\p_{i_{0}i_{1}}(2,1)\\p_{i_{0}i_{1}}(1,2)\\p_{i_{0}i_{1}}(2,2)\end{pmatrix}=$$

$$\begin{pmatrix}
 p_{i_{0}i_{1}}(1,1)p_{i_1i_2}(1,1)+p_{i_{0}i_{1}}(2,1)p_{i_1i_2}(1,1)\\p_{i_{0}i_{1}}(1,2)p_{i_1i_2}(2,1)+p_{i_{0}i_{1}}(2,2)p_{i_1i_2}(2,1)\\p_{i_{0}i_{1}}(1,1)p_{i_1i_2}(1,2)+p_{i_{0}i_{1}}(2,1)p_{i_1i_2}(1,2)\\p_{i_{0}i_{1}}(1,2)p_{i_1i_2}(2,2)+p_{i_{0}i_{1}}(2,2)p_{i_1i_2}(2,2)\end{pmatrix}.$$

From the definition of $p_{ij}(x,y)$, the first entry in the output vector equals the probability that the path $P = (i_{0},i_{1},i_{2})$ exists if  $i_{1}$ is in group $1$ and $i_{2}$ is in group $2$ (and $0$ otherwise).  Similarly, the second entry equals the probability that $P$ exists if $i_{1}$ is in group $2$ and $i_{2}$ is group $1$.  The third entry specifies the probability $P$ exists if $i_{1}$ is in group $1$ and $i_{2}$ is in group $2$.  And finally, the fourth entry specifies the probability $P$ exists if $i_{1}$ and $i_{2}$ are both in group $2$.  In the case $(i_{0},i_{1}) = (i_{1},i_{2})$, then since we already accounted for the probability that the edge $(i_{1},i_{2})$ exists, we instead multiply the vector $\mathbf{p}$ by the matrix $\mathbf{G}_{i_{1}i_{2}}$, as defined in the statement of Lemma \ref{lem:generalized}.  Consequently, since for every path $P$, there is precisely one non-zero entry in the output vector that equals the probability that path $P$ exists, taking the taxicab norm of the sum of such vectors (where we sum over all possible choces of paths that start with $i_{0}$ and end with $i_{2}$) will be the expected number of paths from $i_{0}$ to $i_{2}$. 
 \newline\newline

\textbf{Inductive Step}: 
Suppose we are given a vector with the probability of the existence of a path of length $k$ (consisting of nodes $i_{0},...,i_{k}$) where each component in the vector equals the probability that the path exists if $i_{k}$ belongs to group $y$ and $i_{k-1}$ belongs to group $x$ and $0$ otherwise.  We  denote this quantity as $p_{k}(x,y)$.
Now to compute the probability of the existence of a path of length $k+1$, if $(i_k,i_{k+1})$ is not a repeating edge, we then consider

$$\begin{pmatrix}
p_{i_k i_{k+1}}(1,1) & p_{i_k i_{k+1}}(1,1) & 0 & 0 \\
0 & 0 & p_{i_k i_{k+1}}(2,1) & p_{i_k i_{k+1}}(2,1)  \\
p_{i_k i_{k+1}}(1,2) & p_{i_k i_{k+1}}(1,2)& 0 & 0 \\
0 & 0 & p_{i_k i_{k+1}}(2,2) & p_{i_k i_{k+1}}(2,2)
\end{pmatrix} \begin{pmatrix} p_{k}(1,1)\\p_{k}(2,1)\\p_{k}(1,2)\\p_{k}(2,2)\end{pmatrix}.$$

The output of this matrix vector product will yield a vector, where there will be precisely one non-zero entry equal to the probability the path exists.

Alternatively, if the edge $(i_k,i_{k+1})$ has been already visited earlier in the path, since we already accounted for the probability that the edge exists, we multiply the vector by the matrix $\mathbf{G}_{i_{k}i_{k+1}}$. 

By taking the taxicab norm of the sum of vectors, where each vector has precisely one non-zero entry that equals the probability that the path $(i_{0},...,i_{k+1})$ exists, will yield the expected number of paths from $i_{0}$ to $i_{k+1}$.

\end{proof}
To bound the expected number of paths of length $r$ using Lemma \ref{lem:generalized}, we will want to express the bounds of the norm of a matrix vector product in terms of the dominating eigenvalue of the matrix. Hence, we have the following result proved in the appendix.

\begin{cor} \label{cor:matrixmultiply} Let $B$ be an (entry-wise)  non-negative matrix, $B^{2}$ be an entry-wise positive matrix and let $x$ be the eigenvector corresponding to the dominating eigenvalue.  Furthermore denote $b_{ij}^{(m)}$ as the i,jth entry of $B^{m}$.  Assign $c_{max}$ to be the maximum row sum of $B^{2}\in\mathbb{R}^{n\times n}$ and suppose every entry is at least equal to 1, (hence $r_{max}\geq n$). Then

$$[\sum_{j=1}^{n}b_{ij}^{(m)}]^{\frac{1}{m}}\leq {c_{max}}^{\frac{1}{m}}\rho(B)$$
and $$(c_{max})^{-\frac{1}{m}}\rho(B)\leq [\sum_{j=1}^{n}b_{ij}^{(m)}]^{\frac{1}{m}}.$$
\end{cor}

We can now prove our desired result regarding the expected number of paths of length $r$.

\begin{thm} \label{thm:upperpar}
Consider a realization of a graph in the $2$-Partitioned Chung-Lu random graph model with prescribed expected row and column sums, $\mathbf{b}^{(x,y)},\mathbf{a}^{(x,y)}$ for each of the submatrices as illustrated in Figure \ref{fig:partition}.  Suppose that $S_{11},S_{12},S_{21}$ and $S_{22}$, the expected number of edges in each of the submatrices, are bounded below by $1$. Denote $P_{r}$ as the number of paths of length $r$.  Define $$\mathbf{P} = \begin{pmatrix}\frac{\mathbf{a}^{(1,1)}\cdot\mathbf{b}^{(1,1)}}{S_{11}} &\frac{\mathbf{a}^{(2,1)}\cdot\mathbf{b}^{(1,1)}}{S_{21}} &0 & 0 \\ 0 & 0 & \frac{\mathbf{a}^{(1,2)}\cdot\mathbf{b}^{(2,1)}}{S_{12}} & \frac{\mathbf{a}^{(2,2)}\cdot\mathbf{b}^{(2,1)}}{S_{22}} \\ \frac{\mathbf{a}^{(1,1)}\cdot\mathbf{b}^{(1,2)}}{S_{11}} & \frac{\mathbf{a}^{(2,1)}\cdot\mathbf{b}^{(1,2)}}{S_{21}} & 0 & 0 \\ 0 & 0 & \frac{\mathbf{a}^{(1,2)}\cdot\mathbf{b}^{(2,2)}}{S_{12}}&\frac{\mathbf{a}^{(2,2)}\cdot\mathbf{b}^{(2,2)}}{S_{22}} \end{pmatrix},$$

and assume that all of the $8$ entries in $\mathbf{P}$ that depend on the expected row and column sums of our submatrices are positive, so that $\mathbf{P}^{2}$ is entrywise positive.   For fixed $x$ and $y$, denote the maximum of the vectors $\mathbf{a}^{(x,y)}, \mathbf{b}^{(x,y)}$ as $a_{max}^{(x,y)}$ and $b_{max}^{(x,y)}$ respectively. Furthermore suppose for all choices of $m$ and $i_{1},...,i_{m+1}$, that there exists an $\alpha$ such that  \begin{equation} \label{eq:pmaxmatrix} \begin{psmallmatrix}
b_{max}^{(1,1)} & b_{max}^{(1,1)} & 0 & 0 \\
0 & 0 & b_{max}^{(2,1) }& b_{max}^{(2,1)}  \\
b_{max}^{(1,2)} & b_{max}^{(1,2)}& 0 & 0 \\
0 & 0 & b_{max}^{(2,2)} & b_{max}^{(2,2)}\end{psmallmatrix}
\Pi_{k=1}^{m}\mathbf{G}_{i_{k}i_{k+1}}
\begin{psmallmatrix}
a_{max}^{(1,1)}/S_{11} & 0& 0 & 0 \\
0 & a_{max}^{(2,1)}/S_{21} & 0 & 0 \\
0 & 0& a_{max}^{(1,2)}/S_{12} & 0 \\
0 & 0 & 0 & a_{max}^{(2,2)}/S_{22}\end{psmallmatrix}\leq \alpha \mathbf{P},
\end{equation}

where $\mathbf{G}_{i_{k}i_{k+1}}$ is defined in Lemma \ref{lem:generalized} and the inequality holds entry-wise.  Also suppose that $\rho(\mathbf{P})>2$.

Then $$E(P_{r}) \leq 8Sc_{max}\rho(\mathbf{P})^{r-1}exp(\frac{r^{2}\alpha\rho(\mathbf{P})^{-1}}{1-r\rho(\mathbf{P})^{-1}}), $$

where $c_{max}$ is the maximum column sum of $\mathbf{P}^{2}$.

\end{thm}

\textbf{Remark}.  While condition (\ref{eq:pmaxmatrix}) at first look may appear like a difficult condition to satisfy, this in fact is not so.
Upon careful observation of Lemma \ref{lem:generalized}, $\mathbf{G}_{ij}$ consists of two columns that are from the standard unit basis and two columns that are zero.  Consequently, if $\rho(\mathbf{P})\rightarrow \infty$, and since $\mathbf{G}_{ij}$ has the property that $|\mathbf{G}_{ij}\mathbf{v}| \leq |\mathbf{v}|$, we can often satisfy condition (\ref{eq:pmaxmatrix}) with ease.  It is worth mentioning that even if we cannot satisfy (\ref{eq:pmaxmatrix}), we could still prove a useful generalization of Theorem \ref{thm:upperpar} by requiring that for each $m$ we can find an $\alpha_{m}$ such that
\begin{equation} \begin{psmallmatrix}
b_{max}^{(1,1)} & b_{max}^{(1,1)} & 0 & 0 \\
0 & 0 & b_{max}^{(2,1) }& b_{max}^{(2,1)}  \\
b_{max}^{(1,2)} & b_{max}^{(1,2)}& 0 & 0 \\
0 & 0 & b_{max}^{(2,2)} & b_{max}^{(2,2)}\end{psmallmatrix}
\Pi_{k=1}^{m}\mathbf{G}_{i_{k}i_{k+1}}
\begin{psmallmatrix}
a_{max}^{(1,1)}/S_{11} & 0& 0 & 0 \\
0 & a_{max}^{(2,1)}/S_{21} & 0 & 0 \\
0 & 0& a_{max}^{(1,2)}/S_{12} & 0 \\
0 & 0 & 0 & a_{max}^{(2,2)}/S_{22}\end{psmallmatrix}\leq \alpha_{m} \mathbf{P}^{m}.
\end{equation}

Alternatively, we can also satisfy (\ref{eq:pmaxmatrix}), if the product of the norms of the left and right matrices on the left hand side of (\ref{eq:pmaxmatrix}), are sufficiently small, analogous to the case where $p_{max} \rightarrow 0$.  With this in mind, we now provide the proof, which is similar to Theorem \ref{thm:upper}.

\begin{proof}
Recall that $A = \begin{pmatrix}A_{11} & A_{12} \\ A_{21} & A_{22}   \end{pmatrix}$.  We define $a_{i_k}^{(x,y)}$ to be $0$ if $i_k$ does not belong to group $y$.  If $i_k$ does belong to group $y$, then $a_{i_k}^{(x,y)}$ will be the expected column sum corresponding to node $i_k$ in the submatrix $A_{xy}$.  Analogously, we define  $b_{i_k}^{(x,y)}$ to be $0$ if $i_k$ does not belong to group $x$.  If $i_k$ does belong to group $x$, then $b_{i_k}^{(x,y)}$ will be the expected row sum corresponding to node $i_k$ in the submatrix $A_{xy}$. Let $S_{xy} = \sum_{i_k = 1}^{N}b_{i_k}^{(x,y)}$ be the expected sum of the entries in the submatrix  $A_{xy}$.
Consequently by the definition of the Partitioned Chung-Lu random graph model, we have that 

\begin{align}\label{eq:CLumatrix}
\begin{psmallmatrix} 
p_{i_k i_{k+1}}(1,1) & p_{i_k i_{k+1}}(1,1) & 0 & 0 \\
0 & 0 & p_{i_k i_{k+1}}(2,1) & p_{i_k i_{k+1}}(2,1)  \\
p_{i_k i_{k+1}}(1,2) & p_{i_k i_{k+1}}(1,2)& 0 & 0 \\
0 & 0 & p_{i_k i_{k+1}}(2,2) & p_{i_k i_{k+1}}(2,2)
\end{psmallmatrix} = \mspace{120mu} \notag\\
\begin{psmallmatrix} a_{i_{k+1}}^{(1,1)}/S_{11} & 0& 0 & 0 \\
0 & a_{i_{k+1}}^{(2,1)}/S_{21} & 0 & 0 \\
0 & 0& a_{i_{k+1}}^{(1,2)}/S_{12} & 0 \\
0 & 0 & 0 & a_{i_{k+1}}^{(2,2)}/S_{22}
\end{psmallmatrix} \begin{psmallmatrix}
b_{i_{k}}^{(1,1)} & b_{i_{k}}^{(1,1)} & 0 & 0 \\
0 & 0 & b_{i_{k}}^{(2,1) }& b_{i_{k}}^{(2,1)}  \\
b_{i_{k}}^{(1,2)} & b_{i_{k}}^{(1,2)}& 0 & 0 \\
0 & 0 & b_{i_{k}}^{(2,2)} & b_{i_{k}}^{(2,2)}
\end{psmallmatrix} .
\end{align}

Because this is rather unweildy, we will denote
\begin{equation}\label{eq:notation1}\mathbf{A}_{i_{k+1}} = \begin{psmallmatrix} a_{i_{k+1}}^{(1,1)}/S_{11} & 0& 0 & 0 \\
0 & a_{i_{k+1}}^{(2,1)}/S_{21} & 0 & 0 \\
0 & 0& a_{i_{k+1}}^{(1,2)}/S_{12} & 0 \\
0 & 0 & 0 & a_{i_{k+1}}^{(2,2)}/S_{22}
\end{psmallmatrix} and  
\end{equation}
\begin{equation} \label{eq:notation2}
\mathbf{B}_{i_k}=\begin{psmallmatrix}
b_{i_{k}}^{(1,1)} & b_{i_{k}}^{(1,1)} & 0 & 0 \\
0 & 0 & b_{i_{k}}^{(2,1) }& b_{i_{k}}^{(2,1)}  \\
b_{i_{k}}^{(1,2)} & b_{i_{k}}^{(1,2)}& 0 & 0 \\
0 & 0 & b_{i_{k}}^{(2,2)} & b_{i_{k}}^{(2,2)}
\end{psmallmatrix} .
\end{equation}

To derive the upperbound, we consider paths where the last edge is not a repeating edge.  Furthermore, since we are dealing with matrix multiplication and no longer have commutativity, it will be more helpful to denote the locations of the nodes in the path that are part of the new edge interior as opposed to the nodes themselves.   To construct an arbitrary path, we consider the set $\mathbf{L}_{N}$, which identifies the locations of the nodes in the new edge interior.  Furthermore, we also construct $\mathbf{L}_{R}$, which identifies both the locations of the nodes in repeating edge blocks and the function that assigns nodes in a repeating edge block to equal nodes in earlier positions in the path.   By considering all possible choices for $\mathbf{L}_{N}$, $\mathbf{L}_{R}$, summing over all possible node choices and invoking Lemma \ref{lem:generalized}, we have the following upperbound for the expected number of paths of length $r$, where the last edge is not a repeating edge,

\begin{equation}\label{eq:paths}   
E(P_{r}^{L})\leq  \sum_{\mathbf{L}_{R},\mathbf{L}_{N}}\sum_{\substack{n_j\in [1..N]\\ \forall j \in \mathbf{L}_{N}}}\sum_{n_0 = 1,n_r = 1}^{N}[\Pi_{t=1}^{r-1}\mathbf{A}(n_t,n_{t+1},\mathbf{n}_{t})]\mathbf{p}|.
\end{equation}

Once we fix a particular choice for $\mathbf{L}_{R}$ we know the positions of the first and last nodes in a repeating edge block. Define $$\mathbf{Z}_{n_{k}} = \begin{cases}
       \mathbf{B}_{n_{k}}\mathbf{A}_{n_{k}} &\quad\text{if}\hspace{3pt} n_{k} \in \mathbf{L}_{N}\\
       \alpha\mathbf{P} &\quad\text{if}\hspace{3pt} n_{k} \hspace{3pt}\text{is in the beginning of a repeating edge block}  \\
       \mathbf{I} &\quad\text{otherwise.} \\ 
     \end{cases}, $$ where $\mathbf{I}$ is the identity matrix and let \begin{equation} \mathbf{b}_{n_{0}} = \begin{pmatrix} b_{n_{0}}^{(1,1)} & b_{n_{0}}^{(2,1)}& b_{n_{0}}^{(1,2)} &b_{n_{0}}^{(2,2)}\end{pmatrix}^{T}.
     \end{equation}
  With these definitions, we can simplify (\ref{eq:paths})  by invoking (\ref{eq:pmaxmatrix}).(\ref{eq:notation1}),(\ref{eq:notation2}) to get that,

\begin{equation}\label{eq:paths2}   
E(P_{r}^{L})\leq  \sum_{\mathbf{L}_{R},\mathbf{L}_{N}}\sum_{\substack{n_j\in [1..N]\\ \forall j \in \mathbf{L}_{N}}}\sum_{n_0 = 1,n_r = 1}^{N}[\mathbf{A}_{n_{r}}\Pi_{t=1}^{r-1}\mathbf{Z}_{n_{t}}\mathbf{b}_{n_{0}}|,
\end{equation}

Recall from (\ref{eq:notation1}) and (\ref{eq:notation2}) that
\begin{equation}\mathbf{B}_{i_k}\mathbf{A}_{i_k} = \begin{psmallmatrix}
b_{i_{k}}^{(1,1)}a_{i_{k}}^{(1,1)}/S_{11} & b_{i_{k}}^{(1,1)} a_{i_{k}}^{(2,1)}/S_{21} & 0 & 0 \\
0 & 0 & b_{i_{k}}^{(2,1) }a_{i_{k}}^{(1,2)}/S_{12}& b_{i_{k}}^{(2,1)}a_{i_{k}}^{(2,2)}/S_{22}  \\
b_{i_{k}}^{(1,2)}a_{i_{k}}^{(1,1)}/S_{11} & b_{i_{k}}^{(1,2)} a_{i_{k}}^{(2,1)}/S_{21}& 0 & 0 \\
0 & 0 & b_{i_{k}}^{(2,2)}a_{i_{k}}^{(1,2)}/S_{12} & b_{i_{k}}^{(2,2)}a_{i_{k}}^{(2,2)}/S_{22}\end{psmallmatrix}.
\end{equation}

And from the statement of this Theorem, Theorem \ref{thm:upperpar}, recall the definition of $\mathbf{P}$.
It follows that by summing over all possible choices of nodes for $i_k$ that, 

\begin{equation}\mathbf{P} = \sum_{i_k=1}^{N}\mathbf{B}_{i_k}\mathbf{A}_{i_k}.
\end{equation}

So from (\ref{eq:paths2}) by summing over all possible nodes in the new edge interior, we get that,

\begin{equation}\label{eq:paths3}   
E(P_{r}^{L})\leq  \sum_{\mathbf{L}_{R},\mathbf{L}_{N}}\sum_{n_0 = 1,n_r = 1}^{N}[\mathbf{A}_{n_{r}}\mathbf{P}^{|\mathbf{L}_{N}|}(\alpha \mathbf{P})^{\sum_{i=1}^{r}k_{i}}\mathbf{b}_{n_{0}}|,
\end{equation} 
where for $i>0$, $k_{i}$ is the number of repeating edge blocks of length $i$, which is determined by $\mathbf{L}_{R}$.  

Now instead of summing over all of the possible locations of the new edges and repeating edges (and which edges they must equal to), we sum over the prescribed number of repeating edge blocks of various lengths, consider all possible arrangements for the positions of the repeating edge blocks and multiply this quantity by the number of ways for filling in the repeating edge blocks.

So using the fact that $|\mathbf{L}_{N}| = r - 1 - \sum_{i=1}^{r-2}(i+1)k_{i}$, we conclude that, 

\begin{equation}\label{eq:paths4}   
E(P_{r}^{L})\leq  \sum_{\substack{k_{0}+\sum_{i=1}^{r-2}ik_{i}\\ \forall i, k_{i}\in [0..r] }}\sum_{n_0 = 1,n_r = 1}^{N}\binom{\sum k_{i}}{k_{0},...,k_{r-2}}r^{\sum_{i\geq 1}ik_{i}}[\mathbf{A}_{n_{r}}\mathbf{P}^{r - 1 - \sum_{i=1}^{r}(i+1)k_{i}}(\alpha \mathbf{P})^{\sum_{i=1}^{r}k_{i}}\mathbf{b}_{n_{0}}|.
\end{equation} 

Summing over all possible choices for the first and last nodes yield,

\begin{equation}\label{eq:paths5}   
E(P_{r}^{L})\leq  |\sum_{\substack{k_{0}+\sum_{i=1}^{r-2}ik_{i}\\ \forall i, k_{i}\in [0..r] }}\binom{\sum k_{i}}{k_{0},...,k_{r-2}}r^{\sum_{i\geq 1}ik_{i}}\mathbf{I}\cdot \mathbf{P}^{r - 1 - \sum_{i=1}^{r-2}ik_{i}}(\alpha)^{\sum_{i=1}^{r}k_{i}}S\mathbf{1}|,
\end{equation} 
where  $\mathbf{1}$ is a vector of one's and $S = S_{11}+S_{12}+S_{21}+S_{22}$.
, where $S_{ij}$ denotes the expected sum of entries in the submatrix $A_{ij}$.  Now by definition of $\mathbf{P}$, $\mathbf{P}^{2}$ is an entry-wise positive matrix where each entry is bounded below by $1$.  Define $c_{max}$ to be the maximum column sum of $\mathbf{P}^{2}$.  It follows by Corollary \ref{cor:matrixmultiply}, that each column sum of  $\mathbf{P}^{r-1}$ is bounded above by $c_{max}\rho(\mathbf{P}^{r-1})$.  Hence we conclude that 

\begin{equation}\label{eq:paths6}   
E(P_{r}^{L})\leq  \sum_{\substack{k_{0}+\sum_{i=1}^{r-2}ik_{i}\\ \forall i, k_{i}\in [0..r] }}\binom{\sum k_{i}}{k_{0},...,k_{r-2}}r^{\sum_{i\geq 1}ik_{i}}4c_{max}\rho(\mathbf{P})^{r - 1 - \sum_{i=1}^{r-2}ik_{i}}(\alpha)^{\sum_{i=1}^{r}k_{i}}S.
\end{equation} 

Since we have replaced the norm of a matrix product with scalar multiplication, we can directly appeal to the strategy we used in Theorem \ref{thm:upper} to get an upperbound on (\ref{eq:paths6}), where we will use the facts that  $E(P_{r}) \leq \sum_{i=1}^{r}E(P_{i}^{L})$, $\rho(\mathbf{P}) > 2$ and invoke Lemma \ref{lem:expineq}.
\end{proof}

We now have the desired concentration result.

\begin{cor}\label{cor:markovpar}
Consider a realization of a graph $A$ in the $2$-Partitioned Chung-Lu random graph model with prescribed expected row and column sums, $\mathbf{b}^{(x,y)},\mathbf{a}^{(x,y)}$ for each of the submatrices as illustrated in Figure \ref{fig:partition}.  Suppose that $S_{11},S_{12},S_{21}$ and $S_{22}$ are all positive.  Denote $P_{r}$ as the number of paths of length $r$.
Define $$\mathbf{P} = \begin{pmatrix}\frac{\mathbf{a}^{(1,1)}\cdot\mathbf{b}^{(1,1)}}{S_{11}} &\frac{\mathbf{a}^{(2,1)}\cdot\mathbf{b}^{(1,1)}}{S_{21}} &0 & 0 \\ 0 & 0 & \frac{\mathbf{a}^{(1,2)}\cdot\mathbf{b}^{(2,1)}}{S_{12}} & \frac{\mathbf{a}^{(2,2)}\cdot\mathbf{b}^{(2,1)}}{S_{22}} \\ \frac{\mathbf{a}^{(1,1)}\cdot\mathbf{b}^{(1,2)}}{S_{11}} & \frac{\mathbf{a}^{(2,1)}\cdot\mathbf{b}^{(1,2)}}{S_{21}} & 0 & 0 \\ 0 & 0 & \frac{\mathbf{a}^{(1,2)}\cdot\mathbf{b}^{(2,2)}}{S_{12}}&\frac{\mathbf{a}^{(2,2)}\cdot\mathbf{b}^{(2,2)}}{S_{22}} \end{pmatrix}.$$

Furthermore suppose that he eight non-zero entries in$\mathbf{P}$ are bounded below by $1$ and that condition (\ref{eq:pmaxmatrix}) from Theorem \ref{thm:upperpar} holds.

Then for every $\epsilon > 0$, there exists $\delta_1$,$\delta_2$ such that if there exists $r \in \mathbb{N}$, where $\frac{\log N}{r}<\delta_1$ and $\alpha r^{2}\rho(\mathbf{P})^{-1}<\delta_2$, then

$$Pr(\rho(A)\leq (1+\epsilon)\rho(\mathbf{P}))\geq 1 - \epsilon. $$
\end{cor}

\begin{proof}
The proof is analogous to Theorem \ref{thm:Markov}.  Invoking Theorem \ref{thm:upperpar}, Lemma \ref{lem:theory} and Markov's Inequality prove the result.
\end{proof}

Counting paths using matrix products presents a major challenge for generalizing our results from the Chung-Lu model to the Partitioned Chung-Lu model. More specifically, when counting cycles of length $2$ in the Chung-Lu model, we used commutativity to argue that $\sum_{i,j} p_{ji}p_{ij}=\sum_{i,j} \frac{b_{j}a_{i}}{S}\frac{b_{i}a_{j}}{S} = \sum_{i} \frac{a_{i}b_{i}}{S}\sum_{j}\frac{a_{j}b_{j}}{S}=(\frac{\mathbf{a}\cdot\mathbf{b}}{S})^{2}$.  Unfortunately, we cannot assume that our matrices will commute.  Consequently, we seek a lower bound that retains the desired commutativity property. Define the matrix 
\begin{equation} \label{eq:trCr} \mathbf{E}(i,j,G) = \begin{psmallmatrix}
e_{ij}(1,1) & e_{ij}(1,1) & 0 & 0 \\
0 & 0 & e_{ij}(2,1) & e_{ij}(2,1)  \\
e_{ij}(1,2) & e_{ij}(1,2)& 0 & 0 \\
0 & 0 & e_{ij}(2,2) & e_{ij}(2,2)
\end{psmallmatrix}, 
\end{equation}

where $e_{ij}(m,n)$ is $1$ if in the graph $G$ there is an edge from node $i$ to node $j$, node $i$ belongs to group $m$ and node $j$ belongs to group $n$; if that is not the case then  $e_{ij}(m,n)=0$. \textbf{It follows that for a given graph $G$ the number of paths of length $r-1$ is precisely}     $|\sum_{i_1,...,i_{r}} \frac{1}{2}[\Pi_{k=1}^{r-1}\mathbf{E}(i_k,i_{k+1},G)]|.$  (Compare this with the remark after Lemma \ref{lem:generalized}.)  We will want to consider the random variable  

\begin{equation} \label{eq:trCr2}
trace(C_r):= \newline trace(\sum_{i_1=i_{r+1},i_2...,i_{r}} \frac{1}{2}\Pi_{k=1}^{r}\mathbf{E}(i_k,i_{k+1},G))
\end{equation} and show that with high probability the $trace(C_r)$ is heavily concentrated around its mean.  As suggested earlier, we can easily prove interesting concentration regarding $trace(C_{r})$ as taking the trace of a matrix product has quasi-commutatitive properties.  We stress that $trace(C_{r})$ is a lower bound for the number of cycles of length $r$ and hence a lowerbound for the rth power of the spectral radius.  The following lemma relates the $trace(C_r)$ to the $trace(\mathbf{P}^{r})$.

\begin{lem}\label{lem:explowermtrx}
	Recall that we define the random variable $trace(C_r):= \newline trace(\sum_{i_1=i_{r+1},i_2...,i_{r}} [\frac{1}{2}\Pi_{k=1}^{r}\mathbf{E}(i_k,i_{k+1},G))$, which represents the number of a particular subset of cycles of length $r$ in our graph.  Then it follows that,
	
	$$\frac{1}{2}trace(\mathbf{P}^{r})\leq E(trace(C_{r})).$$
\end{lem} 

\begin{proof}
	Firstly, by definition,
	\begin{equation} E(trace(C_r)):= E(trace(\sum_{\substack{i_1=1,i_2=1...,i_{r}=1\\ i_1 = i_{r+1}}}^{N} [\frac{1}{2}\Pi_{k=1}^{r}\mathbf{E}(i_k,i_{k+1},G)]) =
	\end{equation}
	\begin{equation}\label{eq:edgeprob}
	 \sum_{\substack{i_1=1,i_2=1...,i_{r}=1\\ i_1 = i_{r+1}}}^{N}trace([\frac{1}{2}\Pi_{k=1}^{r}\mathbf{A}(i_k,i_{k+1},\mathbf{i}_{k})]),
	 \end{equation}
	 where (\ref{eq:edgeprob}) is precisely the sum of the probabilities that each cycle that contributes to the $trace(C_{r})$ exists. 
	 We can then entrywise bound the matrix product, $\frac{1}{2}\Pi_{k=1}^{r}\mathbf{A}(i_k,i_{k+1},\mathbf{i}_{k})$ below by $\Pi_{k=1}^{r}\mathbf{A}_{i_{k+1}}\mathbf{B}_{i_{k}}$, as the latter matrix product computes the product of the probabilities that each edge exists, even if we have already visited a particular edge earlier in the path.  We then have that,
	\begin{equation}
	E(trace(C_r)) \geq \sum_{\substack{i_1=1,i_2=1...,i_{r}=1\\ i_1 = i_{r+1}}}^{N}\frac{1}{2} trace(\Pi_{k=1}^{r}\mathbf{A}_{i_{k+1}}\mathbf{B}_{i_{k}})
	 = 
	 \end{equation}
	 \begin{equation} \label{eq:cyclemtrx}
	\sum_{\substack{i_1,...,i_r}} \frac{1}{2}trace(\mathbf{A}_{i_1}(\Pi_{k=2}^{r}\mathbf{B}_{i_{k}}\mathbf{A}_{i_{k}})\mathbf{B}_{i_1})
	 =  \sum_{i_1,...,i_r}\frac{1}{2}trace((\Pi_{k=2}^{r}\mathbf{B}_{i_{k}}\mathbf{A}_{i_{k}})\mathbf{B}_{i_1}\mathbf{A}_{i_1}),
	\end{equation}
	
	,where in the last equality we used the commutativity property that $trace(AB) = trace(BA)$.  It then follows that (\ref{eq:cyclemtrx}) equals
	
	\begin{equation}
	\sum_{i_1,i_2,...,i_{r}} \frac{1}{2}trace(\Pi_{k=1}^{r}\mathbf{B}_{i_{k}}\mathbf{A}_{i_{k}})=\frac{1}{2}trace(\sum_{i_1,i_2,...,i_{r}}\Pi_{k=1}^{r}\mathbf{B}_{i_{k}}\mathbf{A}_{i_{k}})=\frac{1}{2}trace(\mathbf{P}^{r}).
	\end{equation}
	
\end{proof}

In order to show that the standard deviation is much smaller than the mean, we want to show that the contribution for pairs of paths with many repeating edge blocks decreases exponentially as we increase the number and size of the repeating edge blocks.  More precisely, the presence of repeating edge blocks results in a smaller power $r$, of the matrix product $\mathbf{P}^{r}$.  But showing that $|\mathbf{P}^{r}|$ increases at an exponential rate in terms of $r$ is a non-trivial problem as we do not want to make any assumptions about the eigenbasis of $\mathbf{P}$.  

We also note that the expected value of the $trace(C_{r})$ in Lemma  \ref{lem:explowermtrx} is stated in terms of $trace(\mathbf{P}^{r})$.  Since we want to show that asymptotically with high probability $\rho(\mathbf{P}) \leq trace(C_{r})^{\frac{1}{r}}$, we would like to find a lowerbound for $trace(\mathbf{P}^{r})$ in terms of the spectral radius.   The following corollary, proved in the appendix, addresses both of these issues.

\begin{cor}\label{cor:matrixbound}
	Suppose each entry of a matrix $B^{w}\in\mathbb{R}^{k\times k}$ is bounded below by $1$.  Then for all $u \in \mathbb{N}\cup 0$ and all $v\in\mathbb{N}$,
	
	$$\rho(B)^{v}B^{u}\leq B^{u+v+2w}.$$
	
	Furthermore, $$k\rho(B)^{v}\leq trace(B^{v+2w}).$$
\end{cor}

Before we proceed with our upperbound on the variance, we will need one more inequality.

\begin{lem}\label{lem:find}
Suppose that $B^{w}\in\mathbb{R}^{k\times k}$ is entrywise bounded below by $1$, where $w$ is a positive integer and that $B$ is an entrywise non-negative matrix.  Then for every integer $r$ greater than $w$, there exists a non-negative integer $m$ that satisfies the inequality, $0 \leq m \leq w-1$, such that

\begin{equation}\label{eq:find} trace(B^{q}) \leq trace(B^{r-m}) 
\end{equation}

for all non-negative integers $q$ that satisfy the inequality $q \leq r - m$.  

\end{lem}

From Lemma \ref{lem:find}, it is easy to identify values of $r$ that satisfy the inequality $trace(\mathbf{P}^{r-k})\leq trace(\mathbf{P}^{r})$ for all non-negative integers $k$.  In the case where $\mathbf{P}^{2}$ is entry-wise bounded below by $1$, for any choice of $r \geq 3$, it follows that either $trace(\mathbf{P}^{r-k})\leq trace(\mathbf{P}^{r})$ or  $trace(\mathbf{P}^{r-1-k})\leq trace(\mathbf{P}^{r-1})$ for all $k\in[0..r-1]$.  Hence we can easily find (large) values of $r$ that satisfy (\ref{eq:find}).  We are now ready to present the following result bounding the variance for $trace(C_{r})$.

\begin{thm} \label{thm:generallower}
	Recall that we define the random variable $trace(C_r):= \newline trace(\sum_{i_1=i_{r+1},i_2...,i_{r}} [\frac{1}{2}\Pi_{k=1}^{r}\mathbf{E}(i_k,i_{k+1},G))$, which represents the number of a particular subset of cycles of length $r$ in our graph.  Assume that (\ref{eq:pmaxmatrix}) holds and that $\mathbf{P}$ has eight entries that are bounded below by $1$.  Furthermore, we consider $r$ according to Lemma \ref{lem:find} such that $trace(\mathbf{P}^{r-k})\leq trace(\mathbf{P}^{r})$ for all non-negative integers $k \leq r$.  Then it follows that, 
	
	\begin{equation}
	var(trace(C_{r})) \leq E(trace(C_{r}))\cdot [1 + r^{r} + trace(\mathbf{P}^{r}) [exp(\frac{64\alpha r^{5}}{1-2r\rho(\mathbf{P})^{-1}})-1]].
	\end{equation}		
	where $\alpha$ is defined in (\ref{eq:pmaxmatrix}).
\end{thm}

\begin{proof}
The proof is in the same spirit as Theorem \ref{thm:cov}.  First we express each possible cycle that can contribute to $trace(C_{r})$ as an indicator random variable $\mathbf{1}_{y}$ and define the set $D(y)$ to include all of the indices of the indicator random variables that are dependent with $\mathbf{1}_{y}$, except for $y$. .
	
Consequently,

\begin{equation}
\sum_{y \neq z} Cov(\mathbf{1}_y,\mathbf{1}_z) \leq \sum_{\substack{y\\z\in D(y)}}E(\mathbf{1}_{y}\mathbf{1}_{z}) = \sum_{\substack{y\\z\in D(y)}} Pr(\mathbf{1}_y = 1)Pr(\mathbf{1}_z=1|\mathbf{1}_y = 1).
\end{equation}

Now we invoke Lemma \ref{lem:generalized} to bound the $Pr(\mathbf{1}_z=1|\mathbf{1}_y = 1)$.  This yields,

\begin{equation} \label{eq:covpar}
\sum_{\substack{y}} Pr(\mathbf{1}_y = 1)\sum_{\substack{\mathbf{L}_{R},\mathbf{L}_{N}\\ \mathbf{L}_{R} \neq \emptyset}}\sum_{\substack{n_{j}\in[1..N] \\ \forall j\in \mathbf{L}_{N}}}trace(\Pi_{t=0}^{r} \mathbf{A}(n_{t},n_{t+1},\mathbf{n}_{t})),
\end{equation}

where $n_{r+1} = n_{0}$,  $\mathbf{L}_{N}$ denotes the locations of the nodes in the new edge interior and $\mathbf{L}_{R}$ assigns nodes found in the repeating edge blocks to equal specified nodes that are part of a new edge. Note that $\mathbf{n}_{t}$ also includes edges from the cycle $y$.  The constraint $\mathbf{L}_{R}\neq \emptyset$ ensures that the cycles (y and z) are dependent.  We will consider two cases.

\textbf{Case 1:} We consider sets $\mathbf{L}_{R}$ such that $k_{r} = 1$ (which implies that $\mathbf{L}_{N} = \emptyset$).  

\begin{equation} \label{eq:covparCASE1}
\sum_{\substack{y}} Pr(\mathbf{1}_y = 1)\sum_{\substack{\mathbf{L}_{R}: k_{r}=1}}trace(\Pi_{t=0}^{r} \mathbf{A}(n_{t},n_{t+1},\mathbf{n}_{t})) \leq   \sum_{y}Pr(\mathbf{1}_{y} = 1)r^{r},
\end{equation}

as for each edge in a cycle $z\in D(y)$, there are at most $r$ choices to choose from.

\textbf{Case 2:} Suppose $k_{r} = 0$, it follows that $|\mathbf{L}_{N}| = r  - \sum_{i=1}^{r-1}(i+1)k_{i}$. (Here, since we are only considering cycles, the first node, which equals the last node can be part of the new edge interior.)

\textbf{Case 2a:} Suppose there is a node in the new edge interior, we can without loss of generality assume that the first node in the cycle $z$ is a node in the new edge interior.  Similar to Theorem \ref{thm:upperpar}, express the contribution of this case to (\ref{eq:covpar}) as,

\begin{equation}\label{eq:parcov2}   
  \sum_{y}Pr(\mathbf{1}_{y} = 1) \sum_{\substack{\mathbf{L}_{R},\mathbf{L}_{N}\\\mathbf{L}_{N}\neq \emptyset}}\sum_{\substack{n_{j}\in [1..N]\\ \forall j \in \mathbf{L}_{N}}}\sum_{n_0 = 1}^{N}trace(\mathbf{A}_{n_{0}}\Pi_{t=1}^{r}\mathbf{Z}_{n_{t}}\mathbf{B}_{n_{0}}),
\end{equation} 

where $$\mathbf{Z}_{n_{k}} = \begin{cases}
       \mathbf{B}_{n_{k}}\mathbf{A}_{n_{k}} &\quad\text{if}\hspace{3pt} n_{k} \in \mathbf{L}_{N}\\
       \alpha\mathbf{P} &\quad\text{if}\hspace{3pt} n_{k} \hspace{3pt}\text{is in the beginning of a repeating edge block}  \\
       \mathbf{I} &\quad\text{otherwise.} \\ 
     \end{cases} $$

Using the fact that the trace of a matrix product does not change under cyclic permutations, we get that, 

\begin{equation}\label{eq:parcov3}   
  \sum_{y}Pr(\mathbf{1}_{y} = 1) \sum_{\substack{\mathbf{L}_{R},\mathbf{L}_{N}\\\mathbf{L}_{N}\neq \emptyset}}\sum_{\substack{n_{j}\in [1..N]\\ \forall j \in \mathbf{L}_{N}}}\sum_{n_0 = 1}^{N}trace(\Pi_{t=1}^{r}\mathbf{Z}_{n_{t}}\mathbf{B}_{n_{0}}\mathbf{A}_{n_{0}}).
\end{equation} 

Summing over all possible node choices for nodes in the new edge interior gives us an upperbound of, 

\begin{equation}\label{eq:parcov4}   
   \sum_{y}Pr(\mathbf{1}_{y} = 1)\sum_{\substack{\mathbf{L}_{R},\mathbf{L}_{N}\\\mathbf{L}_{N}\neq \emptyset}}trace(\mathbf{P}^{r-\sum_{i=1}^{r-2}(i+1)k_{i}}(\alpha \mathbf{P})^{\sum_{i=1}^{r-2}k_{i}}).
\end{equation} 

\textbf{Case 2b:} If in fact there are no nodes in the interior of a new edge block, from (\ref{eq:parcov2}), we get that 
\begin{equation}\label{eq:parcov9}   
   \sum_{y}Pr(\mathbf{1}_{y} = 1)\sum_{\mathbf{L}_{R}}(2r)^{\sum ik_{i}}trace((\alpha \mathbf{P})^{\sum_{i=1}^{r-1}k_{i}}),
\end{equation} 
 But then we can still use the general expression found in (\ref{eq:parcov4}) ignoring the constraint that $\mathbf{L}_{N} \neq \emptyset$  as if $\mathbf{L}_{N} = \emptyset$, then $r-\sum_{i=1}^{r-1}(i+1)k_{i} = 0$.

So we can combine Cases 2a and 2b, where we re-express(\ref{eq:parcov4}) and (\ref{eq:parcov9}) by specifying the values of $k_{i}$ and considering all possible choices for each edge (or node) in a repeating edge block.   This gives us the upperbound,

\begin{equation}\label{eq:parcov41}   
   \sum_{y}Pr(\mathbf{1}_{y} = 1)\sum_{\substack{k_{0} + \sum_{i=1}^{r} ik_{i} = r\\ k_{0} < r}}\binom{\sum k_{i}}{k_{0},...,k_{r-1}}(2r)^{\sum ik_{i}}trace(\mathbf{P}^{r-\sum_{i=1}^{r-1}(i+1)k_{i}}(\alpha \mathbf{P})^{\sum_{i=1}^{r-2}k_{i}}),
\end{equation}

 where $k_{0} < r$ since we require that there cannot be $r$ edges in the cycle as the one of the edges must appear in the cycle $y$.   So by considering the contribution from Cases 1, 2a and 2b, we get the upperbound, 
 
 \begin{equation}\label{eq:goal01}
 \sum_{y \neq z} Cov(\mathbf{1}_y,\mathbf{1}_z) \leq 
   \sum_{y}Pr(\mathbf{1}_{y} = 1)[ r^{r} + \sum_{\substack{k_{0} + \sum_{i=1}^{r} ik_{i} = r\\ k_{0} < r}}\binom{\sum k_{i}}{k_{0},...,k_{r-1}}(2r)^{\sum ik_{i}}trace(\mathbf{P}^{r-\sum_{i=1}^{r-1}(i+1)k_{i}}(\alpha \mathbf{P})^{\sum_{i=1}^{r-2}k_{i}})]
 \end{equation}

We can then simplify (\ref{eq:goal01}) by invoking Corollary \ref{cor:matrixmultiply} to substitute the inequality, $tr(\mathbf{P}^{r-w-4})\leq \frac{tr(\mathbf{P}^{r})}{\rho(\mathbf{P})^{w}}$ and employing the assumption that $tr(\mathbf{P}^{r-w}) \leq tr(\mathbf{P}^{r})$ for all non-negative integers $w$.  Hence,

 \begin{equation}\label{eq:goal02}
  \sum_{y}Pr(\mathbf{1}_{y} = 1)[r^{r} + \sum_{\substack{k_{0} + \sum_{i=1}^{r} ik_{i} = r\\ k_{0}<r}}\binom{\sum k_{i}}{k_{0},...,k_{r-1}}\Pi_{i=1}^{r-1}\frac{\alpha^{k_{i}}(2r)^{ik_{i}}}{\rho(\mathbf{P}^{\max(ik_{i}-4,0)})}trace(\mathbf{P}^{r})]\leq 
 \end{equation}
 
  \begin{multline}\label{eq:goal04}
   \sum_{y}Pr(\mathbf{1}_{y} = 1)[r^{r} + \sum_{\substack{k_{0} + \sum_{i=1}^{r} ik_{i} = r\\ k_{0}<r}}\binom{\sum k_{i}}{k_{0},...,k_{r-1}}\Pi_{i=1}^{4}\alpha^{k_{i}}(2r)^{ik_{i}}\LargerCdot\\\Pi_{i=5}^{r-1}\frac{\alpha^{k_{i}}(2r)^{ik_{i}}}{\rho(\mathbf{P}^{(i-4)k_{i}})}trace(\mathbf{P}^{r})]\leq 
 \end{multline} 
 
   \begin{multline}\label{eq:goal05}
   \sum_{y}Pr(\mathbf{1}_{y} = 1)[r^{r} + [\sum_{\substack{k_{0} + \sum_{i=1}^{r} ik_{i} = r\\}}\binom{\sum k_{i}}{k_{0},...,k_{r-1}}\Pi_{i=1}^{4}\alpha^{k_{i}}(2r)^{4k_{i}}\LargerCdot\\\Pi_{i=5}^{r-1}\frac{([2r]^{4}\alpha)^{k_{i}}(2r)^{(i-4)k_{i}}}{\rho(\mathbf{P}^{(i-4)k_{i}})}trace(\mathbf{P}^{r})] - trace(\mathbf{P}^{r})],
 \end{multline} 

where for the last inequality we allow $k_{0}$ to equal $r$, but then subtract off the contribution when $k_{0} = r$.    At this juncture, we can apply Lemma \ref{lem:expineq} to complete the result, where to avoid an abuse of notation, we denote the parameters in Lemma \ref{lem:expineq} as $\bm{\alpha},\beta$ and $l$.  In particular, letting $l = 4$, $\bm{\alpha} = \alpha(2r)^{4}$, and $\beta = \frac{2r}{\rho(\mathbf{P})}$ yields the upperbound that  
 
  \begin{equation}\label{eq:goal051}
E(trace(C_{r}))\cdot [r^{r} + trace(\mathbf{P}^{r}) (exp(\frac{64\alpha r^{5}}{1-2r\rho(\mathbf{P})^{-1}})-1)].
 \end{equation}

\end{proof}
It follows immediately from Theorem \ref{thm:generallower} that if $\alpha(\log N)^{5} \rightarrow 0$ and $\rho(\mathbf{P})/r\rightarrow \infty$ that the standard deviation will be much smaller than the mean.  An application of Corollary \ref{cor:matrixbound} yields that if $\mathbf{P}^{2}\in\mathbb{R}^{k\times k}$ is entrywise bounded below by $1$, then $k\rho(\mathbf{P})^{r-4}\leq trace(\mathbf{P}^{r})$.  Consequently, we can show that with high probability that, $\frac{k\rho(\mathbf{P})^{r-4}}{8}\leq \frac{trace(\mathbf{P})^{r}}{8}\leq trace(C_{r}) \leq \rho(A^{r})$.  By taking the $rth$ root on both sides we get that in the limit $\rho(\mathbf{P})\leq \rho(A)$.   As such, Theorems \ref{thm:upperpar} and  \ref{thm:generallower}  provide asymptotic conditions that demonstrate that with high probability that $\rho(A)$ approaches $\rho(\mathbf{P})$.

We would also like to emphasize that while Theorem \ref{thm:generallower} demonstrates asymptotic convergence of the spectral radius, the speed of convergence appears rather slow as we avoided making assumptions regarding $\mathbf{P}$ in order to keep Theorem \ref{thm:generallower} as general as possible.  If however we knew that $trace(\mathbf{P}^{r-1}) \ll trace(\mathbf{P}^{r})$, then the appropriate extention of Theorem \ref{thm:generallower} would yield practical bounds for the distribution of the spectral radius for networks of finite size similar to Theorem \ref{thm:cov}.

At this juncture, we now consider the problem of extending our proof to the case where $p_{max} = \max_{i,j} \max_{m} \frac{a_{m}^{(i,j)}b_{m}^{(j,k)}}{S_{ij}}\rightarrow 0$.  Define $trace(SC_{r})$ to be the number of all of the simple cycles of length $r$ that contribute to $trace(C_{r})$, as defined in (\ref{eq:trCr2}).  We then have the following theorem,

\begin{thm}\label{thm:mtrxsc}
Let $trace(SC_{r})$ be the all of the simple cycles of length $r$ that contribute to $trace(C_{r})$.
Also let $p_{max} = \max_{i,j} \max_{m} \frac{a_{m}^{(i,j)}b_{m}^{(j,k)}}{S_{ij}}$ and assume that all eight non-zero terms of $\mathbf{P}$ are bounded below by $1$.
Then
	$$\frac{1}{2}(1-rp_{max})^{r}trace(\mathbf{P}^{r})\leq E(trace(SC(r)))$$
	
	Furthermore, we have that
	
	$$var(trace(SC_{r}))\leq E(trace(SC_{r}))\cdot [r+ trace(\mathbf{P}^{r}) [exp(\frac{4\alpha r^{2}}{1-r\rho(\mathbf{P})^{-1}})-1]].$$
\end{thm}
\begin{proof}
To sum over all possible simple cycles, we define the list $D$, where $(i_{1},...,i_{r},i_{r+1})\in D$ if the nodes $i_{1},...,i_{r}$ are distinct and $i_{1} = i_{r+1}$.
	\begin{equation}\label{eq:traceSCE} E(trace(SC_{r})) = trace(\sum_{(i_{1},...,i_{r},i_{r+1})\in D} \frac{1}{2}\mathbf{A}_{i_1}(\Pi_{k=2}^{r}\mathbf{B}_{i_{k}}\mathbf{A}_{i_{k}})\mathbf{B}_{i_1}). 
	\end{equation}
	Now since the trace is invariant under cyclic permutation, (\ref{eq:traceSCE}) equals,
	
	$$ trace(\sum_{(i_{1},...,i_{r},i_{r+1})\in D} \frac{1}{2}(\Pi_{k=2}^{r}\mathbf{B}_{i_{k}}\mathbf{A}_{i_{k}})\mathbf{B}_{i_1}\mathbf{A}_{i_1})=$$
	
	$$ trace(\sum_{(i_{1},...,i_{r},i_{r+1})\in D} \frac{1}{2}(\Pi_{k=1}^{r}\mathbf{B}_{i_{k}}\mathbf{A}_{i_{k}})\geq $$
	
	$$\frac{1}{2}(1-rp_{max})^{r}trace(\mathbf{P}^{r}),$$
	
	where the last inequality holds as for any particular fixed node, the largest entry in $\mathbf{B}_{i_k}\mathbf{A}_{i_k}$ is bounded above by $p_{max}$.  By fixing the first $r-1$ entries in the list, and summing over all possible choices for the $rth$ entry in the list, $i_{r}$, we get a matrix whose $ijth$ entry is $0$ if $\mathbf{P}_{ij} = 0$ or $(\mathbf{P}_{ij} - rp_{max})$ otherwise.  Furthermore, we know that $\mathbf{P}_{ij}-rp_{max}\geq \mathbf{P}_{ij}(1-rp_{max}) $,since $\mathbf{P}_{ij}>0 \implies \mathbf{P}_{ij}\geq 1$ by assumption.  Hence, by applying this argument iteratively to each node, we get the desired result.  
	
	 To prove the (co)variance result, we first define a collection of indicator random variables, $\mathbf{Y}$,  such that all cycles in $\mathbf{Y}$ are distinct under cyclic permutations, where
	 
	 \begin{equation} trace(SC_{r}) = r\sum_{y \in \mathbf{Y}}\mathbf{1}_{y}.
	 \end{equation}
	 
	 To compute the covariance term, we define a set $D(y)$, which includes all indices $z\neq y$, such that $\mathbf{1}_{y}$ and $\mathbf{1}_{z}$ are dependent random variables.  Consequently, 
	\begin{equation} \sum_{y \neq z} Cov(\mathbf{1}_y,\mathbf{1}_z) \leq \sum_{\substack{y\\z\in D(y)}}E(\mathbf{1}_{y}\mathbf{1}_{z}) = \sum_{\substack{y\\z\in D(y)}} Pr(\mathbf{1}_y = 1)Pr(\mathbf{1}_z=1|\mathbf{1}_y = 1).
	\end{equation}
	
	Similar to the argument from Theorem \ref{thm:generallower}, we get the upperbound that, \begin{equation} 
\sum_{\substack{y}} Pr(\mathbf{1}_y = 1)\sum_{\substack{\mathbf{L}_{R},\mathbf{L}_{N}\\ \mathbf{L}_{R} \neq \emptyset}}\sum_{\substack{n_{j}\in[1..N] \\ \forall j\in \mathbf{L}_{N}}}trace(\Pi_{t=1}^{r} \mathbf{A}(n_{t},n_{t+1},\mathbf{n}_{t})\mathbf{p}),
\end{equation}

where $n_{r+1} = n_{0}$,  $\mathbf{L}_{N}$ denotes the locations of the nodes in the interior of a new edge block and $\mathbf{L}_{R}$ assigns nodes found in the repeating edge blocks to equal specified nodes that are part of a new edge.  But note that since we are considering simple cycles that are distinct under cyclic permutation, the number of ways of filling in the repeating edge blocks is restricted.  From Lemma \ref{lem:lowercov}, we know that there are at most $r$ ways for deciding nodes in a given repeating edge block of any length.  Continuing the argument from Theorem \ref{thm:generallower}, we have that 

\begin{multline}\label{eq:parcovsimple}   
 \sum_{y \neq z} Cov(\mathbf{1}_y,\mathbf{1}_z) \leq \\ \sum_{y\in\mathbf{Y}}Pr(\mathbf{1}_{y}=1) \frac{1}{r}\sum_{\substack{k_{0} + \sum_{i=1}^{r} ik_{i} = r\\ k_{0} \in [1..r-1]}}\binom{\sum k_{i}}{k_{0},...,k_{r-1}}(r)^{\sum k_{i}}trace(\mathbf{P}^{r-\sum_{i=1}^{r-1}(i+1)k_{i}}(\alpha \mathbf{P})^{\sum_{i=1}^{r-2}k_{i}}),
\end{multline} 

where we must multiply our answer by a factor of $\frac{1}{r}$ as we overcounted, since we are considering a collection of cycles that are distinct under cyclic permutation.  

We can then simplify (\ref{eq:parcovsimple}) by bounding the multinomial coefficient by $\Pi_{i=1}^{r}\frac{r^{k_{i}}}{k_{i}!}$, invoking Corollary \ref{cor:matrixmultiply} to substitute the inequality, $tr(\mathbf{P}^{r-w-4})\leq \frac{tr(\mathbf{P}^{r})}{\rho(\mathbf{P})^{w}}$ and employing the assumption that $tr(\mathbf{P}^{r-w}) \leq tr(\mathbf{P}^{r})$ for all non-negative integers $w$.  Hence,

 \begin{equation}\label{eq:goalSC02}
  \sum_{y}Pr(\mathbf{1}_{y} = 1)\frac{1}{r}[\sum_{\substack{k_{0} + \sum_{i=1}^{r} ik_{i} = r\\ k_{0} \in [1..r-1]}}\Pi_{i=1}^{r-1}\frac{\alpha^{k_{i}}(r)^{2k_{i}}}{k_{i}!\cdot \rho(\mathbf{P}^{\max(ik_{i}-4,0)})}trace(\mathbf{P}^{r})]
 \end{equation}
	
Proceeding as in Theorem \ref{thm:generallower} will give us that, 	

 \begin{multline}\label{eq:goalSC04}
  \sum_{\substack{y \neq z\\ y,z \in \mathbf{Y}}} Cov(\mathbf{1}_y,\mathbf{1}_z) \leq \\
    \frac{1}{r}\sum_{y\in\mathbf{Y}}Pr(\mathbf{1}_{y} = 1)[trace(\mathbf{P}^{r}) [exp(\alpha[r^{2}+r^{2}+r^{2}+r^{2}+r^{2}\rho(\mathbf{P})^{-1}+r^{2}\rho(\mathbf{P}^{-1})+...-1]]\leq \\
     \frac{1}{r}\sum_{y\in\mathbf{Y}}Pr(\mathbf{1}_{y} = 1)[trace(\mathbf{P}^{r}) [exp(\frac{4\alpha r^{2}}{1-r\rho(\mathbf{P})^{-1}})-1]]\leq \\
       \frac{E(trace(SC_{r}))}{r^{2}}[trace(\mathbf{P}^{r}) [exp(\frac{4\alpha r^{2}}{1-r\rho(\mathbf{P})^{-1}})-1]].
 \end{multline} 
 
 Since  $trace(SC_{r}) = r\sum_{y \in \mathbf{Y}}\mathbf{1}_{y}$ and we define $var(trace(SC_{r}))= r^{2}\sum_{y \in \mathbf{Y}}Pr(\mathbf{1}_{y} = 1) + r^{2}\sum_{\substack{y \neq z\\ y,z \in \mathbf{Y}}} Cov(\mathbf{1}_y,\mathbf{1}_z)$, we are done.
 	
\end{proof}
We emphasize that Theorem \ref{thm:mtrxsc} tells us that the standard deviation of the $trace(SC_{r})$, a subset of simple cycles of length $r$, is much smaller than its expected value if $\lim_{N\rightarrow\infty}r^{2}\alpha=0$.  Consequently with high probability, the $trace(SC_{r})$ must be concentrated around its expected value.  Furthermore, if $\lim_{N\rightarrow\infty} r^{2}p_{max}=0$, it then follows from Corollary \ref{cor:matrixbound} that the lowerbound for expected value of the $trace(SC_{r})$ is $\frac{trace(\mathbf{P}^{r})}{2}$ and that $\frac{\rho(\mathbf{P}^{r-2})}{2}$ is an asymptotic lowerbound for the $trace(SC_{r})$; by Lemma \ref{lem:theory}, this proves the desired lowerbound on the spectral radius of the adjacency matrix.

To prove the upperbound, we now aim to generalize our prior results, arguing that when $\rho(\mathbf{P})$ is finite, then the likelihood of seeing multiple cycles together in a short path is really small.  Fortunately, many of the results in Section 3 hold for any (random) graph model.  The theorem below is a generalization of Lemma \ref{lem:pmax} applied to the partitioned Chung-Lu model.

\begin{thm}\label{thm:reducededge}
Recall the definition of $a_{x}^{(i,j)}$ and $b_{y}^{(k,l)}$ from the beginning of Theorem \ref{thm:upperpar}. Consider a sequence of (expected) partitioned degree sequences, where $p_{max}:= \max_{i,j,k,l,x,y} \frac{a_{x}^{(i,j)}b_{y}^{(k,l)}}{S_{ij}} \leq \frac{R}{N^{\tau}}$, $R$ is a fixed constant, $\tau > 0$ ,$\max_{i,j,k}\frac{\mathbf{a}^{(ij)}\cdot\mathbf{b}^{(j,k)}}{S(i,j)} > 1$, and the matrix $\mathbf{P}$ as defined earlier.  Then with probability at least $p_{*} = 1 - \delta$, all paths of length not exceeding   $L=\frac{k\tau}{2}*log_{|\mathbf{P}|}(N)$, have less than  $k+1$ cycle inducing edges, where $\delta = \frac{R^{k}L^{3k-2}}{N^{\frac{k\tau}{2}}}$. 
\end{thm}
\begin{proof}
The proof of the Theorem is analogous to the proofs of Lemmas \ref{lem:represent} and \ref{lem:pmax}.  

First define the function $G(k,i)$ to denote the group membership of the node $x_{k,i}$. 

From Lemma \ref{lem:represent}, it is not hard to show that in the Partitioned Chung-Lu random graph model, the probability that all edges in a reduced edge list exist will be,

\begin{equation} \label{eq:pmax01}
\Pi_{i=1}^{t}\Pi_{k=1}^{|M_{i}|-1} \frac{b_{x_{k,i}}^{(G(k,i),G(k+1,i))}a_{x_{k+1,i}}^{(G(k,i),G(k+1,i))}}{S_{G(k,i),G(k+1,i)}}.
\end{equation}

As the subscripts and superscripts become unweildy, we just denote the inputs and omit $G$.  This yields the following expression,

\begin{equation} \label{eq:pmax1}
\Pi_{i=1}^{t}\Pi_{k=1}^{|M_{i}|-1} \frac{b_{x_{k,i}}^{(k,k+1,i)}a_{x_{k+1,i}}^{(k,k+1,i)}}{S_{(k,k+1,i)}}, 
\end{equation}

As in the Proof of Lemma \ref{lem:pmax}, we can rewrite (\ref{eq:pmax1}) as follows,

\begin{equation} \label{eq:pmax2}
\frac{b_{x_{1,1}}^{(1,2,1)}a_{x_{|M_{t}|,t}}^{(|M_{t}|-1,|M_{t}|,t)}}{S_{(|M_{t}|-1,|M_{t}|,t)}}\Pi_{j=1}^{t-1}\frac{a_{x_{|M_{j}|,j}}^{(|M_{j}|-1,|M_{j}|,j)}b_{x_{1,j+1}}^{(1,2,j+1)}}{S_{(|M_{j}|-1,|M_{j}|,j)}}\Pi_{i=1}^{t}\Pi_{k=2}^{|M_{i}|-1} \frac{a_{x_{k,i}}^{(k-1,k,i)}b_{x_{k,i}}^{(k,k+1,i)}}{S_{(k-1,k,i)}}, 
\end{equation}

Now we can bound (\ref{eq:pmax2}) by defining 

$$\frac{b_{x_{k,i}}^{(*)}a_{x_{k,i}}^{(*)}}{S_*} = \max_{c,d}\frac{a_{x_{k,i}}^{(c,k,i)}b_{x_{k,i}}^{(k,d,i)}}{S(c,k,i)},$$

where we can choose any nodes $x_{c,i}$ and $x_{d,i}$ that maximize the aforementioned quantity.  Furthermore, we also have the property that if we sum over all possible choices of nodes $x_{k,i}$, 

\begin{equation} \label{eq:pmaxupperb}
\sum_{x_{k,i}=1}^{N}\frac{b_{x_{k,i}}^{(*)}a_{x_{k,i}}^{(*)}}{S_*} \leq |\mathbf{P}|.
\end{equation}

Hence we now have the upperbound,

\begin{equation} \label{eq:pmax3}
\frac{b_{x_{1,1}}^{(1,2,1)}a_{x_{|M_{t}|,t}}^{(|M_{t}|-1,|M_{t}|,t)}}{S_{(|M_{t}|-1,|M_{t}|,t)}}\Pi_{j=1}^{t-1}\frac{a_{x_{|M_{j}|,j}}^{(|M_{j}|-1,|M_{j}|,j)}b_{x_{1,j+1}}^{(1,2,j+1)}}{S_{(|M_{j}|-1,|M_{j}|,j)}}\Pi_{i=1}^{t}\Pi_{k=2}^{|M_{i}|-1} \frac{b_{x_{k,i}}^{(*)}a_{x_{k,i}}^{(*)}}{S_*}.
\end{equation}

Then using the argument from Lemma \ref{lem:pmax}, we can bound 
(\ref{eq:pmax3}) by

\begin{equation} \label{eq:pmax4}
p_{max}^{t-1}\frac{b_{x_{1,1}}^{(*)}a_{x_{1,1}}^{(*)}}{S_*}\Pi_{i=1}^{t}\Pi_{k=2}^{|M_{i}|-1} \frac{b_{x_{k,i}}^{(*)}a_{x_{k,i}}^{(*)}}{S_*}, 
\end{equation}

where we bound the terms in (\ref{eq:pmax3}) by $p_{max}$ only if the nodes corresponding to those terms also appear as 
$\frac{b_{x_{k,i}}^{(*)}a_{x_{k,i}}^{(*)}}{S_*}$.  

Now that (\ref{eq:pmax4}) is easy to add up (the is no longer a  dependence relationship between the terms) with inequality (\ref{eq:pmaxupperb}), the proof proceeds analogously as in Lemma \ref{lem:pmax}.

Though for simplicity, we omitted some of the details of the proof in Theorem \ref{thm:reducededge}, from (\ref{eq:pmaxupperb}) and (\ref{eq:pmax4}), it should be clear that the probability that $t$ cycle inducing edges appear in a reduced edge list with $L$ edges should be roughly $O(|\mathbf{P}|^{L}p_{max}^{t-1})$.  Since $p_{max} = O(N^{-\tau})$, we would need to consider paths of length $L = O(\tau log_{|\mathbf{P}|}(N))$ for this event to occur.
\end{proof}

As Lemmas \ref{lem:cyclelength} and \ref{lem:tpath} hold under for any random graph model, we have the desired upperbound. 

\begin{thm}\label{thm:upperpmaxpartition}
Consider a sequence of (expected) partitioned degree sequences, where $p_{max}:= \max_{i,j,k,l,x,y} \frac{a_{x}^{(i,j)}b_{y}^{(k,l)}}{S(i,j)} \leq \frac{R}{N^{\tau}}$, $R$ is a fixed constant, $\tau > 0$  and $\rho(\mathbf{P}) > 1$.   Let $r$ satisfy the inequality  $\frac{(m-1)\tau}{2} log_{|\mathbf{P}|}(N)\leq r \leq  \frac{m\tau}{2} log_{|\mathbf{P}|}(N)$, for some parameter $m$.

Define $P_{r}(\mathbf{G})$ to be the number of paths of length $r$ for a given graph $G\in\mathbf{G}$, where no path of length $r$ in $\mathbf{G}$ has more than $m+1$ cycle inducing edges and no path in $\mathbf{G}$ of length less than $ \frac{\tau}{2}*log_{|\mathbf{P}|}(N)$ has more than $1$ cycle inducing edge.  (Recall Theorem \ref{thm:reducededge}, where it follows that $\lim_{N\rightarrow \infty} Pr(G\in \mathbf{G})=1$.)   For notational simplicity define $$\eta =  [\frac{\tau}{4}\hspace{3pt}log_{|\mathbf{P}|}(N)]^{3}(m+1)exp(2m)$$ and suppose that  $$(\eta)^{\frac{4}{\tau*log_{|\mathbf{P}|}(N)}}<\rho(\mathbf{P}),$$ then,

$$E(P_{r}(\mathbf{G}))\leq \frac{4Sc_{max}\rho(\mathbf{P})^{r-1}}{(1-\rho(\mathbf{P})^{-1})Pr(G\in\mathbf{G})}exp(\frac{r\alpha\rho(\mathbf{P})^{-1}\eta^{1+\frac{4}{\tau*log_{|\mathbf{P}|}(N)})}}{1-\rho(\mathbf{P})^{-1}(\eta)^{\frac{4}{\tau*log_{|\mathbf{P}|}(N)})}}),
$$

where we define $\alpha$ such that inequality (\ref{eq:pmaxmatrix}) holds ,$c_{max}$ is the maximum column sum of the matrix $\mathbf{P}^{2}$ and we assume that the (eight) non-zero entries of $\mathbf{P}$ are at least $1$.
\end{thm}
\begin{proof}
The proof follows the arguments from Theorems \ref{thm:pmaxtrick} and \ref{thm:upperpar}.

Let $E(P_{r}^{L}(\mathbf{G}))$ denote the expected number of paths of length $r$ where the first and last edge cannot be repeating edges for all graphs $G \in \mathbf{G}$.  

Denote  $\mathbf{1}_{\mathbf{E}}$ as an indicator variable that all of the edges in $\mathbf{E}$ exist, then

\begin{equation} \label{eq:probpmaxpath2}
Pr(\mathbf{1}_{\mathbf{E}} = 1|G\in\mathbf{G}) \leq \frac{Pr(\mathbf{1}_{\mathbf{E}} = 1)}{Pr(G\in \mathbf{G})}
\end{equation}

Repeating the argument from Theorem \ref{thm:upperpar} and invoking (\ref{eq:probpmaxpath2}) to compute the probability that edges exist yields the bound,

\begin{equation}\label{eq:pmaxpath}   
E(P_{r}^{L}(\mathbf{G}))\leq  \frac{1}{Pr(G\in\mathbf{G})}\sum_{\mathbf{L}_{R},\mathbf{L}_{N}}\sum_{n_0 = 1,n_r = 1}^{N}|\mathbf{A}_{n_{r}}\mathbf{P}^{|\mathbf{L}_{N}|}(\alpha \mathbf{P})^{\sum_{i=1}^{r}k_{i}}\mathbf{b}|,
\end{equation} 

where $\mathbf{L}_{N}$ identifies the nodes that are in the interior of a new edge block and $\mathbf{L}_{R}$ specifies the how nodes in the repeating edge blocks correspond to the nodes that are part of new edges in the path.  By carefully counting the number of ways we can construct repeating edge blocks (as in Theorem \ref{thm:pmaxtrick}) we have that, 

\begin{multline}
E(P_{r}^{L}(\mathbf{G}))\leq  \\ \sum_{\substack{k_{0}+\sum_{i=1}^{r-2}ik_{i}\\ \forall i, k_{i}\in [0..r] }}\sum_{n_0 = 1,n_r = 1}^{N}\binom{\sum k_{i}}{k_{0},...,k_{r-2}}(\eta)^{\sum_{i\geq 1} k_{i}(1+\frac{4i}{\tau*log_{|\mathbf{P}|}(N)})}|\mathbf{A}_{n_{r}}\mathbf{P}^{r - 1 - \sum_{i=1}^{r}(i+1)k_{i}}(\alpha \mathbf{P})^{\sum_{i=1}^{r}k_{i}}\mathbf{b}|.
\end{multline}

Summing over all initial nodes $n_{0}$ and final nodes $n_{r}$ and invoking the fact that $\binom{\sum k_{i}}{k_{0},...,k_{r-2}}\leq \Pi_{i=1}^{r-2}\frac{r^{k_{i}}}{k_{i}!}$, tells us that

\begin{multline}
E(P_{r}^{L}(\mathbf{G}))\leq  \\ \sum_{\substack{k_{0}+\sum_{i=1}^{r-2}ik_{i}\\ \forall i, k_{i}\in [0..r] }}|\mathbf{P}^{r - 1 - \sum_{i=1}^{r}(i+1)k_{i}}(\alpha \mathbf{P})^{\sum_{i=1}^{r}k_{i}}S\mathbf{1}|\Pi_{i=1}^{r-2}\frac{r^{k_{i}}}{k_{i}!}(\eta)^{k_{i}(1+\frac{4i}{\tau*log_{|\mathbf{P}|}(N)})}\leq 
\end{multline}

\begin{multline} \label{eq:pmaxparpath3}
 \sum_{\substack{k_{0}+\sum_{i=1}^{r-2}ik_{i}\\ \forall i, k_{i}\in [0..r] }}4Sc_{max}\rho(\mathbf{P})^{r - 1 - \sum_{i=1}^{r}ik_{i}}\Pi_{i=1}^{r-2}\frac{(\alpha r)^{k_{i}}}{k_{i}!}(\eta)^{k_{i}(1+\frac{4i}{\tau*log_{|\mathbf{P}|}(N)})},
\end{multline}

where we bounded the taxicab norm of a matrix using the spectral radius.  

We can then finish off the proof by noting that (\ref{eq:pmaxparpath3}) is bounded above by,

\begin{multline}
 4Sc_{max}\rho(\mathbf{P})^{r-1}\sum_{k_{1}=0,...,k_{r-2}=0}^{\infty}\Pi_{i=1}^{r-2}\frac{(\alpha r)^{k_{i}}\rho(\mathbf{P})^{-ik_{i}}(\eta)^{k_{i}(1+\frac{4i}{\tau*log_{|\mathbf{P}|}(N)})}}{k_{i}!}\leq 
 \\ 
  4Sc_{max}\rho(\mathbf{P})^{r-1}\Pi_{i=1}^{r-2}exp(\alpha r\rho(\mathbf{P})^{-i}(\eta)^{(1+\frac{4i}{\tau*log_{|\mathbf{P}|}(N)})})\leq \\
    4Sc_{max}\rho(\mathbf{P})^{r-1}exp(\frac{\alpha r\rho(\mathbf{P})^{-1}\eta^{1+\frac{4}{\tau*log_{|\mathbf{P}|}(N)}}}{1-\rho(\mathbf{P})^{-1}\eta^{\frac{4}{\tau*log_{|\mathbf{P}|}(N)}}})
  \end{multline}
  
  Invoking the fact from Theorem \ref{thm:upper} that $E(P_{r}(\mathbf{G})) \leq \sum_{m=1}^{r} E(P_{r}^{L}(\mathbf{G}))$, finishes the proof.

\end{proof}
We conclude this section by elucidating how Theorem \ref{thm:upperpmaxpartition} yields the desired asymptotic result. By requiring that $r$ satisfy the constraint $O(log(N)) \ll r \ll O(log(N)^{2})$ and $\alpha$ scale like $O(N^{-\tau})$, it then follows that $\lim_{N\rightarrow\infty}\eta^{\frac{4}{\tau\log_{|\mathbf{P}|}(N)}}=1$ and $\lim_{N\rightarrow \infty} r\alpha\eta^{1+\frac{4}{\tau*log_{|\mathbf{P}|}(N)})}=0$.   Similarly, $\lim_{N\rightarrow \infty} Pr(G\in\mathbf{G}) = 1$ and we get that asymptotically, $E(P_{r}(\mathbf{G})) \leq \frac{4Sc_{max}\rho(\mathbf{P})^{r-1}}{1-\rho(\mathbf{P})^{-1}}$.  Consequently, an application of Markov's Inequality demonstrates that $\rho(\mathbf{P})$ is also an asymptotic upperbound for the spectral radius.  Similarly, Theorem \ref{thm:mtrxsc} and Corollary \ref{cor:matrixbound}  show that $\rho(\mathbf{P})$ is also an aymptotic lowerbound for the spectral radius, hence the spectral radius must converge to $\rho(\mathbf{P})$. 

\section{Applications}
Determining the dominating eigenvalue of the adjacency matrix can have a profound effect on the underlying dynamics of the network.  For example consider a susceptible-infected-susceptible (SIS) epidemiological model, where at each step an infected node infects a neighbor with probability $\beta \Delta t$ and recovers (from sick to healthy) with probability $\Delta t$, where $\Delta t$ denotes the length of the time step.  This leads us to the following result,

\begin{thm}[Ganesh, Massoulie, Towsley \cite{ganesh2005effect}] \label{thm:Ganesh}
Consider an SIS epidemiological model, where infected nodes infect neighbors with probability $\beta \Delta t$ at each time step and recover with probability $\Delta t$. Furthermore, suppose our adjacency matrix $A\in\mathbb{R}^{N\times N}$ is symmetric. Then for $\Delta t$ sufficiently small, if $\rho(A) < \frac{1}{\beta}$, then the expected (stopping) time for the network to be infection free from any initial condition is $O(log(N))$.
\end{thm}

While our adjacency matrices are not symmetric, Theorem \ref{thm:Ganesh} relates the dominating eigenvalue of the adjacency matrix to the stability of the healthy state and provides a framework for constructing cases where differences in the spectral radius of the adjacency matrix between the Chung-Lu and Partitioned Chung-Lu model could have severe repercussions on the dynamics.  

\begin{figure}
\includegraphics[scale=.6]{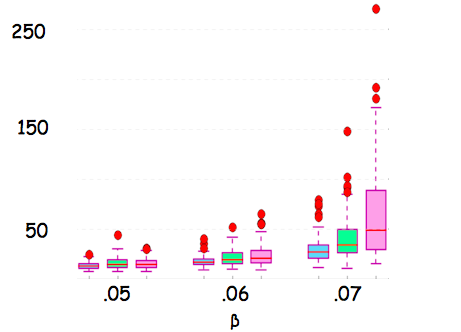}
\centering
\caption[Distribution of Stopping Times in the SIS Model for Networks with Community Structure]{Distribution of Stopping Times in the SIS Model for Networks with Community Structure. The three colored box and whisker plots correspond to the three different networks ordered increasingly in terms of the spectral radius of the adjacency matrix. The $x$ axis indicates the value for the paramater $\beta$ and the $y$ axis indicates the stopping time when nodes can no longer be infectious.}
\label{fig:SIS}
\end{figure}

Consequently in Figure \ref{fig:SIS}, we generated three realizations from the 2-Partitioned Chung-Lu random graph model of $500$ nodes, all with approximately the same value for $\frac{\mathbf{a}\cdot\mathbf{b}}{S}$, but different values for $\rho(\mathbf{P})$. We then simulated $100$ trials of the SIS epidemiolgocial stochastic process for each choice of $\beta \in \{.05,.06,.07\}$ with the initial condition that half of our network starts out infected.  As expected $\rho(\mathbf{P})$ accurately predicted the network resilience to the pathogen; in contrast, the predictor $\frac{\mathbf{a}\cdot\mathbf{b}}{S}$,  could not effectively discern differences among the networks.

\appendix
\section{Appendix}
We now provide a proof for an inequality that we employed frequently throughout this work.
\begin{lem}\label{lem:expineq}
For parameters, $m, l, r \in \mathbb{N}$, where $l < m$  and $\alpha, \beta \in \mathbb{R}$, where $\beta < 1$, we have that
\begin{equation} \label{eq:expcrit}
\sum_{\substack{k_{0}+\sum_{i=1}^{m}ik_{i} = r \\ \forall i \in [0..m], k_{i}\in [0..r]}}\binom{\sum_{i=0}^{m}k_{i}}{k_{0},...,k_{m}}\Pi_{i=1}^{l}\alpha^{k_{i}}\Pi_{i=l+1}^{m}\alpha^{k_{i}}\beta^{(i-l)k_{i}} \leq exp(\frac{lr\alpha}{1-\beta}).
\end{equation}
\end{lem}
\begin{proof}
First, we have the bound on the multinomial coefficient that $\binom{\sum_{i=0}^{m}k_{i}}{k_{0},...,k_{m}}\leq \Pi_{i=1}^{m}\frac{r^{k_{i}}}{k_{i}!}.$  Substituting this into the left hand side of (\ref{eq:expcrit}) we have the upperbound that,
\begin{equation}\label{eq:expcrit2}
\sum_{\substack{k_{0}+\sum_{i=1}^{m}ik_{i} = r \\ \forall i \in [0..m], k_{i}\in [0..r]}}\Pi_{i=1}^{l}\frac{(r\alpha)^{k_{i}}}{k_{i}!}\Pi_{i=l+1}^{m}\frac{(r\alpha)^{k_{i}}\beta^{(i-l)k_{i}}}{k_{i}!}.
\end{equation}
We can further bound (\ref{eq:expcrit2}) from above by removing the constraint that $k_{0}+\sum ik_{i} = r$, where we let $k_{1},..,k_{m}$ take on any non-negative integer value (and require that $k_{0} = r - \sum_{i=1}^{m} ik_{i}$).  This yields the bound,

\begin{equation}\label{eq:expcrit3}
\sum_{k_{1}=0,...,k_{m}=0}^{\infty}\Pi_{i=1}^{l}\frac{(r\alpha)^{k_{i}}}{k_{i}!}\Pi_{i=l+1}^{m}\frac{(r\alpha)^{k_{i}}\beta^{(i-l)k_{i}}}{k_{i}!} = \Pi_{i=1}^{l}exp(r\alpha)\Pi_{i=l+1}^{m}exp(r\alpha\beta^{i-l}).
\end{equation}

By adding the exponents, we get that,
\begin{equation}\label{eq:expcrit4}
\Pi_{i=1}^{l}exp(r\alpha)\Pi_{i=l+1}^{m}exp(r\alpha\beta^{i-l})\leq exp(rl\alpha +\sum_{i=1}^{\infty} \alpha\beta^{i})\leq exp(\frac{rl\alpha}{1-\beta}).	
\end{equation}

\end{proof}
At this juncture, we provide proofs to the linear algebra results necessary to derive our desired inequalities on the moments for the number of paths of length $r$.

\begin{cor} \label{cor:matrixmultiply} Let $B$ be an (entry-wise)  non-negative matrix, $B^{2}$ be an entry-wise positive matrix and let $x$ be the eigenvector corresponding to the dominating eigenvalue.  Furthermore let $b_{ij}^{(m)}$ denote the i,jth entry of $B^{m}$ .Let $c_{max}$ be the maximum row sum of $B^{2}\in\mathbb{R}^{n\times n}$ and suppose every entry is at least equal to 1, (hence $r_{max}\geq n$). Then

$$[\sum_{j=1}^{n}b_{ij}^{(m)}]^{\frac{1}{m}}\leq {c_{max}}^{\frac{1}{m}}\rho(B)$$
and $$(c_{max})^{-\frac{1}{m}}\rho(B)\leq [\sum_{j=1}^{n}b_{ij}^{(m)}]^{\frac{1}{m}}.$$
\end{cor}

Corollary \ref{cor:matrixmultiply} follows immediately from the following two lemmas, the first of which can be found in \cite{horn2012matrix}, page 494.

\begin{lem}Let $B\in\mathbb{R}^{N\times N}$ be a (entry-wise) nonnegative matrix and let $x$ be the eigenvector corresponding to the dominating eigenvalue. Assume that the eigenvector $x$ is strictly positive.
Furthermore let $b_{ij}^{(m)}$ denote the i,jth entry of $B^{m}$. Then for all integers $m$ and integers $j$ such that $1\leq j \leq N$, we have that $$\sum_{i=1}^{n}b_{ij}^{(m)}\leq \frac{\max_{k} x_{k}}{\min_{k} x_{k}}\rho(B)^{m}$$
and $$\frac{\min_{k} x_{k}}{\max_{k} x_{k}}\rho(B)^{m} \leq \sum_{j=1}^{n}b_{ij}^{(m)}.$$
\end{lem}

To make the most use of the aforementioned lemma, we need to bound the entries $\max_{k} x_{k}$ and $\min_{k} x_{k}$ in the dominating eigenvector.  

\begin{lem}Let $r_{max},c_{max}$ be the maximum row sum and column sum of $B\in\mathbb{R}^{n\times n}$. In addition, suppose every entry is at least equal to $m>0$.  Denote $\rho(B) = \lambda_{max}$.  Since by the Gresgorin Disc Theorem $\lambda_{max}\leq \min(r_{max},c_{max})$.  Then $$\frac{m}{\min(c_{max},r_{max})-m(n-1)} \leq \frac{m}{\lambda_{max}-m(n-1)}\leq \frac{\min_{k}x_{k}}{\max_{k}x_{k}}$$
and $$\frac{\max_{k}x_{k}}{\min_{k}x_{k}}\leq \frac{\lambda_{max} - m(n-1)}{m} \leq \frac{\min(c_{max},r_{max})- m(n-1)}{m}.$$
\end{lem}
\begin{proof}
Consider the eigenvector $x$ and require that $\sum_{j=1}^{n}x_{j}=1$ where we are guaranteed that each entry in the eigenvector is non-negative by the Perron-Frobenius Theorem.  Then we have for all k,$$m=\sum_{j=1}^{n} mx_{j}\leq\sum_{j=1}^{n} b_{jk}x_{j}=\lambda_{max}x_{k}.$$
It then follows that for all k $$ \frac{m}{\lambda_{max}}\leq x_{k}.$$
Consequently, $$\frac{m}{\lambda_{max}}\leq \min_{k} x_{k}.$$
Furthermore, since $\sum_{j} x_{j}=1$, we have that $$\max_{k} x_{k}\leq 1-\frac{m*(n-1)}{\lambda_{max}}.$$
This implies that $$\frac{m}{\lambda_{max}-m(n-1)}\leq \frac{\min_{k}x_{k}}{\max_{k}x_{k}}$$ and
$$\frac{\max_{k}x_{k}}{\min_{k}x_{k}}\leq \frac{\lambda_{max} - m(n-1)}{m}.$$ 
\end{proof}

\begin{cor}\label{cor:matrixbound}
	Suppose each entry of a matrix $B^{t}\in\mathbb{R}^{k\times k}$ is bounded below by $1$.  Then for all $u \in \mathbb{N}\cup 0$ and all $v\in\mathbb{N}$,
	
	$$\rho(B)^{v}B^{u}\leq B^{u+v+2t}.$$
	
	Furthermore, $$k\rho(B)^{v}\leq tr(B^{v+2t}).$$
\end{cor}

\begin{proof}
Without loss of generality, assume that the index $m$ satisfies the inequality that, $\sum_{n=1}^{k} b_{mn}^{(s)}\geq \rho(B)^{s}$. Succinctly, 
	
$$b_{ij}^{(r+s+2t)}\geq \sum_{n=1}^{k} b_{ij}^{(r)}b_{jm}^{(t)}b_{mn}^{(s)}b_{nj}^{(t)}\geq b_{ij}^{(r)}\sum_{n=1}^{k}b_{mn}^{(s)}\geq \rho(B)^{s}b_{ij}^{(r)}.$$	
\end{proof}

\begin{lem}
Suppose that $B^{w}\in\mathbb{R}^{k\times k}$ is entrywise bounded below by $1$, where $w$ is a positive integer and that $B$ is an entrywise non-negative matrix.  Then for every integer $r$ greater than $w$, there exists a non-negative integer $m$ that satisfies the inequality, $0 \leq m \leq w-1$, such that

\begin{equation} tr(B^{q}) \leq tr(B^{r-m}) 
\end{equation}

for all non-negative integers $q$ that satisfy the inequality $q \leq r - m$.  

\end{lem}

\begin{proof}
The key observation is that since $B^{w}$ is entrywise bounded below by $1$, that for any choice of $r$, $B^{r} \leq B^{r+w}$.  Consequently, it follows that for $z \geq w$,  $$\max_{t\in [1..z]} tr(B^{t}) = \max_{t\in[z-w-1..z]}tr(B^{t}).$$  Once we identify the $t$ bounded from $z-w-1$ to $z$ that optimizes the $tr(B^{t})$, the result follows immediately.  
\end{proof}

\bibliographystyle{plain} 
\bibliography{thesis}{}

\end{document}